\documentclass[reqno]{amsart}

\usepackage{a4wide}

\usepackage[pagebackref,citecolor=black,linkcolor=blue,colorlinks=true]{hyperref}

\usepackage[shortlabels]{enumitem}
\setlist[enumerate,1]{label={\arabic*. }}

\usepackage{amsmath,amssymb}
\usepackage{mathrsfs} 

\usepackage{tikz}

\newcommand{\dd}[1]{\mathrm{d}{#1}} 
\newcommand{\derivative}[2]{\frac{\dd{{#1}}}{\dd{{#2}}}} 
\newcommand{\D}{\mathrm{d}}  
\newcommand{\loc}{\mathrm{loc}}

\newcommand{\dist}{\mathrm{dist}}
\newcommand{\graph}{\mathrm{graph}}
\newcommand{\exc}{\mathrm{exc}}

\newcommand\HH{\mathcal{H}}
\newcommand{\N}{\mathbb{N}}
\newcommand{\R}{\mathbb{R}}
\newcommand{\eps}{\varepsilon}
\newcommand{\mres}{\mathbin{\vrule height 1.6ex depth 0pt width
0.13ex\vrule height 0.13ex depth 0pt width 1.3ex}}

\newcommand\myatop[2]{\genfrac{}{}{0pt}{}{#1}{#2}}

\DeclareMathOperator*{\argmin}{arg\,min}

\DeclareMathOperator\supp{supp}
\DeclareMathOperator*{\diverg}{div} 
 
\DeclareMathOperator{\lip}{Lip}

\DeclareMathOperator*{\aplim}{aplim}

\providecommand{\abs}[1]{\lvert#1\rvert}
\providecommand{\norm}[1]{\lVert#1\rVert}
\newcommand{\minimizer}[4]{$(#1,#2) \in \argmin \{E(J,v): (J,v)\in \mathcal{A}(#3,#4)\}$}

\newcommand{\fsln}{for sufficiently large $n\in \N$}

\newcommand{\myset}{}

\newcounter{hints}
\renewcommand{\thehints}{\arabic{hints}}
\newcommand{\hintedrel}[2][]{
  \refstepcounter{hints}
  \if\relax\detokenize{#1}\relax\else\label{#1}\fi
  \mathrel{\overset{\mathrm{(\thehints)}}{\vphantom{\le}{#2}}}
}

\newtheorem{thm}{Theorem}[section]
\newtheorem{prop}[thm]{Proposition}
\newtheorem{lmm}[thm]{Lemma}

\theoremstyle{definition}
\newtheorem{dfn}[thm]{Definition}
\newtheorem{cor}[thm]{Corollary}
\newtheorem{claim}{Claim}[thm]

\theoremstyle{remark}
\newtheorem{rmk}[thm]{Remark}

\numberwithin{equation}{section}

\begin{document}

\title[Regularity for Mumford-Shah with Dirichlet boundary conditions]{Boundary Regularity for the Mumford-Shah functional with Dirichlet boundary conditions}

\author{Francesco Deangelis}
\address{Gran Sasso Science Institute - Viale F. Crispi 7, 67100 L'Aquila, Italy}

\email{francesco.deangelis@gssi.it}

\begin{abstract}
    In this paper we consider minimizers of the Mumford-Shah functional with Dirichlet boundary conditions. We study blow-ups at the boundary and prove an epsilon-regularity theorem.
\end{abstract}

\maketitle

\section{Introduction}
\label{section: introduction}
Given a bounded domain $\Omega \subset \R^2$ with boundary of class $C^{1,\gamma}$ and a boundary datum $g \in C^{1,\gamma}(\partial\Omega)$ for some $\gamma \in (0,1)$, we are interested in the regularity of minimizers of the Mumford-Shah functional 
\begin{equation}
    \label{eq: mumford-shah functional}
    E(K,u) = \int_{\Omega \setminus K} |\nabla u|^2 + \HH^1(K)
\end{equation}
among all pairs $(K,u)$ such that $K$ is a closed subset of $\overline{\Omega}$, $u \in W^{1,2}(\Omega \setminus K)$ and $u = g$ on $\partial \Omega \setminus K$. We will always assume that minimizers, whose existence was proved in \cite{carriero-leaci}, are normalized, that is $\HH^1(K \cap B_\rho(x)) >0$ for every $x \in K$ and $\rho>0$. 

We prove that if $(K,u)$ is a minimizer, $x \in K \cap \partial \Omega$ and $K$ is sufficiently close to the tangent line $x+T_x \partial \Omega$ in a small ball centered at $x$, then in a smaller concentric ball $K$ is the graph of a $C^{1,\alpha}$ function. The precise statement is the following, where $\dist_H$ denotes the Hausdorff distance. 
\begin{thm}[Epsilon-regularity] 
    \label{thm: epsilon-regularity}
    There exist constants $\eta,\varepsilon, \alpha \in (0,1)$ with the following property. Assume 
    \begin{enumerate}[label=(\roman*)]
        \item $(K,u)$ is a minimizer of $E$ on $\Omega$ with boundary datum $g$ and $x \in K \cap \partial\Omega$,
        \item $r^{-1} \dist_H(K \cap \overline{B}_{2r}(x), (x+T_x \partial \Omega) \cap \overline{B}_{2r}(x)) + r^\frac{1}{2} < \varepsilon$.
    \end{enumerate}
    Then $K \cap B_{\eta r}(x)$ is the graph of a $C^{1,\alpha}$ function over $x+T_x \partial \Omega$.
\end{thm}

The motivation for this theorem comes from a partial classification of blow-ups of minimizers (see Section \ref{section: blow-up} for a precise definition) at boundary points. 

We conjecture that, if $x \in K \cap \partial \Omega$, then the blow-up of $(K,u)$ at $x$, which we call $(K_0, u_0)$, is such that  $K_0 = T_x \partial \Omega$. 

We are able to prove that this is the case under one of the following additional assumptions: 
\begin{enumerate}
    \item $\int |\nabla u_0|^2 = 0$,
    \item $K_0$ is connected,
    \item for a.e. $r>0$ either $\HH^0(K_0 \cap \partial B_r) \in \{0\} \cup [2,\infty)$ or $K_0 \cap B_r = \{r e^{i\theta}\}$ with $\theta \in (\frac{1}{4}\pi,\frac{3}{4}\pi)$.
\end{enumerate}
Moreover, we prove that it is not possible to have $\HH^0(K_0 \cap \partial B_r) = 1$ for a.e. $r>0$. 
This, or the case $K_0$ connected, excludes the possibility that $K_0$ is a half-line originating at $0$, which corresponds to the so-called cracktip in the literature for the Mumford-Shah problem.
\begin{rmk}
    \label{rmk: after main theorem}
    The proof of the above conjecture would yield, as a corollary of Theorem \ref{thm: epsilon-regularity}, a complete boundary regularity theory for the Dirichlet problem. More precisely, it would imply that, given $x \in \partial \Omega$, either $x \notin K$, and $u$ is harmonic, with $u=g$, in a neighborhood of $x$, or $x \in K$, and $K$ is a $C^{1,\alpha}$ graph in a neighborhood of $x$, with $T_x K = T_x \partial \Omega$.

    At the moment, the regularity at the boundary for the Mumford-Shah functional, with a fidelity term $\int|u-f|^2$ given $f \in L^\infty(\Omega)$, is only known in the case of zero Neumann boundary conditions. This was obtained in \cite{david} and \cite{maddalena-solimini}: if $\partial \Omega$ is $C^1$ then there is a neighborhood of $\partial\Omega$ where K is a finite union of $C^1$ curves that end on $\partial\Omega$ perpendicularly.
 
    We could include the fidelity term in our functional. However, since this adds only technical complications, we prefer to omit it for the sake of simplicity. 
    Observe that, without the fidelity term, both the homogeneous Dirichlet problem and the homogeneous Neumann problem are trivial. 
\end{rmk}
The Mumford-Shah functional was first introduced in \cite{Mumford-Shah} with the aim to study image segmentation with a variational approach. 
In the same paper, the two authors formulated a conjecture about the regularity of minimizers (see \cite{delellis-focardi} for the precise definition), which is still unsolved.
The original functional includes a fidelity term $\int|u-f|^2$ for a given $f \in L^\infty(\Omega)$, but this is not relevant for regularity issues (see \cite{delellis-focardi}). Moreover, the original problem does not have any Dirichlet boundary conditions, giving rise to homogeneous Neumann boundary conditions (see Remark \ref{rmk: after main theorem} for the regularity at the boundary in this case).

Our problem, instead, has a physical interpretation in the context of fracture mechanics \cite{Francfort-Marigo-1998,Bourdin-Francfort-Marigo-2008} in the particular case of antiplane shear strain, in which $u$ represents the vertical displacement of a membrane $\Omega$, with an imposed displacement $g$ on the boundary, and $K$ represents a fracture. 

The main progress for the regularity at the interior is represented by the partial classification of blow-ups and by the epsilon-regularity theorems, that make the solution of the Mumford-Shah conjecture equivalent to a complete classification of blow-ups. A complete account of the current state of the art in the regularity theory can be found in \cite{delellis-focardi}. The main references are \cite{degiorgi-carriero-leaci,dal-maso-morel-solimini,david-c1-arcs,ambrosio-pallara-free-discont-1,ambrosio-fusco-pallara-free-discont-2,afp-higher-regularity,leger,bonnet-david-cracktip,lemenant-r3,delellis-focardi-densitylowerbound,delellis-focardi-higher,bucur-luckhaus,dephilippis-figalli,lemenant-rigidity,lemenant-mikayelyan,andersson-mikayelyan,delellis-focardi-ghinassi}.

The regularity at the boundary for the Dirichlet problem is still an open question, and its investigation is the purpose of the present paper. 
This task seems easier than the regularity at the interior, but more difficult than the regularity at the boundary for the homogeneous Neumann problem. 
Even though we are not able to get a monotonicity formula that works in general, giving a complete classification of blow-ups, we obtain two monotonicity formulae that complement each other and give us several partial classification results. 
Moreover, in the cases that we are able to classify blow-ups, the situation is simpler compared to the interior problem. For example, under the assumption that $K_0$ is connected, we reduce the list of possible blow-ups, which in the interior case \cite{Bonnet} is constituted by pure jump, triple junction and cracktip, to only the pure jump. 
Finally, compared to the homogeneous Neumann boundary problem, we expect that $K$ cannot end on $\partial\Omega$, but needs to touch it tangentially, as it is the case under the assumptions of Theorem \ref{thm: epsilon-regularity}. At an intuitive level, this represents an attempt of $K$ to exclude a portion of the boundary datum which is too costly to attain. 

The main ideas to prove our results are borrowed from the regularity theory at the interior \cite{delellis-focardi}, with non-trivial difficulties due to the presence of the boundary datum. 
The proof of Theorem \ref{thm: epsilon-regularity} deviates from what is typically encountered in proofs of epsilon-regularity theorems. More specifically, for points $x \in K \setminus \partial\Omega$ very close to the boundary we are unable to obtain directly a decay of the flatness (see Section \ref{section: overview of the proof of eps-regularity} for a precise definition) when we pass from $B_r(x)$ to $B_{\tau r}(x)$, with $\tau \in (0,1)$. This is what happens, for example, in the proof of \cite[Theorem 3.1.1]{delellis-focardi}. In our case, instead, we have to distinguish three different regimes. As we decrease the radius of the ball $B_\rho(x)$ we pass from the first regime, in which we have a decay, to the second regime, in which we do not have a decay but the flatness does not grow too much, to the third regime, in which we finally have the desired decay. 
To make an analogy, we can think of this process as the evolution of a one-parameter dynamical system, where the parameter is represented by the radius $\rho$ of the ball $B_\rho(x)$. Then the second regime corresponds to a center manifold of the dynamical system. 
The definition of the three different regimes and a sketch of the proof of Theorem \ref{thm: epsilon-regularity} can be found in the next section.

\subsection{Overview of the proof of Theorem \ref{thm: epsilon-regularity}}
\label{section: overview of the proof of eps-regularity}
In order to sketch the main ideas of the proof of Theorem \ref{thm: epsilon-regularity}, let us introduce some notation. We define
\begin{equation*}
    \mathcal{L} = \{\mathscr{V} \, | \, \mathscr{V} \subset \R^2 \text{ 1-dimensional subspace}\}, \quad
    \mathcal{A} = \{z + \mathscr{V}\, | \, z \in \R^2, \mathscr{V} \in \mathcal{L}\}. 
\end{equation*}
Given $\mathscr{V} ,\mathscr{V'} \in \mathcal{L}$, we define $|\mathscr{V} - \mathscr{V'}| :=  |\pi_\mathscr{V} - \pi_\mathscr{V'}|$,
where $| \cdot |$ on the right hand side denotes the Hilbert-Schmidt norm, while $\pi_\mathscr{V}: \R^2 \to \mathscr{V}$ is the orthogonal projection over $\mathscr{V}$.
For any $x \in \overline{\Omega}$, $r > 0$, $\mathscr{A} \in \mathcal{A}$ and $\mathscr{V} \in \mathcal{L}$, we define
\begin{align*}
    d(x,r) &= \frac{1}{r} \int_{B_r(x) \cap \Omega \setminus K} |\nabla u|^2, \\
    \beta_\mathscr{A}(x,r) &= \frac{1}{r^3} \int_{B_r(x) \cap K} \dist(y,\mathscr{A})^2 \, \dd{\HH^1}(y), \\ 
    \exc_\mathscr{V}(x,r) &= \frac{1}{r}\int_{B_r(x) \cap K} |\mathscr{V} - T_y K|^2 \, \dd{\HH^1}(y)
\end{align*}
Finally, we define $\beta(x,r) = \min_{\mathscr{A} \in \mathcal{A}} \beta_{\mathscr{A}}(x,r)$ and $\exc(x,r) = \min_{\mathscr{V} \in \mathcal{L}} \exc_\mathscr{V}(x,r)$. In the literature, $\beta$ and $\exc$ are usually called, respectively, flatness and excess. 

As it is common in the proofs of epsilon-regularity theorems (see for example \cite[Section 3.2]{delellis-focardi}), the main ingredient to prove Theorem \ref{thm: epsilon-regularity} is a power-law decay of the excess. More precisely, we would like to show that if $(K,u)$ is a minimizer of $E$ on $\Omega$ with boundary datum $g$, $x \in K$ and $d(x,r)$, $\beta(x,r)$ and $r$ are sufficiently small, then $\exc(x,\rho) \leq C (\frac{\rho}{r})^a$ for some $C,a>0$ and for every $\rho \in (0,r)$. The precise statement is Proposition \ref{prop: decay}.
This is based on two separate decay results: one for the rescaled Dirichlet energy $d$ and one for the flatness $\beta$. This is possible since the excess is related to the other two quantities by the Tilt Lemma (Lemma \ref{lmm: tilt lemma}).

Clearly, if $D=\dist(x,\partial\Omega)$ and $D/r >1$, then the excess decay follows from the theory at the interior (\cite[Proposition 3.2.1]{delellis-focardi}).
The situation $D/r <1$, instead, is delicate. 
The main issue arises from the term $(D/r)^2$ that we get in our Tilt Lemma (Lemma \ref{lmm: tilt lemma}). More precisely, if $D/r<1$ and $\mathscr{V}$ is the tangent line to $\partial \Omega$ at the closest point to $x$ (see Figure \ref{fig: boundary of omega, tangent lines}), we get 
\begin{equation*}
    \exc_\mathscr{V}(x,\tfrac{r}{4}) \leq C \left(d(x,r) + \beta_{x+\mathscr{V}}(x,r)+\left(\frac{D}{r}\right)^2+r\right).
\end{equation*}
\begin{figure}[ht]
    \begin{tikzpicture}[scale=1.7]
        \def\a{20}
        \def\xa{1}
        \def\ya{0}
        \coordinate (A) at (\xa,\ya);
        \draw (0,0) to[out=10, in={\a+180}] (A) to[out=\a, in=-70](1.5,1) to[out=180-70, in=-70]  node[above] {\footnotesize $\partial\Omega$} (0.5,1.5) to[out=110, in=20] (0,2) to[out=200, in=30] (-0.5, 1.5) to[out=210, in=70] (-1,1) to[out=250, in=100] (-0.5, 0.5) to[out=280, in=190] (0,0);
        \def\m{-1/(tan(\a))}
        \def\dx{0.07}
        \def\dy{\m*\dx}
        \def\r{sqrt(\dx*\dx+\dy*\dy)}
        \def\xb{{\xa-\dx}}
        \def\yb{{\ya-\dy}}
        \coordinate (B) at (\xb,\yb);
        \draw[very thin] (A) -- (B);
        \draw[very thin] (B) circle [radius = \r*2];
        \filldraw[black] (B) circle (0.5pt) node[above] {\footnotesize $x$};
        \filldraw[black] (A) circle (0.5pt) node[below] {\footnotesize $y$};
        \node at (\xa+0.05,\ya+0.12) {\footnotesize $D$};
        \draw[very thin] (B) -- ({\xa-\dx-\r*2}, \yb);
        \node at ({\xa-\dx-\r}, {\ya-\dy+0.07}) {\footnotesize$r$};
        \def\mm{tan(\a)}
        \def\t{0.7}
        \draw ({\xa-\t}, {\ya-\mm*\t}) -- ({\xa+\t}, {\ya+\mm*\t}) node[right] {\footnotesize$y+\mathscr{V}$};
        \draw ({\xa-\dx-\t}, {\ya-\dy-\mm*\t}) -- ({\xa-\dx+\t}, {\ya-\dy+\mm*\t}) node[right] {\footnotesize$x+\mathscr{V}$};
    \end{tikzpicture}
    \caption{}
    \label{fig: boundary of omega, tangent lines}
\end{figure}
Notice that, if $\frac{D}{r} \leq A \beta_{x+\mathscr{V}}(x,r)^\frac{1}{2}$ for some constant $A$, then the excess is controlled by the quantities of interest. Under this additional assumption, we obtain a decay of the flatness (Lemma \ref{lmm: flatness improvement} \ref{item: flatness improvement, case D/r < b^1/2}) and of the rescaled Dirichlet energy (Lemma \ref{lmm: rescaled dirichlet energy decay} \ref{item: Dirichlet, case 1}): 
\begin{align*}
    \beta_{x+\mathscr{V}}(x,\tau r) \leq \tau \beta_{x+\mathscr{V}}(x,r), \quad d(x,\tau r) \leq \tau^\frac{1}{2} d(x,r).
\end{align*}
If the point $x$ is exactly on the boundary, that is $D=0$, then we are able to iterate the above estimates, since $\frac{D}{\tau^k r} = 0 \leq \beta_{x+\mathscr{V}}(x,\tau^k r)^\frac{1}{2}$ for every $k \in \N$. This leads to the desired decay property (Proposition \ref{prop: decay} \ref{item: decay lemma, D=0}).

If $0 < \frac{D}{r} < 1$, the situation is more delicate. Indeed, observe that, even if initially $\frac{D}{r} \leq \beta_{x+\mathscr{V}}(x,r)^\frac{1}{2}$, for sufficiently large $k$ we get $\frac{D}{\tau^k r} >  \beta_{x+\mathscr{V}}(x,\tau^k r)^\frac{1}{2}$, since the term on the right decreases as we increase $k$, while the term on the left increases. 

Let $R_\theta$ denote the counterclockwise rotation of angle $\theta$. If $\frac{D}{r} > B \, \beta_{x+R_\theta(\mathscr{V})}(x,r)^\frac{1}{3}$ for a small angle $\theta$ and for some constant $B$, our Tilt Lemma (Lemma \ref{lmm: tilt lemma} \ref{item: tilt lemma, case > b^1/3}) becomes: 
\begin{equation*}
    \exc_{R_\theta(\mathscr{V})}(x,\tfrac{r}{4}) \leq C (d(x,r)+\beta_{x+R_\theta(\mathscr{V})}(x,r)).
\end{equation*}
Also in this case, adding the assumption that $\frac{D}{r} \leq \eta$ for some constant $\eta$, we are able to obtain a decay (Lemma \ref{lmm: flatness improvement} \ref{item: flatness improvement, case D/r > b^1/3} and Lemma \ref{lmm: rescaled dirichlet energy decay} \ref{item: Dirichlet, case 3}, \ref{item: Dirichlet, case 4}). More precisely, we get, for $b \lesssim \beta_{x+R_\theta(\mathscr{V})}(x,r)^\frac{1}{2}$, 
\begin{equation*}
    \min_{|\tan(\phi)|\leq b}\beta_{x+R_{\phi+\theta}(\mathscr{V})}(x,\tau r) \leq \tau \beta_{x+R_\theta(\mathscr{V})}(x,r), \quad d(x,\tau r) \leq \tau^\frac{1}{2} d(x,r).
\end{equation*}
Moreover, observe that if $\frac{D}{r} > B \, \beta_{x+R_\theta(\mathscr{V})}(x,r)^\frac{1}{3}$ then $\frac{D}{\tau^k r} > B \, \beta_{x+R_\theta(\mathscr{V})}(x,\tau^k r)^\frac{1}{3}$ for every $k \in \N$. This allows us to iterate the above estimates until we get $\frac{D}{\tau^k r} > \eta$. At that point we can jump to the interior case ($\frac{D}{\eta\tau^k r}> 1$) and everything works fine.

This leads us to consider three different regimes: 
\begin{enumerate}
    \item $\frac{D}{r} \leq A \beta_{x+\mathscr{V}}(x,r)^\frac{1}{2}$
    \item $A \beta_{x+\mathscr{V}}(x,r)^\frac{1}{2} < \frac{D}{r} \leq B \beta_{x+\mathscr{V}}(x,r)^\frac{1}{3}$
    \item $B \beta_{x+\mathscr{V}}(x,r)^\frac{1}{3} > \frac{D}{r}$.
\end{enumerate}

If $r$ is in the second regime, then we are not able to get a nice decay as in the other two regimes. Nevertheless, we obtain $\sigma$ (see Lemma \ref{lmm: flatness preserved in the unfavourable case}) such that $\sigma r$ is in the third regime and 
\begin{equation*}
    \beta_{x+\mathscr{V}}(x,\rho) \leq C \beta_{x+\mathscr{V}}(x,r)^\frac{1}{8} \left(\frac{\rho}{r}\right)^\frac{1}{8}, \quad d(x,\rho) \leq C \beta_{x+\mathscr{V}}(x,r)^\frac{1}{8} \left(\frac{\rho}{r}\right)^\frac{1}{8}
\end{equation*}
for any $\sigma r \leq \rho \leq r$. Our quantities do not decay, but stay controlled by a power of $\beta$ times a constant. This is the bridge between the first and the third regime, and allows us to obtain the desired decay result (Proposition \ref{prop: decay}).
\subsection*{Acknowledgements} The author is very grateful to Camillo De Lellis for proposing this interesting problem and for the helpful suggestions. 
\tableofcontents

\section{Definitions and main results}
\label{section: definitions and main results}
In the section we give some definitions and state precisely the main results of this paper, other than Theorem \ref{thm: epsilon-regularity}.

Recall the definition of $\Omega$ and $g$ given in Section \ref{section: introduction}. In order to avoid technical details, we will assume throughout the paper that 
\begin{equation*}
    \partial \Omega \in C^2, \quad g \in C^2(\partial\Omega).
\end{equation*}
Moreover, without renaming it, we replace $g$ with a $C^{2}$ compactly supported extension defined on $\R^2$.
We define the set of admissible pairs $\mathcal{A}(\Omega, g)$ as 
\begin{equation*}
    \mathcal{A}(\Omega,g) = \{(K,u): K \subset \overline{\Omega} \text{ closed }, u \in W^{1,2}(\Omega\setminus K), u=g \text{ on } \partial \Omega \setminus K \}.
\end{equation*}
\subsection{Density lower bound}
A fundamental result needed in this paper is the following density lower bound. 
\begin{thm}[Density lower bound]
    \label{thm: density lower bound for K}
    There exist $\eps, r >0$ depending on $\Omega$ and $g$, such that, if \minimizer{K}{u}{\Omega}{g}, then 
    \begin{equation*}
        \HH^1(K \cap \overline{B}_\rho(x)) > \eps \rho \quad \forall x \in K, \, \forall \rho \in (0,r).
    \end{equation*}
\end{thm}
Its counterpart at the interior was obtained in \cite{degiorgi-carriero-leaci} to prove existence of minimizers and a precise statement can be found in \cite[Theorem 2.1.3]{delellis-focardi}.

The proof of this theorem follows essentially from the ideas contained in \cite{carriero-leaci}, where the authors obtain the above result, even if they do not write it explicitly, in the case that $x \in K \cap \partial \Omega$. 

However, in this paper, we also need to work with points that are not on the boundary, but very close. This requires an additional argument that is interesting. 

Since we think that an explicit statement is useful, in Section \ref{section: density lower bound}  we report a Lemma of \cite{carriero-leaci} and give a proof Theorem \ref{thm: density lower bound for K}.
\subsection{Blow-up}
\label{section: blow-up}
To understand the properties of $K$, we take a point $x \in K \cap \partial \Omega$ and look at the limit of $\frac{K-x}{\rho}$ as $\rho \downarrow 0$, called the blow-up of $K$ at point $x$. More in general, throughout the paper, we will also need to take blow-ups of a sequence of minimizers with respect to a sequence of points converging to the boundary. Let us introduce some notation. For any set $A$ and function $f$, we write, for a point $x$ and a radius $r>0$, 
\begin{equation*}
    A_{x,r} = \frac{A-x}{r}, \quad f_{x,r}(y) = r^{-\frac{1}{2}}f(x+ry)
\end{equation*}
whenever $x+ry$ is in the domain of $f$.
Suppose that
\begin{equation*}
    \begin{split}
        &(K^n,u^n) \in \argmin\{E(J,v): (J,v) \in \mathcal{A}(\Omega,g)\}, \quad x_n \in \overline{\Omega}, \quad r_n >0 \quad \forall n \in \N, \\
        &x_n \to x \in \partial \Omega, \quad r_n \to 0.
    \end{split}
\end{equation*}
Define, for every $n \in \N$,
\begin{equation*}
    K_n = (K^n)_{x_n, r_n}, \; \Omega_n = \Omega_{x_n, r_n}, \; u_n = (u^n-g(x_n))_{x_n, r_n}, \; g_n = (g-g(x_n))_{x_n,r_n}.
\end{equation*}
Up to translations and rotations, we can assume that $x = 0$ and $T_x \partial \Omega = \{x_2=0\}$.
Then, up to subsequences, $\Omega_n \to \Omega_0$ locally in the Hausdorff distance, where 
\begin{equation*}
    \Omega_0 = \{x_2 > -a\}, \quad \partial\Omega_0 = \{x_2 = -a\} = T_x\partial\Omega -a,
\end{equation*}
for some $a \in [0,+\infty]$. 

We are not interested in the case $a = +\infty$ since, in this case, $\Omega_0 = \R^2$ and the theory for the classical Mumford-Shah functional (at the interior) applies (see \cite[Chapter 2.2]{delellis-focardi}). Therefore, in the following, unless otherwise specified, we will always assume $a \in \R$. 

By Blaschke's Theorem (see \cite[Theorem 6.1, p.320]{AFP}) and a classical diagonal argument, $\{K_n\}_{n\in\N}$ converges, up to subsequences, to a closed subset $K_0$ of $\overline{\Omega}_0$ locally in the Hausdorff distance.

Since $\Omega_0 \setminus K_0$ is an open subset of $\R^2$, it has at most countably many connected components, which will be denoted by $\{\Omega^k\}_{k \in \mathscr{I}}$, where $\mathscr{I} \subset \N \setminus \{0\}$ is a set of indices.

Moreover, we define the set $\Omega^*$ as follows.
\begin{dfn}
    \label{def: Omega^*}
    We denote by $\Omega^*$ the connected component of $\R^2 \setminus K_0$ containing point $z = (0,-a-1)$.
\end{dfn}

We prove, in Section \ref{section: proof of blow-up convergence}, the following convergence result. On connected components that see the boundary ($\Omega^k \subset \Omega^*$), using the fact that the sequence of boundary data $g_n$ is converging locally uniformly to zero, we are able to get convergence of the sequence $\{u_n\}$, while on connected components that do not see the boundary, we need to subtract the value of $u_n$ at a point $z^k \in \Omega^k$ in order to get convergence.  
\renewcommand{\myset}{\Omega^k \cup (\overline{\Omega^k} \cap \partial\Omega_0 \setminus K_0)}
\begin{prop}[Blow-up convergence]
    \label{prop: global convergence of functions, blow-up}
    Let $\Omega^k$ be a connected component of $\Omega_0 \setminus K_0$. Then 
    \begin{enumerate}[label=(\roman*)]
        \item If $\Omega^k \subset \Omega^*$, then, up to subsequences, 
        \begin{equation*}
            u_n \to u^k \quad \text{in } C^\infty_\loc (\Omega^k) \cap C^1_\loc(\myset),
        \end{equation*}
        for some function $u^k \in C^\infty(\Omega^k) \cap C^1(\myset)$ that is harmonic in $\Omega^k$, where the above convergence is understood in the sense of 
        \eqref{eq: C^1_loc convergence abuse of notation}.
        \item Otherwise, let $z^k \in \Omega^k$ and define $v_n = u_n - u_n(z^k)$. Then, up to subsequences, 
        \begin{equation*}
            v_n \to u^k \quad \text{in } C^\infty_\loc (\Omega^k),
        \end{equation*}
        for some harmonic function $u^k: \Omega^k \to \R$. 
    \end{enumerate}
\end{prop}
Now define $u_0: \overline{\Omega}_0\setminus K_0 \to \R$  by
\begin{equation*}
    \label{eq: definition of u_0, blow-up limit}
    u_0(y) = u^k(y) \text{ if } y \in \overline{\Omega^k}.
\end{equation*}
Then $u_0 \in C^{\infty}(\Omega_0 \setminus K_0) \cap C^1(\overline{\Omega}_0 \setminus K_0)$ and $u_0$ is harmonic in $\Omega_0 \setminus K_0$.

We define the blow-up limit as follows. 
\begin{dfn}[Blow-up limit]
    \label{def: blow-up limit}
    The pair $(K_0,u_0)$ constructed above will be called blow-up limit of (a subsequence of) the blow-up sequence $\{(K_n, u_n)\}_{n \in \N}$.
\end{dfn}

Adapting \cite[Definition 2.2.2]{delellis-focardi}, we define admissible competitors and generalized minimizers on a half-plane with zero boundary datum.

\begin{dfn}[Admissible competitor and generalized minimizer on a half-plane]
    \label{def: admissible competitor and minimizer on a half-plane}
   Let $U \in \R^2$ be an open half-plane. Let 
    \begin{equation*}
       \mathcal{A}(U) := \{(K,u): K \subset \overline{U}, K \text{ closed }, u \in W^{1,2}_\loc (U \setminus K), u=0 \text{ on } \partial U \setminus K\}
    \end{equation*}
    and $(K,u) \in \mathcal{A}(U)$. 

    A pair $(J,w) \in \mathcal{A}(U)$ is called an \textit{admissible competitor} for $(K,u)$ on $U$ provided that, for some $O$ open bounded subset of $\R^2$, 
    \begin{enumerate}
        \item $(J,w)$ coincides with $(K,u)$ in $\overline{U} \setminus O$,
        \item if $z,w \in \R^2 \setminus (O \cup K)$ belong to distinct connected components of $\R^2 \setminus K$, then they belong to distinct connected components of $\R^2 \setminus J$.
    \end{enumerate}

    Moreover, we say that $(K,u)$ is a \textit{generalized minimizer} on $U$ if 
    $E(K,u) \leq E(J,v)$ for every admissible competitor $(J,v)$. 
\end{dfn}
\begin{rmk}
    Our definition of admissible competitor seems identical to the one of \text{topological competitor} given in \cite[Definition 2.2.2]{delellis-focardi} and \cite[Definition 7.15]{david}, but is actually different. The reason is that, in condition 2, we consider the connected components of $\R^2 \setminus K$, and not those of $U \setminus K$. 

    In particular, we allow competitors to reconnect points in distinct connected components that see the boundary of the half-plane. In the case that $U$ is equal to the blow-up domain $\Omega_0$ defined above, those are the connected components contained in $\Omega^*$ (see Definition \ref{def: Omega^*}).
\end{rmk}

We prove that the blow-up limit has interesting properties. As one would expect, its energy on bounded sets is the limit of the energy of the blow-up sequence. Moreover, the blow-up limit is a minimizer in the class of competitors defined above. The precise statement, that will be proved in Section \ref{section: proof of blow-up limit theorem}, is the following.  
\begin{thm}[Blow-up limit]
    \label{thm: blow-up limit}
    Let $(K_0, u_0)$ be the above constructed pair. Then 
    \begin{enumerate}[label=(\roman*)]
        \item $K_0$ is $\HH^1$-rectifiable.
        \item For every open set $U \subset \subset \overline{\Omega}_0$ such that $\HH^1(\partial U \cap K_0 \cap \Omega_0) = 0$, we have 
        \begin{align*}
            \lim_{n \to \infty} \HH^1(K_n \cap \overline{U}) = \HH^1(K_0 \cap \overline{U}), \\ 
            \lim_{n \to \infty} \int_{U \cap (\Omega_n \setminus K_n)} |\nabla u_n|^2 = \int_{U \setminus K_0} |\nabla u_0|^2.
        \end{align*}
        \item $(K_0,u_0)$ is a generalized minimizer on $\Omega_0$ according to Definition \ref{def: admissible competitor and minimizer on a half-plane}.
    \end{enumerate}
\end{thm}
\begin{rmk}
    Our blow-up limit is the equivalent the global generalized minimizer in $\R^2$ defined, for the problem at the interior, in \cite[Definition 2.2.4]{delellis-focardi}.
    However, to avoid cumbersome notation that we do not need in the following, we do not state a result like \cite[Theorem 2.2.3(iii)]{delellis-focardi}.
\end{rmk}
\subsection{Generalized minimizers on the upper half-plane}
We are interested in studying generalized minimizers (see Definition \ref{def: admissible competitor and minimizer on a half-plane}) on the open upper half-plane $U = \{x_2 > 0\}$ with homogeneous Dirichlet boundary conditions, since blow-up limits, up to change a coordinates, belong to this category by Theorem \ref{thm: blow-up limit}(iii).
Ideally, we would like to classify them. Even if we are not able to do so, we prove some monotonicity formulae and obtain, as a corollary, some partial classification results.

In order to use blow-up and blow-in arguments, we need the following compactness theorem, whose proof follows at once from the one of Theorem \ref{thm: blow-up limit} and can be found in Section \ref{section: generalized minimizers on the upper half-plane}.

\begin{thm}[Compactness for generalized minimizers on the upper half-plane]
    \label{thm: compactness for generalized minimizers}
    Let $(K,u)$ be a generalized minimizer on $U$. 
    Let $x \in \overline{U}$. 
    If $x \in \partial U$, suppose that $r_j \downarrow 0$ or $r_j \uparrow +\infty$ as $j\to \infty$, otherwise suppose that $r_j \uparrow +\infty$. 
    Let $K_j = K_{x,r_j}$, $u_j = u_{x,r_j}$, $U_j = U_{x,r_j}$.
     Then, up to subsequences, $U_j \to \{x_2>0\}$ locally in the Hausdorff distance, $(K_j,u_j)$ converges to $(K_0, u_0)$ in the sense of Proposition \ref{prop: global convergence of functions, blow-up}, and the conclusions of Theorem \ref{thm: blow-up limit} hold. 
\end{thm}

The main tools to analyze blow-up ($r_j \downarrow 0$) and blow-in ($r_j \uparrow \infty$) limits are monotonicity formulae. 
We prove, in Section \ref{section: monotonicity formuale}, the following proposition, which is the analogous for our problem of the famous David-Léger monotonicity formula \cite[Proposition 2.13]{david-leger}.
\begin{prop}
    \label{prop: david-leger monotonicity}
    Let $(K,u)$ be a generalized minimizer on $U$ according to Definition \ref{def: admissible competitor and minimizer on a half-plane}. Assume that $x \in \partial U \cap K$. Define, for $r>0$, 
    \begin{equation*}
        F(r) = \frac{2}{r} \int_{B_r(x)\cap U \setminus K} |\nabla u|^2 + \frac{1}{r} \HH^1(K \cap B_r(x)).
    \end{equation*}
    Then $F'(r) \geq 0$ for a.e. $r>0$ such that either $\HH^0(K \cap \partial B_r(x)) \in \{0\} \cup [2,\infty)$ or $K \cap B_r(x) = \{x + r e^{i\theta}\}$ with $\theta \in (\frac{1}{4}\pi,\frac{3}{4}\pi)$.

    Moreover, if $F$ is constant and $K$ satisfies the above property for a.e. $r$, then $(K,u)$ is an elementary generalized minimizer. 
    
\end{prop}

We also prove, in Section \ref{section: monotonicity formuale}, the following result, which is the analogous of famous Bonnet's monotonicity formula \cite[Theorem 3.1]{Bonnet}.
\begin{prop}
    \label{prop: Bonnet monotonicity}
    Let $(K,u) \in \mathcal{A}$ be a generalized minimizer on $U$ according to Definition \ref{def: admissible competitor and minimizer on a half-plane}.
    Let $x \in K \cap \partial U$. 
    We define, for $r>0$,
    \begin{equation*}
        \omega(r):= \frac{1}{r} \int_{B_r(x) \cap U \setminus K} |\nabla u|^2
    \end{equation*}
    Then the following hold:
    \begin{enumerate}[label=(\roman*)]
        \item If $K$ is connected, then $\omega'(r) \geq 0$ for a.e. $r >0$. Moreover, if $\omega$ is constant, then $\omega \equiv 0$. 
        \item $\omega'(r) \geq 0$ for a.e. $r>0$ such that $\HH^0(K \cap \partial B_r(x)) = 1$. Moreover, if $\omega$ is constant, then $\omega \equiv 0$. 
    \end{enumerate}
\end{prop}

Observe that, if we don't assume that $K$ is connected, the first monotonicity formula (Proposition \ref{prop: david-leger monotonicity}), without further assumptions, works only if $\HH^0(\partial B_r(x) \cap K) \in \N \setminus \{1\}$, while the second one works only if $\HH^0(\partial B_r(x) \cap K)  = 1$. The two monotonicity results complement each other, but we do not know on how to combine them in a meaningful way. For this reason, we are unable to give a complete classification of generalized minimizers. Nevertheless, we are able to give a partial classification based on the above monotonicity results. 

First of all, using Allard's monotonicity formula for stationary varifolds (see \cite{allard-interior,allard-boundary}), we classify generalized minimizers with zero Dirichlet energy, that we define below.
\begin{dfn}[Elementary generalized minimizer]
    \label{def: elementary generalized minimizers}
   We say that a generalized minimizer $(K,u)$ is elementary if $\int |\nabla u|^2 = 0$. 
\end{dfn}

We obtain the following theorem, whose proof can be found in Section \ref{section: elementary generalized minimizers}.
\begin{thm}[Classification of elementary generalized minimizers on the upper half-plane]
    \label{prop: classification of elementary generalized minimizers}
    Let $(K,u)$ be an elementary generalized minimizer according to Definition \ref{def: elementary generalized minimizers}. Then one and only one of the following holds. 
    \begin{enumerate}[label=(\roman*)]
        \item $K= \emptyset$ and $u=0$,
        \item $K = \partial U = \{x_2 = 0\}$ and $u$ is constant,
        \item $K = \{x_2 = a\}$ for some $a \in (0,+\infty)$ and $u = u^+ \chi_{\{x_2 > a\}}$ for some constant $u^+ \in \R$.
    \end{enumerate}
\end{thm}

The above theorem implies at once the following, proved in Section \ref{section: elementary generalized minimizers}, which is the analogous of \cite[Theorem 2.4.1]{delellis-focardi}.
\begin{thm}[Classification of elementary blow-up limits] 
    \label{thm: classification of elementary blow-up limits}
    Suppose that $(K,u)$ is a blow-up limit of $\{(K_n,u_n)\}_n$ as in Definition \ref{def: blow-up limit}, and suppose that $\Omega_n \to \Omega_0 = \{x_2 > -a\}$, with $a \in [0,+\infty)$. Moreover, suppose that $\int |\nabla u|^2 = 0$. Then one and only one of the following holds. 
    \begin{enumerate}[label=(\roman*)]
        \item $K= \emptyset$ and $u=0$,
        \item $K = \partial \Omega_0 = \{x_2 = -a\}$ and $u$ is constant. In this case, let $\Omega^1 = \{x_2 > -a\}$ and let $z^1 \in \Omega^1$ be the point used in the convergence of Proposition \ref{prop: global convergence of functions, blow-up}. Then, up to subsequences, $\lim_{n \to \infty} |u_n(z^1)| = \infty$.
        \item $K = \{x_2 = b\}$ for some $b > -a$ and $u = u^+ \chi_{\{x_2 > b\}}$ for some constant $u^+ \in \R$. In this case, let $\Omega^1 = \{x_2 > b\}$ and let $z^1 \in \Omega^1$ be the point used in the convergence of Proposition \ref{prop: global convergence of functions, blow-up}. Then, up to subsequences, $\lim_{n \to \infty} |u_n(z^1)| = \infty$.
    \end{enumerate}
\end{thm}

In the same spirit as \cite{Bonnet} we classify connected generalized minimizers using the monotonicity formula of Proposition \ref{prop: Bonnet monotonicity}. 
\begin{prop}[Classification of connected generalized minimizers]
    \label{prop: classification of connected generalized minimizers}
    Let $(K,u)$ be a generalized minimizer on $U = \{x_2 > 0\}$ as in Definition \ref{def: admissible competitor and minimizer on a half-plane}. 
    Suppose that $K$ is connected and $x \in K \cap \partial U$. Then $(K,u)$ is an elementary generalized minimizer according to Definition \ref{def: elementary generalized minimizers}. In particular, $K = \partial U$. 
\end{prop}

Moreover, using (ii) in Proposition \ref{prop: Bonnet monotonicity}, we show the following result.
\begin{prop}
    \label{prop: exclude 1 intersection}
    Let $(K,u)$ be a generalized minimizer on $U = \{x_2 > 0\}$ as in Definition \ref{def: admissible competitor and minimizer on a half-plane}. 
    Suppose that $x \in K \cap \partial U$. Then it is not possible to have $\HH^0(K \cap \partial B_r(x)) = 1$ for a.e. $r>0$. 
\end{prop}

Finally, we classify generalized minimizers that satisfy, for a.e. $r$, the property of Proposition \ref{prop: david-leger monotonicity}.
\begin{prop}
    \label{prop: david-leger classification of generalized minimizers}
    Assume that $(K,u)$ is a generalized minimizer on $U = \{x_2>0\}$ and $x \in K \cap \partial U$.  If, for a.e. $r>0$, either $\HH^0(K \cap \partial B_r(x)) \in \{0\} \cup [2,\infty)$ or $K \cap B_r(x) = \{x + r e^{i\theta}\}$ with $\theta \in (\frac{1}{4}\pi,\frac{3}{4}\pi)$, then $(K,u)$ is elementary, and hence $K = \partial U$. 
\end{prop}
Propositions \ref{prop: classification of connected generalized minimizers}, \ref{prop: exclude 1 intersection} and \ref{prop: david-leger classification of generalized minimizers} will be proved in Section \ref{section: partial classification of generalized minimizers}.
\section{Preliminary results and proof of Theorem \ref{thm: epsilon-regularity}}
\label{section: preliminary results and proof of epsilon-regularity}
In this section we prove Theorem \ref{thm: epsilon-regularity}, after stating the necessary preliminary results. Recall the definition of $\beta(x,r)$, $d(x,r)$ and $\exc(x,r)$ given in Section \ref{section: overview of the proof of eps-regularity}. For all the above quantities, we drop the dependence on $x$ if $x=0$.
We define, for every $x \in \partial \Omega$ and $r>0$, 
\begin{equation*}
    \omega(x,r) = r^{-1} \dist_H(K \cap \overline{B}_{2r}(x), (x+T_x \partial \Omega) \cap \overline{B}_{2r}(x)) + r^{-1} \int_{B_{2r}(x) \cap \Omega \setminus K} |\nabla u|^2, 
\end{equation*}
where $\dist_H$ denotes the Hausdorff distance. This quantity tells us how close $K$ is to the tangent line $x + T_x \partial \Omega$ in a certain ball, and how small is the rescaled Dirichlet energy of $u$. We will show that if this quantity is small at a certain scale, then at a smaller scale $K$ is the graph of a $C^{1,\alpha}$ function.

The following Lemma, that will be proved in Section \ref{section: Proof of epsilon-regularity}, implies that if the first term in $\omega$ is small, then $\omega$ is automatically small at a smaller scale. 
\begin{lmm}
    \label{lmm: the Dirichlet term can be removed from eps-regularity statement}
    For any $\eps > 0$, there exist $\delta, \eta > 0$ with the following property. 
    Suppose that \minimizer{K}{u}{\Omega}{g}, $x \in \partial \Omega$ and 
    \begin{equation*}
        r^{-1} \dist_H(K \cap \overline{B}_{2r}(x), (x+T_x \partial \Omega) \cap \overline{B}_{2r}(x)) + r^\frac{1}{2}< \delta.
    \end{equation*}
    Then $\omega(x,\eta r) + (\eta r)^\frac{1}{2} < \eps$. 
\end{lmm}

In particular, thanks to previous lemma, in order to prove Theorem \ref{thm: epsilon-regularity}, it suffices to prove the following proposition.
\begin{prop}
    \label{prop: epsilon-regularity with Dirichlet term}
    There exist constants $\varepsilon, \alpha, \eta > 0$ with the following property. Assume 
    \begin{enumerate}[label=(\roman*)]
        \item $(K,u) \in \argmin\{E(J,v): (J,v) \in \mathcal{A}(\Omega,g)\}$ and $x\in K \cap \partial\Omega$,
        \item $\omega(x,r) + r^\frac{1}{2} < \varepsilon$
    \end{enumerate}
    Then $K \cap B_{\eta r}(x)$ is the graph of a $C^{1,\alpha}$ function $f$ over $x+T_x \partial \Omega$.
\end{prop}

We will prove Proposition \ref{prop: epsilon-regularity with Dirichlet term} after stating the necessary propositions.

\subsection{Tilt Lemma}
The following lemma, that we prove in Section \ref{section: proof of tilt lemma}, relates the excess to the rescaled Dirichlet energy and to the flatness.
In this section, we need only point \ref{item: tilt lemma, case D/r < 1}, but we state the complete lemma since it will be needed in the rest of the paper. 
Let $R_\theta$ be the counterclockwise rotation of angle $\theta$.

\begin{lmm}[Tilt Lemma]
    \label{lmm: tilt lemma}
    Let $2\overline{r} >0 $ be given by Lemma \ref{lmm: radius for boundary}.
    There exist constants $A,B,C,E,\delta > 0$ with the following property. 
    Suppose that 
    \begin{equation}
        \label{eq: tilt lemma, K,u,x,r,D,V condition}
        \begin{split}
            &(K,u) \in \argmin\{E(J,v): (J,v) \in \mathcal{A}
            (\Omega,g)\},\\
            &x \in K,\, r \leq \overline{r},\\
            &D:= \dist(x,\partial\Omega), \, y \in \partial \Omega \text{ is such that } |y-x|=D,\, \mathscr{V} := T_y \partial \Omega.
        \end{split}
    \end{equation}
    \begin{enumerate}[label=(\roman*)]
        \item \label{item: tilt lemma, case D/r < 1} If 
        \begin{equation}
            \label{eq: tilt lemma, case D/r < 1}
           0\leq \frac{D}{r} <1 
        \end{equation}
        then 
        \begin{equation*}
            \exc_\mathscr{V}(x,\tfrac{r}{4}) \leq C \left(d(x,r) + \beta_{x+\mathscr{V}}(x,r)+\left(\frac{D}{r}\right)^2+r\right).
        \end{equation*}
        \item \label{item: tilt lemma, case > b^1/3} If 
        \begin{align}
            &1 > \frac{D}{r} > A \, \beta_{x+R_\theta(\mathscr{V})}(x,r)^\frac{1}{3},  \label{eq: tilt lemma, case > b^1/3}\\
            &\frac{D}{r} > Br, \label{eq: tilt lemma, condition on the radius to have curvature small}\\
            &\beta_{x+R_\theta(\mathscr{V})}(x,r) < \delta,  \label{eq: tilt lemma, smallness of flatness assumption}\\
            &|\sin\theta| < E \frac{D}{r}, \quad |\cos\theta| > \frac{9}{10} \label{eq: tilt lemma, angle condition}
        \end{align}
        for some $\theta \in [0,2\pi)$,
        then 
        \begin{equation*}
            \exc_{R_\theta(\mathscr{V})}(x,\tfrac{r}{4}) \leq C (d(x,r)+\beta_{x+R_\theta(\mathscr{V})}(x,r)),
        \end{equation*}
    \end{enumerate}
\end{lmm}
\begin{cor}
    In the above Lemma, if $0\leq \frac{D}{r} \lesssim \beta_{x + \mathscr{V}}(x,r)^\frac{1}{2}$, then we also get 
    \begin{equation*}
        \exc_\mathscr{V}(x,\tfrac{r}{4}) \leq C \left(d(x,r) + \beta_{x+\mathscr{V}}(x,r)+r\right) \quad \forall r \leq \overline{r}.
    \end{equation*}
\end{cor}
\subsection{Lipschitz approximation}
The following proposition, that we prove in Section \ref{section: proof of Lipschitz approximation}, allows us to approximate $K$ with a Lipschitz function where flatness and rescaled Dirichlet energy are small. In this section, we need only point \ref{item: Lipschitz approximation, D/r < 1}, but we state the complete result since it will be needed in the rest of the paper. 

For $\alpha \in [0,2\pi)$ and $x \in \R^2$ we define $\pi_\alpha (x) = x_1 \cos\alpha + x_2 \sin\alpha$, $\pi_\alpha^\perp(x) = -x_1 \sin\alpha + x_2 \cos\alpha$ and $Q^\alpha_\rho (x) = \{y: |\pi_\alpha (y-x)|\leq \rho, |\pi_\alpha^\perp(y-x)| \leq \rho\}$.
Moreover, let $R_\alpha$ be the counterclockwise rotation of angle $\alpha$. 
\begin{lmm}[Lipschitz approximation]
    \label{lmm: Lipschitz approximation}
    There exist $C, \delta, \sigma$ with the following properties. 
    Suppose that $(K,u)$, $x$, $r$, $y$, $D$, $\mathscr{V}$ satisfy property \eqref{eq: tilt lemma, K,u,x,r,D,V condition} and $\frac{D}{r} < \eta$, where $\eta$ is the constant of Lemma \ref{lmm: vertical separation}. Let $\phi \in [0,2\pi)$ be such that $\mathscr{V} = R_\phi(\mathscr{V}_0)$, where $\mathscr{V}_0 = \{x_2=0\}$.

    \begin{enumerate}[label=(\alph*)]
        \item \label{item: Lipschitz approximation, D/r < 1} 
        Suppose that 
        \begin{equation*}
            d(x,r) + \beta_{x+\mathscr{V}}(x,r)+r < \delta.
        \end{equation*}
        Then 
        \begin{enumerate}[label=(\roman*)]
            \item $K \cap B_\frac{r}{2}(x) \subset \{y: |\pi_\phi^\perp(y-x)| \leq C r \beta_{x+\mathscr{V}}(x,r)^\frac{1}{3}\}$, 
        \end{enumerate} 
         there exists a Lipschitz function $f: [-\sigma r, \sigma r] \to \R$ with $\mathrm{Lip}(f) \leq 1$ such that $\Gamma_f := x+R_\phi(\graph(f)) \subset \overline{\Omega}$ and with the following properties. 
         \begin{enumerate}[label=(\roman*), resume]
            \item $\norm{f}_{C^0} \leq C r \beta_{x+\mathscr{V}}(x,r)^\frac{1}{3}$,
            \item
                $\HH^1((\Gamma_f \triangle K) \cap Q_{\sigma r}^\phi) \leq C r \left(d(x,r)+\beta_{x+\mathscr{V}}(x,r)+\left(\frac{D}{r}\right)^2+r\right)$.
            \item
               $ \int|f'|^2 \leq C r \left(d(x,r)+\beta_{x+\mathscr{V}}(x,r)+\left(\frac{D}{r}\right)^2+r\right)$.
            \item For any $\Lambda > 0$, $\delta$ can be chosen such that, for any $\eps > C \beta_{x+\mathscr{V}}(x,r)^\frac{1}{3}$, 
            \begin{equation*}
                \HH^1([-\sigma r, \sigma r]\setminus \pi_{\mathscr{V}}(K \cap \{y: |\pi^\perp(y-x)| < \eps r\})) \leq \Lambda r d(x,r).
            \end{equation*}
            \item There exist $\overline{\eps}$ and $\overline{\tau}$ positive constants such that 
            \begin{equation*}
                \tilde{K}:= \left\{z \in K \cap B_{\frac{\overline{\tau}}{4}r} | \sup_{0<\rho<\frac{r}{2}} (d(z,\rho)+\exc_\mathscr{V}(z,\rho)) < \overline{\eps}\right\} \subset \Gamma_f.
            \end{equation*}
         \end{enumerate}
         \item \label{item: Lipschitz approximation, D/r > b^1/3} Suppose that conditions \eqref{eq: tilt lemma, case > b^1/3} \eqref{eq: tilt lemma, condition on the radius to have curvature small} and \eqref{eq: tilt lemma, angle condition} hold for some $\theta \in [0,2\pi)$. 
         Suppose that 
         \begin{equation*}
             d(x,r) + \beta_{x+R_\theta(\mathscr{V})}(x,r)+r < \delta.
         \end{equation*}
         Let $\alpha = \phi + \theta$.
         Then 
        \begin{enumerate}[label=(\roman*)]
            \item $K \cap B_\frac{r}{2}(x) \subset \{y: |\pi_\alpha^\perp(y-x)| \leq C r \beta_{x+R_\theta(\mathscr{V})}(x,r)^\frac{1}{3}\}$, 
        \end{enumerate} 
         there exists a Lipschitz function $f: [-\sigma r, \sigma r] \to \R$ with $\mathrm{Lip}(f) \leq 1$ such that $\Gamma_f := x+R_\alpha(\graph(f)) \subset \overline{\Omega}$ and with the following properties. 
         \begin{enumerate}[label=(\roman*), resume]
            \item $\norm{f}_{C^0} \leq C r \beta_{x+R_\theta(\mathscr{V})}(x,r)^\frac{1}{3}$,
            \item
               $ \HH^1((\Gamma_f \triangle K) \cap Q^\alpha_{\sigma r}) \leq C r (d(x,r)+\beta_{x+R_\theta(\mathscr{V})}(x,r))$.
            \item
                $\int|f'|^2 \leq  C r (d(x,r)+\beta_{x+R_\theta(\mathscr{V})}(x,r))$.
            \item For any $\Lambda > 0$, $\delta$ can be chosen such that, for any $\eps > C \beta_{x+R_\theta(\mathscr{V})}(x,r)^\frac{1}{3}$, 
            \begin{equation*}
                \HH^1([-\sigma r, \sigma r]\setminus \pi_{R_\theta(\mathscr{V})}(K \cap \{y: |\pi^\perp(y-x)| < \eps r\})) \leq \Lambda r d(x,r).
            \end{equation*}
         \end{enumerate}
    \end{enumerate}
\end{lmm}
\subsection{Excess decay}
The proposition below is the main ingredient in the proof of Theorem \ref{thm: epsilon-regularity}.
Its proof is based on several decay results and it is postponed to Section \ref{section: proof of decay proposition}.
\begin{prop}[Decay]
    \label{prop: decay}
    There exist $C, \delta >0$ with the following property.

    Suppose that $(K,u)$, $x$, $r$, $D$, $\mathscr{V}$ satisfy property \eqref{eq: tilt lemma, K,u,x,r,D,V condition} and 
    \begin{equation*}
        d(x,r) + \beta_{x+\mathscr{V}}(x,r) + r < \delta.
    \end{equation*} 
    Then the following holds. 
    \begin{enumerate}[label=(\roman*)]
        \item \label{item: decay lemma, D=0} If $D=0$ then, for any $0< \rho <r$,
        \begin{align*}
            &d(x,\rho) + \beta_{x+\mathscr{V}}(x,\rho) + \rho^\frac{1}{2} \leq \delta \left(\frac{\rho}{r}\right)^\frac{1}{2}, \\
            &\exc_{\mathscr{V}}(x,\tfrac{\rho}{4}) \leq C \delta \left(\frac{\rho}{r}\right)^\frac{1}{2}
        \end{align*}
        \item  \label{item: decay lemma, D>0} If $D>0$ then, for any $0 <\rho < r$, 
        \begin{align*}
            &d(x,\rho) + \beta(x,\rho) \leq C \delta^\frac{1}{12} \left(\frac{\rho}{r}\right)^\frac{1}{24} \\
            &\exc(x,\tfrac{\rho}{4}) \leq C \delta^\frac{1}{12} \left(\frac{\rho}{r}\right)^\frac{1}{24}.
        \end{align*}
    \end{enumerate}
\end{prop}
\subsection{Proof of Theorem \ref{thm: epsilon-regularity}}
\label{section: Proof of epsilon-regularity}
As we argued above, Theorem \ref{thm: epsilon-regularity} follows directly from Proposition \ref{prop: epsilon-regularity with Dirichlet term} and Lemma \ref{lmm: the Dirichlet term can be removed from eps-regularity statement}, which we prove in this section.
\begin{proof}[Proof of Proposition \ref{prop: epsilon-regularity with Dirichlet term}]
    Up to translations and rotations, assume that $x=0$ and $T_x \partial \Omega = \mathscr{V}_0 = \{x_2=0\}$. 
    By definition of $\omega$ and by the energy upper bound (Proposition \ref{prop: Energy upper bound at the boundary}), for some geometric constant $C>0$, $\beta_{\mathscr{V}_0}(0,2r) \leq C \omega(0,r)$.
    Therefore 
    \begin{equation}
        \label{eq: proof of eps-regularity, flatness and energy bound on (0,2r)}
        \beta_{\mathscr{V}_0}(0,2r) + d(0,2r) + (2r)^\frac{1}{2} \leq C (\omega(0,r)+r^\frac{1}{2}) =: \eps(r)  < C \eps
    \end{equation}

    Let $\overline{r}$ and $C'$ be the constants of Lemma \ref{lmm: radius for boundary} and take $\eps < \frac{1}{2}\overline{r}$.
    Then $\Omega \cap B_{2r} = \{x_2 > h(x_1)\} \cap B_{2r}$, with $h \in C^2(\R)$ and $\norm{h''}_\infty \leq C'$. 

    By the height bound (Lemma \ref{lmm: Lipschitz approximation}\ref{item: Lipschitz approximation, D/r < 1}(i)), we have 
    \begin{equation*}
        K \cap B_r \subset \{y: |y_2| \leq C r \beta_{\mathscr{V}_0}(0,2r)^\frac{1}{3} \}.
    \end{equation*}
    
    Let $z \in K \cap B_r$. 
    Then there exists $y = (t_y,h(t_y)) \in \partial\Omega \cap B_{2r}$ such that $|y-z| = \dist(z,\partial\Omega)$ and $|h'(t_y)| \leq C' r$.
    Let $\theta = \arctan h'(t_y)$.
    Then (see Figure \ref{fig: eps-reg})
    \begin{equation*}
        \begin{split}
            \beta_{z+T_y\partial\Omega}(z,r) &\lesssim \beta_{z+\mathscr{V}_0}(z,r) + \sin^2\theta \\
            &\lesssim \beta_{z+\mathscr{V}_0}(z,r) + r^2 \\
            &\lesssim \beta_{\mathscr{V}_0}(z,r) + \left(\frac{|z_2|}{r}\right)^2 + r^2 \\
            &\lesssim \beta_{\mathscr{V}_0}(0,2r) + \beta_{\mathscr{V}_0}(0,2r)^\frac{2}{3} + r^2.
        \end{split}
    \end{equation*}

    \begin{figure}[ht]
        \begin{tikzpicture}[>=stealth]
            \def\r{1.3}
            \draw[->](-3.4*\r,0)--(3.4*\r,0) node[right] {$x_1$};
            \draw[->](0,-2.6*\r)--(0,2.6*\r) node[above] {$x_2$};
            \draw (0,0) circle [radius=\r];
            \draw (0,0) circle [radius=\r*2];
            \def\a{0.1}
            \draw[thick][domain=-2.1*\r:2.1*\r] plot (\x,\a*\x*\x);
            \draw[thick][dashed][domain=-2.6*\r:-2.1*\r] plot (\x, \a*\x*\x);
            \draw[thick][dashed][domain=2.1*\r:2.6*\r] plot (\x, \a*\x*\x) node[right] {$h$};
            \node[above right] at (0,\r) {$r$};
            \node[above right] at (0,2*\r) {$2r$};
            \def\xz{0.7*\r*cos(160)}
            \def\yz{0.7*\r*sin(160)}
            \coordinate (Z) at ({\xz},{\yz});
            \def\t{-0.71*\r}
            \coordinate (Y) at (\t, \a*\t*\t);
            \filldraw[black] (Y) circle (1pt) node[below] {$y$};
            \filldraw[black] (Z) circle (1pt) node[above] {$z$};
            \def\m{2*\a*\t}
            \draw[domain=(\t-1.5*\r):(\t+3*\r)] plot (\x, {\a*\t*\t + \m*(\x-\t)})node[right,yshift=-2*\r]{\footnotesize $y+T_y\partial\Omega$};
            \draw[domain=(\t-1.5*\r):(\t+3*\r)] plot (\x, {\yz + \m*(\x-\xz)}) node[right]{\footnotesize $z+T_y\partial\Omega$};
            \draw[domain=(\t-1.5*\r):(\t+3.1*\r)] plot (\x, {\yz}) node[right]{\footnotesize $z+\mathscr{V}_0$};
        \end{tikzpicture}
        \caption{}
        \label{fig: eps-reg}
    \end{figure}

    Therefore, for some constant $C$ that we do not rename, 
    \begin{equation*}
        \beta_{z+T_y\partial\Omega}(z,r) + d(z,r) + r^\frac{1}{2} \leq C \eps(r)^\frac{2}{3}.
    \end{equation*}
    If we choose $\eps$ sufficiently small, we have, by Proposition \ref{prop: decay}, 
    \begin{equation}
        \label{eq: proof of eps-regularity, power decay for exc}
        \exc(z,\rho) \leq C \eps(r)^\frac{1}{18} \left(\frac{\rho}{r}\right)^\frac{1}{24} \quad \forall \rho < \frac{r}{4}, \, \forall z \in K \cap B_r
    \end{equation}
    and 
    \begin{equation}
        \label{eq: proof of eps-regularity, power decay for d and beta}
        d(z,\rho)+ \beta(z,\rho) \leq C \eps(r)^\frac{1}{18} \left(\frac{\rho}{r}\right)^\frac{1}{24} \quad \forall \rho < \frac{r}{4}, \, \forall z \in K \cap B_r,
    \end{equation}
    where $C$ is new constant that we do not rename, and from now on we stop renaming.

    Let $\mathscr{V}(z,\rho) \in \mathcal{L}$ be such that 
    \begin{equation*}
        \exc(z,\rho) = \exc_{\mathscr{V}(z,\rho)}(z,\rho) \quad \forall \rho < \frac{r}{4}, \, \forall z \in K \cap B_r.
    \end{equation*}
    Then, by the density lower bound (Theorem \ref{thm: density lower bound for K}), 
    \begin{equation*}
        |\mathscr{V}(z,\rho)- \mathscr{V}(z,t)|^2 \leq C\eps(r)^\frac{1}{18} \left(\frac{\rho}{r}\right)^\frac{1}{24} \quad \forall \rho < t < \frac{r}{8}, \, t \leq 2\rho, \, z \in K \cap B_r.
    \end{equation*}
    This, combined with the density lower bound (Theorem \ref{thm: density lower bound for K}), the Tilt Lemma (Lemma \ref{lmm: tilt lemma}), estimate \eqref{eq: proof of eps-regularity, flatness and energy bound on (0,2r)} and an elementary summation argument on dyadic scales, implies that 
    \begin{equation*}
        |\mathscr{V}(z,\rho)-\mathscr{V}_0|^2 \leq C \eps(r)^\frac{1}{18} \quad \forall z \in K \cap B_{\frac{r}{16}}, \, \forall \rho < \frac{r}{16}.
    \end{equation*}
    Then, by the energy upper bound (Proposition \ref{prop: Energy upper bound at the boundary}), we conclude that
    \begin{equation*}
        \sup_{0<\rho<\frac{r}{16}}\exc_{\mathscr{V}_0}(z,\rho) \leq C \eps(r)^\frac{1}{18} \quad \forall z \in K \cap B_{\frac{r}{16}}.
    \end{equation*}
    At this point, for sufficiently small $\eps$, by the previous estimate, by \eqref{eq: proof of eps-regularity, flatness and energy bound on (0,2r)} and by \eqref{eq: proof of eps-regularity, power decay for d and beta}, we have 
    \begin{equation*}
        K \cap B_{\sigma \frac{r}{8}} \subset \Gamma_f,
    \end{equation*}
    where $f: [-\sigma \frac{r}{8}, \sigma \frac{r}{8}] \to \R$ is given by Lemma \ref{lmm: Lipschitz approximation}\ref{item: Lipschitz approximation, D/r < 1}.

    Now we claim that
    \begin{equation*}
        \lim_{\rho \downarrow 0} \frac{\HH^1(K \cap B_\rho(z))}{2\rho} = 1 \quad \forall z \in K \cap B_{\frac{r}{16}}.
    \end{equation*}
    If $z \in \Omega$, this follows from \cite[(3.5.4) in Proof of Theorem 3.1.1]{delellis-focardi}. Suppose, instead, that $z \in \partial \Omega$. Since $\lim_{\rho \downarrow 0} d(z,\rho) = 0$ by \eqref{eq: proof of eps-regularity, power decay for d and beta}, by Theorem \ref{thm: classification of elementary blow-up limits}, every blow-up limit $(K_0, u_0)$ of $(K_{z,\rho}, u_{z,\rho})$ as $\rho \downarrow 0$ is such that $K_0$ is a line. The conclusion follows from Theorem \ref{thm: blow-up limit}.

    Then, arguing as in \cite[Proof of Theorem 3.1.1]{delellis-focardi}, we conclude that 
    \begin{equation*}
        K \cap B_{\sigma \frac{r}{16}} = \Gamma_f \cap B_{\sigma \frac{r}{16}}
    \end{equation*}
    and that every $s \in (-\sigma \frac{r}{16}, \sigma \frac{r}{16})$ is a Lebesgue point of $f'$.
    Combining the above estimates with \cite[8.9.5]{allard-interior} and the area formula \cite[Theorem 3.9]{evans-gariepy}, taking $\eps$ sufficiently small, we get 
    \begin{equation*}
        \frac{1}{\delta}\int_{s-\delta}^{s+\delta} |f'(t) - f'_{s,\delta}|^2 \, \dd{t} \lesssim  \exc((s,f(s)),\delta) \leq C \eps(r)^\frac{1}{18} \left(\frac{\delta}{r}\right)^\frac{1}{24}
    \end{equation*}
    for every $s \in (-\sigma \frac{r}{16}, \sigma \frac{r}{16})$ and $\delta \in (0,\sigma \frac{r}{16})$.
    This implies, by Campanato's criterion \cite[Theorem 6.1]{maggi}, that $f \in C^{1,\frac{1}{48}}(-\sigma \frac{r}{16}, \sigma \frac{r}{16})$, with $[f']_{\frac{1}{48}} \leq C \frac{\eps(r)^\frac{1}{36}}{r^\frac{1}{48}}$.  
\end{proof}
\begin{proof}[Proof of Lemma \ref{lmm: the Dirichlet term can be removed from eps-regularity statement}]
    Suppose by contradiction that the statement is false. Then there exists $\eps >0$ such that, for every $j \in \N$, there is \minimizer{K_j}{u_j}{\Omega}{g}, $x_j \in \partial \Omega$, $r_j>0$ such that 
    \begin{equation*}
        r_j^{-1} \dist_H(K_j \cap \overline{B}_{2r_j}(x_j), (x_j+T_{x_j} \partial \Omega) \cap \overline{B}_{2r_j}(x_j)) + r_j^\frac{1}{2}< \frac{1}{j^2}
    \end{equation*}
    and 
    \begin{equation*}
        \omega(x_j, \tfrac{1}{j}r_j) + (\tfrac{1}{j}r_j)^\frac{1}{2} \geq \eps.
    \end{equation*}

    Up to subsequences, $x_j \to x \in \partial \Omega$. Without loss of generality, assume that $x=0$ and $T_x\partial\Omega = \{x_2=0\}$.

    For every $j\in\N$, we define $K_j' = (K_j)_{x_j, \frac{1}{j}r_j}$, $u_j' = (u_j)_{x_j,\frac{1}{j}r_j}$, $\Omega_j = \Omega_{x_j, \frac{1}{j}r_j}$. Then, by Proposition \ref{prop: global convergence of functions, blow-up} and Theorem \ref{thm: blow-up limit}, $(K_j', u_j')$ converges, up to subsequences, to a generalized minimizer $(K_0, u_0)$ on $\Omega_0 = \{x_2 > 0\}$.

    Moreover, we have
    \begin{equation*}
        \dist_H(K_j' \cap B_{2j}, T_{x_j}\partial\Omega \cap B_{2j}) < \frac{1}{j},
    \end{equation*}
    and hence $K_0 = \{x_2 = 0\}$. 

    Then, since $K_0$ is connected, we conclude by Proposition \ref{prop: classification of connected generalized minimizers} that $\int |\nabla u_0|^2 = 0$.
    This is a contradiction, since
    \begin{equation*}
        \int_{B_2 \cap (\Omega_0 \setminus K_0)} |\nabla u_0|^2 = \lim_{j \to \infty} \int_{B_2 \cap (\Omega_j \setminus K_j')} |\nabla u_j'|^2 = \lim_{j \to \infty} \omega(x_j, \tfrac{1}{j}r_j) + (\tfrac{1}{j}r_j)^\frac{1}{2} \geq \eps.
    \end{equation*}
\end{proof}

As a corollary of the the above results, we are finally able to prove the main Theorem. 
\begin{proof}[Proof of Theorem \ref{thm: epsilon-regularity}]
    The result follows by combining Proposition \ref{prop: epsilon-regularity with Dirichlet term} and Lemma \ref{lmm: the Dirichlet term can be removed from eps-regularity statement}.
\end{proof}

\section{Decay results and proof of Proposition \ref{prop: decay}}
\label{section: decay results}
In this section, after stating the needed lemmas, we prove Proposition \ref{prop: decay}.

\subsection{Flatness improvement}
The following lemma, whose proof is postponed to Section \ref{section: proof of flatness improvement}, is of fundamental importance. Roughly speaking, it gives an improvement of the flatness passing from $B_r(x)$ to $B_{\tau r}(x)$. The result is delicate and is based on the Tilt Lemma (Lemma \ref{lmm: tilt lemma}). We distinguish two different regimes and carefully choose the line with respect to which we compute the flatness.
\begin{lmm}(Flatness improvement)
    \label{lmm: flatness improvement}
    Let $\overline{r}>0$ be as in Lemma \ref{lmm: tilt lemma}.
    There exists $\tau_1 \in (0,1)$ such that, for any $\tau \leq \tau_1$ and $\overline{\delta}>0$, there exists $\eta_1 = \eta_1(\tau,\overline{\delta})$, $\iota_1 = \iota_1(\tau,\overline{\delta})$ with the following property.
    Suppose that $(K,u)$, $x$, $r$, $y$, $D$, $\mathscr{V}$ satisfy property \eqref{eq: tilt lemma, K,u,x,r,D,V condition} and $\frac{D}{r} < \eta$, where $\eta$ is the constant of Lemma \ref{lmm: vertical separation}.

    \begin{enumerate}[label=(\roman*)]
        \item \label{item: flatness improvement, case D/r < b^1/2} If 
        \begin{equation*}
            \overline{\delta}\max\{d(x,r), r^\frac{1}{2}\} \leq \beta_{x+\mathscr{V}}(x,r) \leq \eta_1
        \end{equation*}
        and 
        \begin{equation}
            \label{eq: flatness case < b^1/2}
            0 \leq \frac{D}{r} \leq \iota_1 \beta_{x+\mathscr{V}}(x,r)^\frac{1}{2}
        \end{equation}
        then 
        \begin{equation*}
          \beta_{x+\mathscr{V}}(x,\tau r) \leq \tau \beta_{x+\mathscr{V}}(x,r).
        \end{equation*}
        \item \label{item: flatness improvement, case D/r > b^1/3} Let $A$ be the constant of the Tilt Lemma (Lemma \ref{lmm: tilt lemma}). 
        If conditions \eqref{eq: tilt lemma, case > b^1/3}, \eqref{eq: tilt lemma, angle condition} hold for some $\theta \in [0,2\pi)$, 
        \begin{equation*}
            \beta_{x+R_\theta(\mathscr{V})}(x,r) \leq \eta_1
        \end{equation*} 
        and
        \begin{equation*}
            \overline{\delta}\max\{d(x,r), r^\frac{1}{2}\} \leq \beta_{x+R_\theta(\mathscr{V})}(x,r),
        \end{equation*}
        then

        \begin{equation*}
            \min_{|\tan(\phi)|\leq b}\beta_{x+R_{\phi+\theta}(\mathscr{V})}(x,\tau r) \leq \tau \beta_{x+R_\theta(\mathscr{V})}(x,r),
        \end{equation*}
        where $R_\theta$ is the counterclockwise rotation of angle $\theta$, and 
        \begin{equation}
            \label{eq: flatness improvement, bound on tangent of angle}
            b = \sqrt{\frac{6(1+2\overline{\delta}^{-1})C}{\sigma}} \beta_{x+R_\theta(\mathscr{V})}(x,r)^\frac{1}{2},
        \end{equation}
        with $C$ and $\sigma$ being the constants of Lemma \ref{lmm: Lipschitz approximation}.
    \end{enumerate}
\end{lmm}
\subsection{Flatness and energy in the intermediate regime}
The following Lemma, whose proof is postponed to Section \ref{section: flatness in the unfavourable case}, gives a control on flatness and rescaled Dirichlet energy in what we call the intermediate regime, that is when condition \eqref{eq: flatness unfavourable case} is satisfied. 
Passing from $r$ to $\sigma r$, where $\sigma$ is given by the lemma, we end up in the third regime, which is given by condition \ref{item: intermediate regime, D/sigma > beta(sigma)^1/3} below. Moreover, for every $\sigma r \leq \rho \leq r$, the two quantities of interest do not grow too much. 
\begin{lmm}
    \label{lmm: flatness preserved in the unfavourable case}
    Let $\overline{r}$ be as in Lemma \ref{lmm: tilt lemma}. 
    For every $\overline{\delta}>0$ and $\iota_1>0$ there exist $\eta_3 = \eta_3(\overline{\delta}, \iota_1) \in (0,1)$ and $C = C(\overline{\delta}, \iota_1) > 0$ with the following property. 
    Suppose that $(K,u)$, $x$, $r$, $D$, $\mathscr{V}$ satisfy property \eqref{eq: tilt lemma, K,u,x,r,D,V condition}.
    Moreover, suppose that either 
    \begin{enumerate}[label=(\alph*), start=4]
        \item $\overline{\delta}\max\{d(x,r), r^\frac{1}{2}\} \leq \beta_{x+\mathscr{V}}(x,r) \leq \eta_3$
    \end{enumerate} 
    or 
    \begin{enumerate}[label=(\alph*), resume]
        \item $\overline{\delta}\max\{\beta_{x+\mathscr{V}}(x,r), r^\frac{1}{2}\} \leq d(x,r) \leq \eta_3$ and $d(x,r)^\frac{1}{2} < \frac{D}{r}$.
    \end{enumerate}
    Let $A$ be the constant of the Tilt Lemma (Lemma \ref{lmm: tilt lemma}). 
    If 
    \begin{equation}
        \label{eq: flatness unfavourable case}
        \iota_1 \beta_{x+\mathscr{V}}(x,r)^\frac{1}{2} < \frac{D}{r} < A \beta_{x+\mathscr{V}}(x,r)^\frac{1}{3}
    \end{equation}
    then there exists $\sigma>0$, depending on $\frac{D}{r}$, on $\overline{\delta}$ and on $\iota_1$, such that 
    \begin{enumerate}[label=(\roman*)]
        \item \label{item: intermediate regime, D/sigma > beta(sigma)^1/3} $\frac{D}{\sigma r} > A \beta_{x+\mathscr{V}}(x,\sigma r)^\frac{1}{3}$,
        \item \label{item: intermediate regime, beta estimate} $\beta_{x+\mathscr{V}}(x,\rho) \leq C \beta_{x+\mathscr{V}}(x,r)^\frac{1}{8} \left(\frac{\rho}{r}\right)^\frac{1}{8}$,
        \item \label{item: intermediate regime, d estimate} $d(x,\rho) \leq C \beta_{x+\mathscr{V}}(x,r)^\frac{1}{8} \left(\frac{\rho}{r}\right)^\frac{1}{8}$.
    \end{enumerate}
    for any $\sigma r \leq \rho \leq r$.
\end{lmm}

\subsection{Rescaled Dirichlet energy decay}
The following lemma, whose proof is postponed to Section \ref{section: rescaled dirichlet energy decay}, gives a decay of the rescaled Dirichlet energy passing from $B_r(x)$ to $B_{\tau r}(x)$.
\begin{lmm}[Rescaled Dirichlet energy decay] 
    \label{lmm: rescaled dirichlet energy decay}
    Let $\overline{r}>0$ be as in Lemma \ref{lmm: tilt lemma}. 
    There exists $\tau_2 \in (0,1)$ such that, for any $\tau \leq \tau_2$ and $\overline{\delta}>0$, there exist $\eta_2 = \eta_2(\tau, \overline{\delta})$ and $\iota_2 = \iota_2(\tau,\overline{\delta})$ with the following property.  
    Suppose that $(K,u)$, $x$, $r$, $D$, $\mathscr{V}$ satisfy property \eqref{eq: tilt lemma, K,u,x,r,D,V condition} and 
    \begin{equation*}
        \frac{D}{r} \leq \iota_2.
    \end{equation*}
    Let $A$ and $B$ be the constants of the Tilt lemma (Lemma \ref{lmm: tilt lemma}). Then 
    \begin{enumerate}[label=(\roman*)]
        \item \label{item: Dirichlet, case 1} If 
        \begin{equation*}
            \overline{\delta}\max\{\beta_{x+\mathscr{V}}(x,r), r^\frac{1}{2}\} \leq d(x,r) \leq \eta_2
        \end{equation*}
        and 
        \begin{equation*}
            0 \leq \frac{D}{r} \leq A \beta_{x+\mathscr{V}}(x,r)^\frac{1}{2},
        \end{equation*}
        then 
        \begin{equation*}
            d(x,\tau r) \leq \tau^\frac{1}{2} d(x,r).
        \end{equation*}

        \item \label{item: Dirichlet, case 2}    If 
        \begin{equation*}
            0\leq \frac{D}{r} \leq A \beta_{x+\mathscr{V}}(x,r)^\frac{1}{3},
        \end{equation*} 
        \begin{equation*}
            \frac{D}{r} \leq d(x,r)^\frac{1}{2},
        \end{equation*}
        and
        \begin{equation*}
            \overline{\delta}\max\{\beta_{x+\mathscr{V}}(x,r), r^\frac{1}{2}\} \leq d(x,r) \leq \eta_2
        \end{equation*}
        then 
        \begin{equation*}
            d(x,\tau r) \leq \tau^\frac{1}{2} d(x,r).
        \end{equation*}
        \item  \label{item: Dirichlet, case 3}   If
        \begin{equation*}
            A \beta_{x+\mathscr{V}}(x,r)^\frac{1}{3} < \frac{D}{r} \leq B r,
        \end{equation*}
        and 
        \begin{equation*}
            \overline{\delta}\max\{\beta_{x+\mathscr{V}}(x,r), r^\frac{1}{2}\} \leq d(x,r) \leq \eta_2
        \end{equation*}
        then 
        \begin{equation*}
            d(x,\tau r) \leq \tau^\frac{1}{2} d(x,r).
        \end{equation*}
        \item  \label{item: Dirichlet, case 4}    If
        \begin{equation*}
            A \beta_{x+R_\theta(\mathscr{V})}(x,r)^\frac{1}{3} < \frac{D}{r},
        \end{equation*}
        \eqref{eq: tilt lemma, condition on the radius to have curvature small} holds, \eqref{eq: tilt lemma, angle condition} holds for some $\theta \in [0,2\pi)$ and 
        \begin{equation*}
            \overline{\delta}\max\{\beta_{x+R_\theta(\mathscr{V})}(x,r), r^\frac{1}{2}\} \leq d(x,r) \leq \eta_2,
        \end{equation*}
        then 
        \begin{equation*}
            d(x,\tau r) \leq \tau^\frac{1}{2} d(x,r).
        \end{equation*}
    \end{enumerate}
\end{lmm}

\subsection{Proof of Proposition \ref{prop: decay}}
\label{section: proof of decay proposition}
Using the above lemmas, we finally prove Proposition \ref{prop: decay}.
\begin{proof}[Proof of Proposition \ref{prop: decay}]
    Let $\eps_i, \tau_i > 0$ be given by \cite[Proposition 3.2.1]{delellis-focardi}. Let $\delta \leq \eps_i$.
    
    Assume that $D < r$, otherwise the result follows from \cite[Proposition 3.2.1]{delellis-focardi}.

    Let $\tau_1$ and $\tau_2$ be the constants of Lemma \ref{lmm: flatness improvement} and Lemma \ref{lmm: rescaled dirichlet energy decay}. 
    Let $\tau = \min\{\tau_1, \tau_2\}$.
    Let $\overline{\delta} \leq \tau^\frac{7}{2}$. 
    Let $\eta_1 = \eta_1(\tau,\overline{\delta}), \iota_1 = \iota_1(\tau,\overline{\delta}), \eta_2 = \eta_2(\tau, \overline{\delta}), \iota_2 = \iota_2(\tau,\overline{\delta}), $ be given by Lemma \ref{lmm: flatness improvement} and Lemma \ref{lmm: rescaled dirichlet energy decay}. Let $\eta_3 = \eta_3(\overline{\delta}, \iota_1)$ and $C = C(\overline{\delta}, \iota_1)$ be given by Lemma \ref{lmm: flatness preserved in the unfavourable case}. Let $C' = C'(\overline{\delta})$ be the constant appearing in front of $\beta$ in \eqref{eq: flatness improvement, bound on tangent of angle}.
    Let $A, B, E$ be the constants of Lemma \ref{lmm: tilt lemma}.
    Let 
    \begin{multline*}
        \textstyle
        \delta^\frac{1}{8} := \frac{1}{C}\min\bigg\{\iota_2^3 \eps_i , \eta_1, \eta_2, \eta_3, \left(\frac{9 E A (1-\tau^\frac{1}{2}) }{10 C'}\right)^6,
        \left(\frac{9(1-\tau^\frac{1}{2})}{100 C'}\right)^2, \\ \textstyle \left(\frac{16(\sqrt{10}-3)}{9\sqrt{3}C'}\right)^2, \left(\frac{3\pi (1-\tau^\frac{1}{2})}{4 \sqrt{10} C'}\right)^2 ,
        \left(\frac{(\sqrt{10}-3)(1-\tau^\frac{1}{2})}{2\sqrt{10}C'}\right)^2\bigg\}.
    \end{multline*}
    Define 
    \begin{equation*}
        m_\theta (x,t) := \max\{d(x,t), \beta_{x+R_\theta(\mathscr{V})}(x,t), t^\frac{1}{2}\} \quad \forall t>0, \, \forall \theta \in [0,2\pi),
    \end{equation*} 
    where $R_\theta$ is the counterclockwise rotation of angle $\theta$.

    \begin{enumerate}[leftmargin=*]
        \item Let us prove \ref{item: decay lemma, D=0}. 
        In this case, arguing as in \cite[Proof of Proposition 3.2.1]{delellis-focardi} and iterating, we get 
        \begin{equation*}
            m_0(x,\tau^k r) \leq \tau^\frac{k}{2} m_0(x, r) < \delta \tau^\frac{k}{2} \quad \forall k \in \N.
        \end{equation*}
        Let $0<\rho<r$. then there exists $k \in \N$ such that $\tau^k r < \rho \leq \tau^{k-1} r$.
        Then, by the Tilt Lemma (Lemma \ref{lmm: tilt lemma}),
        \begin{equation*}
            \exc_{\mathscr{V}}(x,\tfrac{\rho}{4}) \leq \tau^{-1}  \exc_{\mathscr{V}}(x,\tfrac{\tau^{k-1} r}{4}) \leq C'' \delta \tau^{-1} \tau^\frac{k-1}{2} \leq C'' \delta \tau^{-\frac{3}{2}} \left(\frac{\rho}{r}\right)^\frac{1}{2}.
        \end{equation*}
        \item Let us prove \ref{item: decay lemma, D>0}.
        Assume that \eqref{eq: flatness case < b^1/2} holds, otherwise the proof is easier.
        Then, arguing as in \cite[Proof of Proposition 3.2.1]{delellis-focardi}, we get $m_0(x,\tau r) \leq \tau^\frac{1}{2} m_0(x, r)$.
        Iterating, we get 
        \begin{equation*}
            \beta_{x+\mathscr{V}}(x,\tau^k r) \leq m_0(x,\tau^k r) \leq \tau^\frac{k}{2} m_0(x, r) 
        \end{equation*}
        for every $k \in \N$ such that 
        \begin{equation*}
            \frac{D}{\tau^{k-1}r} \leq \iota_1(\tau,\overline{\delta}) \beta_{x+\mathscr{V}}(x,\tau^{k-1}r)^\frac{1}{2}.
        \end{equation*}
    
        Let $k_0 \in \N$ be large enough so that the above inequality is not satisfied. We set $r_0 = \tau^{k_0-1} r $.
        Suppose that condition \eqref{eq: flatness unfavourable case} is satisfied with $r_0$ in place of $r$. We have to distinguish different cases. 
        \begin{enumerate}[\text{Case} 1:]
            \item $\overline{\delta}\max\{d(x,r_0), r_0^\frac{1}{2}\} \leq \beta_{x+\mathscr{V}}(x,r_0)$. 
    
            In this case, by Lemma \ref{lmm: flatness preserved in the unfavourable case}, there exists $\sigma$, depending on $\frac{D}{r_0}$, on $\overline{\delta}$ and on $\iota_1$, such that $\frac{D}{\sigma r_0} > A \beta_{x+\mathscr{V}}(x,\sigma r_0)^\frac{1}{3}$, $\beta_{x+\mathscr{V}}(x,\rho) \leq C \beta_{x+\mathscr{V}}(x,r_0)^\frac{1}{8} \left(\frac{\rho}{r_0}\right)^\frac{1}{2}$, and $d(x,\rho) \leq C \beta_{x+\mathscr{V}}(x,r_0)^\frac{1}{8} \left(\frac{\rho}{r_0}\right)^\frac{1}{2}$
            for any $\sigma r_0 \leq \rho \leq r_0$.
            In particular, 
            \begin{equation*}
                m_0(x,\sigma r_0) \leq C m_0(x,r_0)^\frac{1}{8}.
            \end{equation*}
            \item $\overline{\delta}\max\{d(x,r_0), r_0^\frac{1}{2}\} > \beta_{x+\mathscr{V}}(x,r_0)$, $\overline{\delta}\max\{\beta_{x+\mathscr{V}}(x,r_0), r_0^\frac{1}{2}\} \leq d(x,r_0)$ and $d(x,r_0)^\frac{1}{2} < \frac{D}{r_0}$.
    
            This case gives the same result as Case 1, by Lemma \ref{lmm: flatness preserved in the unfavourable case}.
            \item $\overline{\delta}\max\{d(x,r_0), r_0^\frac{1}{2}\} > \beta_{x+\mathscr{V}}(x,r_0)$, $\overline{\delta}\max\{\beta_{x+\mathscr{V}}(x,r_0), r_0^\frac{1}{2}\} \leq d(x,r_0)$ and $d(x,r_0)^\frac{1}{2} \geq \frac{D}{r_0}$.
    
            In this case, by Lemma \ref{lmm: rescaled dirichlet energy decay}, $d(x,\tau r_0) \leq \tau^\frac{1}{2} d(x,r_0)$.
            Moreover, 
            \begin{equation*}
                \beta_{x+\mathscr{V}}(x,\tau r_0) \leq \tau^{-3} \beta_{x+\mathscr{V}}(x,r_0) \leq \tau^\frac{1}{2} \max\{d(x,r_0), r_0^\frac{1}{2}\}.
            \end{equation*}
            Therefore $m_0(x,\tau r_0) \leq \tau^\frac{1}{2} m_0(x,r_0).$
            \item $\overline{\delta}\max\{d(x,r_0), r_0^\frac{1}{2}\} > \beta_{x+\mathscr{V}}(x,r_0)$ and $\overline{\delta}\max\{\beta_{x+\mathscr{V}}(x,r_0), r_0^\frac{1}{2}\} > d(x,r_0)$.
            Since $\overline{\delta} < 1$, we have $m_0(x,r_0) = r_0^\frac{1}{2}$.
            Moreover, $\beta_{x+\mathscr{V}}(x,\tau r_0) \leq \tau^\frac{1}{2} r_0^\frac{1}{2}$ and  $d(x,\tau r_0) \leq \tau^\frac{1}{2} r_0^\frac{1}{2}$, whence $m_0(x,\tau r_0) \leq \tau^\frac{1}{2} m_0(x,r_0)$.
        \end{enumerate}

        If we are in Case 1 or Case 2, then we can continue the proof. If not, we continue to iterate and get 
        \begin{equation*}
            m_0(x,\tau^k r_0) \leq \tau^\frac{k}{2} m_0(x,r_0)
        \end{equation*}
        for $k \in \N$ until either 
        \begin{equation*}
            \frac{D}{\tau^k r_0} > A \beta_{x+\mathscr{V}}(x,\tau^k r_0)^\frac{1}{3}
        \end{equation*}
        or we fall into Case 1 or Case 2, with $r_0$ replaced by $\tau^k r_0$. 
    
        At the end of this procedure, which terminates in a finite number of steps, we have $0<r_2 < r_0$ such that
        \begin{equation*}
            \frac{D}{r_2} > A \beta_{x+\mathscr{V}}(x,r_2)^\frac{1}{3}, \quad  m_0(x,r_2) \leq C m_0(x,r_0)^\frac{1}{8}.
        \end{equation*}
        
        If we never fell into Case 1 or Case 2, then we define $r_1 = r_2$ and we have, arguing as in the proof of \ref{item: decay lemma, D=0}, for any $r_1 \leq \rho < r$, 
        \begin{equation*}
            \exc_{\mathscr{V}}(x,\tfrac{\rho}{4}) \leq C'' \delta \tau^{-\frac{3}{2}} \left(\frac{\rho}{r}\right)^\frac{1}{2}.
        \end{equation*}
        Otherwise, we define $r_1 = \sigma^{-1} r_2$, where $\sigma$ is the constant of Case 1.
        The above inequality holds for $r_1 \leq \rho < r$. Furthermore, for any $r_2 \leq \rho < r_1$,
        \begin{equation*}
        m_0(x,\rho) \leq C \beta_{x+\mathscr{V}}(x,r_1)^\frac{1}{8} \left(\frac{\rho}{r_1}\right)^\frac{1}{2}.
        \end{equation*}
        Observe that $r_1 = \tau^{k_1} r$ for some $k_1 \in N$ and 
        \begin{equation*}
            \beta_{x+\mathscr{V}}(x,r_1) \leq \tau^{-3} \beta_{x+\mathscr{V}}(x,\tau^{k_1-1}r) \leq \delta \tau^{-3} \tau^\frac{k_1-1}{2} \leq \delta \tau^{-\frac{7}{2}} \left(\frac{r_1}{r}\right)^\frac{1}{2}
        \end{equation*}
        whence 
        \begin{equation}
            \label{eq: decay lmm, intermediate regime, m_0(rho) < (rho/r)^1/16}
            m_0(x,\rho) \leq C \delta^\frac{1}{8} \left(\frac{r_1}{r}\right)^\frac{1}{16} \left(\frac{\rho}{r_1}\right)^\frac{1}{16} = C \delta^\frac{1}{8} \left(\frac{\rho}{r}\right)^\frac{1}{16}.
        \end{equation}
        
        We claim that 
        \begin{equation*}
            \exc_\mathscr{V}(x,\tfrac{\rho}{4}) \leq C'' \delta^\frac{1}{12} \left(\frac{\rho}{r}\right)^\frac{1}{24}.
        \end{equation*}
        If $\frac{D}{\rho} \leq A \beta_{x+\mathscr{V}}(x,\rho)^\frac{1}{3}$, it follows from Lemma \ref{lmm: tilt lemma}(i). If $\frac{D}{\rho} > A \beta_{x+\mathscr{V}}(x,\rho)^\frac{1}{3}$ and $\frac{D}{\rho} > B \rho$, it follows from Lemma \ref{lmm: tilt lemma}(ii). Finally, if $\frac{D}{\rho} > A \beta_{x+\mathscr{V}}(x,\rho)^\frac{1}{3}$ and $\frac{D}{\rho} \leq B \rho$, it follows from Lemma \ref{lmm: tilt lemma}(i).

        Now suppose that $\frac{D}{r_2} < \iota_2$, otherwise the proof is simpler.
        \begin{claim} 
            \label{claim: decay lmm, rotation decay}
            Define $\theta_0 := 0$. 
            For every $k \in \N$ such that 
            \begin{equation*}
                \frac{D}{\tau^{k-1} r_2} \leq \iota_2
            \end{equation*}
            there exists $\theta_k \in (-\frac{\pi}{2}, \frac{\pi}{2})$ such that $\theta_k - \theta_{k-1} \in (-\frac{\pi}{2}, \frac{\pi}{2})$ and
            \begin{align*}
                &\frac{D}{\tau^k r_2} > A \beta_{x+R_{\theta_k}(\mathscr{V})}(x,\tau^k r_2)^\frac{1}{3}\\
                &|\tan(\theta_k - \theta_{k-1})| \leq C' \beta_{x+R_{\theta_{k-1}}(\mathscr{V})}(x,\tau^{k-1}r_2)^\frac{1}{2} \\
                &|\tan\theta_k| \leq \frac{10C'}{9(1-\tau^\frac{1}{2})} \beta_{x+\mathscr{V}}(x,r_2)^\frac{1}{2}\\
                &m_{\theta_k}(x,\tau^k r_2) \leq \tau^\frac{1}{2} m_{\theta_{k-1}}(x,\tau^{k-1}r_2)
            \end{align*}
        \end{claim}
        \begin{proof}[Proof of Claim] 
            Let $\theta_0 = 0$. We proceed by induction. First we prove the claim for $k=1$. 
            We have to distinguish different cases. 
            \begin{enumerate}[\text{Case} 1:]
                \item $\overline{\delta}\max\{d(x,r_2), r_2^\frac{1}{2}\} \leq \beta_{x+\mathscr{V}}(x,r_2)$ and $\overline{\delta}\max\{\beta_{x+\mathscr{V}}(x,r_2), r^\frac{1}{2}\} \leq d(x,r_2)$.
            
                In this case, by Lemma \ref{lmm: flatness improvement} there exists $\theta_1 \in (-\frac{\pi}{2}, \frac{\pi}{2})$ such that 
                \begin{equation*}
                    |\tan\theta_1| \leq C' \beta_{x+\mathscr{V}}(x,r_2)^\frac{1}{2}, \quad \beta_{x+R_{\theta_1}(\mathscr{V})}(x,\tau r_2) \leq \tau \beta_{x+\mathscr{V}}(x,r_2).
                \end{equation*}
                Moreover, by Lemma \ref{lmm: rescaled dirichlet energy decay}, $d(x,\tau r_2) \leq \tau^\frac{1}{2} d(x,r_2)$.
                Therefore 
                \begin{equation*}
                    m_{\theta_1}(x,\tau r_2) \leq \tau^\frac{1}{2} m_0(x,r_2).
                \end{equation*}
                \item $\overline{\delta}\max\{d(x,r_2), r_2^\frac{1}{2}\} \leq \beta_{x+\mathscr{V}}(x,r_2)$ and $\overline{\delta}\max\{\beta_{x+\mathscr{V}}(x,r_2), r^\frac{1}{2}\} > d(x,r_2)$.
            
                In this case, by Lemma \ref{lmm: flatness improvement} and by the choice of $\overline{\delta}$, the conclusion is the same of Case 1.
                \item $\overline{\delta}\max\{d(x,r_2), r_2^\frac{1}{2}\} > \beta_{x+\mathscr{V}}(x,r_2)$ and $\overline{\delta}\max\{\beta_{x+\mathscr{V}}(x,r_2), r^\frac{1}{2}\} \leq d(x,r_2)$. 
                In this case we set $\theta_1=0$ and, by Lemma \ref{lmm: rescaled dirichlet energy decay} and by the choice of $\overline{\delta}$, we have $m_0(x, \tau r_2) \leq \tau^\frac{1}{2} m_0(x, r_2)$.
                \item $\overline{\delta}\max\{d(x,r_2), r_2^\frac{1}{2}\} > \beta_{x+\mathscr{V}}(x,r_2)$ and $\overline{\delta}\max\{\beta_{x+\mathscr{V}}(x,r_2), r^\frac{1}{2}\} > d(x,r_2)$.

                In this case we set $\theta_1=0$ and we have $m_0(x,r_2) = r_2^\frac{1}{2}$ and, by the choice of $\overline{\delta}$, $m_0(x, \tau r_2) \leq \tau^\frac{1}{2} m_0(x, r_2)$.
            \end{enumerate}
 
            Now we assume that the claim is true for every $i \in \{1,\ldots, k-1\}$ and we show that it holds for $k$. 
            Observe that, by the choice of $\delta$, we have 
            \begin{equation*}
                |\sin\theta_{k-1}| \leq |\tan\theta_{k-1}| \leq E \frac{D}{r_2} < E \frac{D}{\tau^{k-1} r_2}.
            \end{equation*}
            We have to distinguish different cases. 
            \begin{enumerate}[\text{Case} 1:]
                \item             \begin{align*}
                    &\overline{\delta}\max\{d(x,\tau^{k-1} r_2), (\tau^{k-1} r_2)^\frac{1}{2}\} \leq \beta_{x+R_{\theta_{k-1}}(\mathscr{V})}(x,\tau^{k-1} r_2), \\ 
                    &\overline{\delta}\max\{\beta_{x+R_{\theta_{k-1}}(\mathscr{V})}(x,\tau^{k-1} r_2), (\tau^{k-1} r_2)^\frac{1}{2}\} \leq d(x,\tau^{k-1} r_2).
                \end{align*}
                
                In this case, by Lemma \ref{lmm: flatness improvement} and Lemma \ref{lmm: rescaled dirichlet energy decay}, there exists $\phi \in (-\frac{\pi}{2}, \frac{\pi}{2})$ such that, defining $\theta_k := \phi + \theta_{k-1}$,
                \begin{align*}
                    &|\tan(\theta_k - \theta_{k-1})| \leq C' \beta_{x+R_{\theta_{k-1}}(\mathscr{V})}(x,\tau^{k-1} r_2)^\frac{1}{2}\\
                    &\beta_{x+R_{\theta_k}(\mathscr{V})}(x,\tau^k r_2) \leq \tau \beta_{x+R_{\theta_{k-1}}(\mathscr{V})}(x,\tau^{k-1} r_2), \\
                    &d(x,\tau^k r_2) \leq \tau^\frac{1}{2} d(x,\tau^{k-1} r_2).
                 \end{align*}
                 In particular, 
                 \begin{equation*}
                    m_{\theta_k}(x,\tau^k r_2) \leq \tau^\frac{1}{2} m_{\theta_{k-1}}(x,\tau^{k-1} r_2).
                 \end{equation*}
        
                 Using the fact that $|\arctan(t) - t| \leq \frac{3\sqrt{3}}{16} t^2$ for every $t \in \R$ and the choice of $\delta$, for every $i \in \{1, \ldots, k\}$ we get 
                 \begin{equation*}
                    \arctan(C' \beta_{x+R_{\theta_{i-1}(\mathscr{V})}}(x, \tau^{i-1}r_2)^\frac{1}{2}) \leq \frac{\sqrt{10}}{3} C' \beta_{x+R_{\theta_{i-1}(\mathscr{V})}}(x, \tau^{i-1}r_2)^\frac{1}{2}.
                 \end{equation*}
                 Then, by the inductive hypothesis,
                 \begin{equation*}
                    |\theta_k| \leq \frac{\sqrt{10}}{3} C' \sum_{\myatop{i = 1}{\theta_i \neq \theta_{i-1}}}^{k} \beta_{x+R_{\theta_{i-1}(\mathscr{V})}}(x, \tau^{i-1}r_2)^\frac{1}{2} \leq  \frac{\sqrt{10} C'}{3(1-\tau^\frac{1}{2})} \beta_{x+\mathscr{V}}(x,r_2)^\frac{1}{2}.
                 \end{equation*}
                 Moreover, by the choice of $\delta$, $|\theta_k| < \frac{\pi}{4}$, and hence $\tan|\theta_k| \leq |\theta_k|(1+2|\theta_k|)$.
                 Therefore, by the choice of $\delta$, 
                 \begin{equation*}
                    |\tan\theta_k| \leq \frac{10C'}{9(1-\tau^\frac{1}{2})} \beta_{x+\mathscr{V}}(x,r_2)^\frac{1}{2}
                 \end{equation*}
                 \item  \begin{align*}
                    &\overline{\delta}\max\{d(x,\tau^{k-1} r_2), (\tau^{k-1} r_2)^\frac{1}{2}\} \leq \beta_{x+R_{\theta_{k-1}}(\mathscr{V})}(x,\tau^{k-1} r_2), \\ 
                    &\overline{\delta}\max\{\beta_{x+R_{\theta_{k-1}}(\mathscr{V})}(x,\tau^{k-1} r_2), (\tau^{k-1} r_2)^\frac{1}{2}\} > d(x,\tau^{k-1} r_2).
                \end{align*}
            
                In this case, the result is the same of Case 1.
                \item \begin{align*}
                    &\overline{\delta}\max\{d(x,\tau^{k-1} r_2), (\tau^{k-1} r_2)^\frac{1}{2}\} > \beta_{x+R_{\theta_{k-1}}(\mathscr{V})}(x,\tau^{k-1} r_2), \\ 
                    &\overline{\delta}\max\{\beta_{x+R_{\theta_{k-1}}(\mathscr{V})}(x,\tau^{k-1} r_2), (\tau^{k-1} r_2)^\frac{1}{2}\} \leq d(x,\tau^{k-1} r_2).
                \end{align*}
            
                In this case we set $\theta_k=\theta_{k-1}$ and, by Lemma \ref{lmm: rescaled dirichlet energy decay} and by the choice of $\overline{\delta}$, we have
                \begin{equation*}
                    m_{\theta_k}(x,\tau^k r_2) \leq \tau^\frac{1}{2} m_{\theta_{k-1}}(x,\tau^{k-1} r_2).
                 \end{equation*}
                \item  \begin{align*}
                    &\overline{\delta}\max\{d(x,\tau^{k-1} r_2), (\tau^{k-1} r_2)^\frac{1}{2}\} > \beta_{x+R_{\theta_{k-1}}(\mathscr{V})}(x,\tau^{k-1} r_2), \\ 
                    &\overline{\delta}\max\{\beta_{x+R_{\theta_{k-1}}(\mathscr{V})}(x,\tau^{k-1} r_2), (\tau^{k-1} r_2)^\frac{1}{2}\} > d(x,\tau^{k-1} r_2).
                \end{align*}
                In this case we set $\theta_k=\theta_{k-1}$ and, arguing as above and by the choice of $\overline{\delta}$,
                \begin{equation*}
                    m_{\theta_k}(x,\tau^k r_2) \leq \tau^\frac{1}{2} m_{\theta_{k-1}}(x,\tau^{k-1} r_2).
                 \end{equation*}
            \end{enumerate}
        \end{proof}

        Let $N$ be the smallest natural number such that
        \begin{equation*}
            \frac{D}{\tau^{N-1}r_2} > \iota_2
        \end{equation*}
        and let 
        \begin{equation*}
            r_3 := \tau^{N-1} r_2.
        \end{equation*}

        Let $r_3 \leq \rho < r_2$. Then there is $k \in \N$ such that $\tau^k r_2 \leq \rho < \tau^{k-1}r_2$. Observe that, by Claim \ref{claim: decay lmm, rotation decay},
        $m_{\theta_j}(x,\tau^j r_2) \leq \tau^\frac{j}{2} m_0(x,r_2)$ for every $j \in \{1, \ldots, N-1\}$.
        Therefore, 
        \begin{equation*}
            m_{\theta_{k-1}}(x,\rho) \leq \tau^{-3} \tau^\frac{k-1}{2} m_0(x,r_2) \leq \tau^{-\frac{7}{2}} m_0(x,r_2) \left(\frac{\rho}{r_2}\right)^\frac{1}{2}
        \end{equation*}
        and, by \eqref{eq: decay lmm, intermediate regime, m_0(rho) < (rho/r)^1/16}, 
        \begin{equation*}
            m_{\theta_{k-1}}(x,\rho) \leq C \delta^\frac{1}{8} \tau^{-\frac{7}{2}} \left(\frac{r_2}{r}\right)^\frac{1}{16} \left(\frac{\rho}{r_2}\right)^\frac{1}{16} = \tau^{-\frac{7}{2}} C \delta^\frac{1}{8} \left(\frac{\rho}{r}\right)^\frac{1}{16}.
        \end{equation*}
        Then, by Lemma \ref{lmm: tilt lemma},
        \begin{equation*}
            \exc(x,\tfrac{\rho}{4}) \leq C'' \delta^\frac{1}{8} \left(\frac{\rho}{r}\right)^\frac{1}{16}.
        \end{equation*}
        We define 
        \begin{equation*}
            r_i:= \iota_2 r_3.
        \end{equation*}
        Let $r_i \leq \rho < r_3$. Then 
        \begin{equation*}
            m_{\theta_{N-1}}(x,\rho) \leq \iota_2^{-3}  m_{\theta_{N-1}}(x,r_3) \leq \iota_2^{-3} C \delta^\frac{1}{8} \left(\frac{r_3}{r}\right)^\frac{1}{16} \leq C'' \delta^\frac{1}{8} \left(\frac{\rho}{r}\right)^\frac{1}{16}
        \end{equation*}
        where we do not rename the constant $C''$.

        Observe that $r_i < D$ and 
        \begin{equation*}
            m_i(x,r_i):= \max\{d(x,r_i), \beta(x,r_i)\} \leq \iota_2^{-3} C m_0(x,r)^\frac{1}{8} < \eps_i.
        \end{equation*}
        Therefore, by \cite[Proposition 3.2.1]{delellis-focardi}, 
        \begin{equation*}
            m_i(x,\tau_i^k r_i) \leq \tau_i^\frac{k}{2} m_i(x,r_i) \quad \forall k \in \N.
        \end{equation*}
        Let $0<\rho<r_i$. Then there is $k \in \N$ such that $\tau_i^k r_i < \rho \leq \tau_i^{k-1}r_i$ and 
        \begin{equation*}
            m_i(x,\rho) \leq \tau_i^{-\frac{7}{2}} \tau_i^\frac{k}{2} m_i(x,r_i) \leq \tau_i^{-\frac{7}{2}} m_i(x,r_i) \left(\frac{\rho}{r_i}\right)^\frac{1}{2} \leq C'' \delta^\frac{1}{8} \left(\frac{\rho}{r}\right)^\frac{1}{16}.
        \end{equation*}
    \end{enumerate}
\end{proof}
\section{Proof of technical lemmas}
In this section we prove Lemma \ref{lmm: tilt lemma}, \ref{lmm: Lipschitz approximation}, \ref{lmm: flatness improvement}, \ref{lmm: flatness preserved in the unfavourable case}, \ref{lmm: rescaled dirichlet energy decay}.
\subsection{Proof of Lemma \ref{lmm: tilt lemma}}
\label{section: proof of tilt lemma}
\begin{proof}[Proof of Lemma \ref{lmm: tilt lemma}]
    Let $\eps$ be the constant of the density lower bound (Theorem \ref{thm: density lower bound for K}). Define $A:= \frac{40}{\eps^\frac{1}{3}}$ and 
    let $2B'$ the constant associated to radius $2\overline{r}$, given by Lemma \ref{lmm: radius for boundary}.
    \begin{enumerate}[leftmargin=*]
        \item Let us prove \ref{item: tilt lemma, case D/r < 1}. 
        Up to rotations and translations, we can assume that $\mathscr{V} = \mathscr{V}_0 := \{(x_1,x_2): x_2 = 0\}$ and $x=0$.
        By the choice of $\overline{r}$ there exists $h \in C^2(\R)$, with  $h(0)=h'(0)=0$, such that
       $
            B_r \cap \Omega = \{(x_1,x_2) \in B_r: x_2 > h(x_1) - D\}.
        $
        Then it suffices to show that
        \begin{equation*}
            \frac{1}{r} \int_{K \cap B_\frac{r}{3}} e_2^2 \, \dd{\HH^1} \leq C \left(d(\tfrac{2r}{3}) + \frac{1}{r^3} \int_{K \cap B_\frac{2r}{3}} |x_2|^2 \, \dd{\HH^1} +\left(\frac{D}{r}\right)^2+r\right), 
        \end{equation*}
        where $e: K \to S^1$ is a tangent vector field to $K$ with $e(x) = (e_1(x), e_2(x))$. 
        Adapting the Proof of \cite[Lemma 3.3.2]{delellis-focardi}, define 
        \begin{equation*}
            \eta(x) = \varphi^2(x)(0,x_2-h(x_1)+D),
        \end{equation*}
       where $\varphi \in C^\infty_c(B_\frac{2r}{3}, [0,1])$, $\varphi \equiv 1$ on $B_\frac{r}{3}$, $\norm{\nabla \varphi}_{L^\infty(B_\frac{2r}{3})} \leq \frac{6}{r}$.
        
        Define $\Phi_\eps(x) = x + \eps \eta(x)$ for every $x \in \R^2$ and $\eps \in \R$.
        Observe that $\Phi_\eps |_{(\partial B_r \cap \Omega)\cup (\partial \Omega \cap B_r)}$ is the identity and, for sufficiently small $\abs{\eps}$, $\Phi_\eps$ is a diffeomorphism of $B_r \cap \Omega$ onto itself. 
    
        Define $K_\eps = \Phi_\eps(K)$ and $u_\eps = u \circ \Phi_\eps^{-1}$.
        Then $(K_\eps, u_\eps) \in \mathcal{A}(\Omega,g)$ for sufficiently small $\abs{\eps}$.
        Since $(K,u) \in \argmin\{E(J,v): (J,v) \in \mathcal{A}(\Omega,g)\}$, we have 
        $\derivative{}{\eps}E(K_\eps,u_\eps)|_{\eps = 0} = 0$. Lemma \ref{lmm: first derivative of the energy for generalized minimizers} then implies  
        \begin{equation*}
            \int_{B_r \cap \Omega \setminus K} (|\nabla u|^2 \diverg \eta - 2 \nabla u^T D\eta \, \nabla u) + \int_{K \cap B_r} e^T D\eta \, e \, \dd{\HH^1} = 0.
        \end{equation*}
        Computing $D\eta$ we get 
        \begin{equation*}
            \varphi^2 e_2^2 = e^T D\eta e - 2 \varphi (x_2 - h(x_1)+D) e_2 \nabla \varphi \cdot e + \varphi^2 h'(x_1) e_1e_2.
        \end{equation*}
        Then, by Cauchy's inequality \cite[B.2.a]{evans}, 
        \begin{equation*}
            \varphi^2e_2^2 \leq 2e^T D\eta e + 4 (x_2-h(x_1)+D)^2 |\nabla \varphi|^2 + 2\varphi^2 |h'(x_1)|.
        \end{equation*}
        The conclusion then follows by integrating the above on $K \cap B_r$ and taking into account the fact that 
        $|h(x_1)|\leq B' x_1^2$ and $|h'(x_1)|\leq 2B'|x_1|$ by the choice of $\overline{r}$ and $B'$.
    \end{enumerate}
        We define $\delta = \frac{\eps}{27}, B = 9B', E = \frac{1}{9}$.
    \begin{enumerate}[1. ,resume,leftmargin=*]
        \item Let us prove \ref{item: tilt lemma, case > b^1/3}. Assume that \eqref{eq: tilt lemma, case > b^1/3}, \eqref{eq: tilt lemma, condition on the radius to have curvature small}, \eqref{eq: tilt lemma, smallness of flatness assumption} and \eqref{eq: tilt lemma, angle condition} hold.
        Without loss of generality, assume that $\theta \in (0,\frac{\pi}{2})$. 
        Change coordinates so that $x=0$, $R_\theta(\mathscr{V}) = \{x_2=0\} = \mathscr{V}_0$.
        We have to show that $\exc_{\mathscr{V}_0}(\tfrac{r}{4}) \leq C (d(r)+\beta_{\mathscr{V}_0}(r))$.
         By the density lower bound (Theorem \ref{thm: density lower bound for K}) and by \eqref{eq: tilt lemma, smallness of flatness assumption}, arguing by contradiction, it is easy to obtain the following height bound for $K$:
        \begin{equation}
            \label{eq: tilt lemma, height bound}
            K \cap B_{\frac{2r}{3}} \subset \{x: |x_2| \leq 2 \left(\frac{\beta_{\mathscr{V}_0}(r)}{\eps}\right)^\frac{1}{3} r\}.
        \end{equation}
        Moreover, by the choice of $A$, $B$, $E$ and $\theta$, it is easy to show that
            \begin{equation}
                \label{eq: tilt lemma, strip does not intersect the boundary}
                \{x: |x_2| \leq 4 \left(\frac{\beta_{\mathscr{V}_0}(r)}{\eps}\right)^\frac{1}{3} r\} \cap B_{r} \cap \partial \Omega = \emptyset.
            \end{equation}
        It remains to show that 
        \begin{equation*}
            \frac{1}{r} \int_{K \cap B_\frac{r}{3}} e_2^2 \, \dd{\HH^1} \leq C \left(d(\tfrac{2r}{3}) + \frac{1}{r^3} \int_{K \cap B_\frac{2r}{3}} |x_2|^2 \, \dd{\HH^1}(x)\right).
        \end{equation*}
        Let us define $\rho = 2 \left(\frac{\beta_{\mathscr{V}_0}(r)}{\eps}\right)^\frac{1}{3} r$ and let $\zeta \in C^\infty_c((-2\rho, 2\rho); [0,1])$ be such that $\zeta \equiv 1$ in $(-\rho,\rho)$, $\norm{\zeta'}_\infty \leq \frac{6}{\rho}$. Let $\varphi \in C^\infty_c(B_\frac{2r}{3}, [0,1])$, $\varphi \equiv 1$ on $B_\frac{r}{3}$, $\norm{\nabla \varphi}_{L^\infty(B_\frac{2r}{3})} \leq \frac{6}{r}$.
        Define, for every $x \in \R^2$
        \begin{equation*}
            \eta(x) = \varphi^2(x) \zeta(x_2) (0,x_2)
        \end{equation*}
        and observe that $\eta(x) = \varphi^2(x)(0,x_2)$   for every $x \in K \cap B_{\frac{2r}{3}}$ and $\eta(x) = 0$ for every $ x \in \partial \Omega \cap B_{\frac{2r}{3}}$.
        Then, arguing as above, 
        \begin{equation*}
            \varphi^2 e_2^2 \leq 2 e^T D\eta \, e + 4 |x_2|^2 |\nabla \varphi|^2 \quad \text{ on } K \cap B_{\frac{2r}{3}}.
        \end{equation*}
        Define $K_\eps$ and $u_\eps$ as above. Then, for sufficiently small $|\eps|$, $(K_\eps, u_\eps) \in \mathcal{A}(\Omega,g)$ and hence, with easy computations and for some constant $C$, 
        \begin{equation*}
            \begin{split}
                \left|\int_{K \cap B_{\frac{2r}{3}}} e^T D\eta \, e \right| &\leq C \norm{D \eta}_{L^\infty(B_{\frac{2r'}{3}}\cap \{x: \, |x_2| \leq 2\rho\})} \int_{B_{\frac{2r}{3}} \cap \Omega \setminus K} |\nabla u|^2
                \leq C r d(\tfrac{2r}{3})
            \end{split}
        \end{equation*}
    \end{enumerate}
\end{proof}
\subsection{Proof of Lemma \ref{lmm: Lipschitz approximation}}
\label{section: proof of Lipschitz approximation}
\begin{lmm}[Invertible projection]
    \label{lmm: vertical separation}
    Let $\overline{r}>0$ be as in Lemma \ref{lmm: tilt lemma}.
    There exist $\eps, \tau, \eta > 0$ such that if
    \begin{enumerate}[label=(\roman*)]
        \item  $(K,u)$, $x$, $y$, $r$, $D$, $\mathscr{V}$ satisfy property \eqref{eq: tilt lemma, K,u,x,r,D,V condition}, 
        \item $\frac{D}{r}<\eta$,
        \item $z,z' \in K \cap B_{\frac{\tau}{4}r}(x)$, $\alpha:=|z-z'|$,  
        \item $\sup_{\frac{\alpha}{2}< \rho < \frac{\alpha}{\tau}} (d(z,\rho)+\exc_{\mathscr{V}}(z,\rho)) < \eps$,
    \end{enumerate}
    then $$|\pi^\perp(z-z')| \leq |\pi(z-z')|,$$
    where $\pi$ and $\pi^\perp$ are the orthogonal projections onto $\mathscr{V}$ and $\mathscr{V}^\perp$.
\end{lmm}
\begin{proof}
    The proof follows the same idea of \cite[Proof of Lemma 3.3.3]{delellis-focardi}.
    Suppose, by contradiction, that the statement is false. 
    Then, for every $j\in\N$, there exist \minimizer{K_j}{u_j}{\Omega}{g}, $x_j \in K_j$, $r_j \in (0,\overline{r})$ such that $D_j = \dist(x_j,\partial\Omega) < \frac{1}{j} r_j$, $y_j \in \partial \Omega$ such that $|y_j-x_j|=D_j$, $T_{y_j}\partial\Omega = \mathscr{V}_j$, $z_j, z_j' \in K_j \cap B_{\frac{1}{4j}r_j}(x_j)$, $\alpha_j:= |z_j-z_j'|$, 
    \begin{equation*}
        \sup_{\frac{\alpha_j}{2}<\rho<j \alpha_j} (d(z_j,\rho) + \exc_{\mathscr{V}_j}(z_j,\rho))< \frac{1}{j}, \quad  |\pi_j^\perp(z_j-z_j')| > |\pi_j(z_j-z_j')|,
    \end{equation*} 
    where $\pi_j$ and $\pi_j^\perp$ are the orthogonal projections onto $\mathscr{V}_j$ and $\mathscr{V}_j^\perp$.

    For a subsequence that we do not rename, $y_j \to y \in \partial \Omega$, and hence $z_j \to y$.
    Up to translations and rotations, we can assume that $y=0$ and $T_y\partial \Omega = \{x_2=0\}$.

    We define $\varphi_j(x) = \frac{x-z_j}{\alpha_j}
    $ for every $x \in \R^2$ and $K_j' = \varphi_j(K_j) = (K_j)_{z_j,\alpha_j}$, $\Omega_j' = \Omega_{z_j,\alpha_j}$,$u_j' = (u_j - g(z_j))_{z_j,\alpha_j}.$
    Then \minimizer{K_j'}{u_j'}{\Omega_j'}{(g-g(z_j))_{z_j,\alpha_j}}.
    Observe that $\varphi_j(z_j) = 0$, $|\varphi_j(z_j')|=1$ for every $j$. Moreover, $|\pi_j^\perp\varphi_j(z_j')| > |\pi_j\varphi_j(z_j')|$.

    For a subsequence that we do not rename, $\Omega_j' \to H$ locally in the Hausdorff distance. 
    \begin{enumerate}[leftmargin=*]
        \item If $H = \R^2$ then the conclusion follows from \cite[Proof of Lemma 3.3.3]{delellis-focardi}.
        \item Assume, instead, that $H=\{x_2 > 0\}$.
        By Proposition \ref{prop: global convergence of functions, blow-up} and Theorem \ref{thm: blow-up limit}, for a subsequence that we do not rename $(K_j',u_j')$ converges to $(K_0,u_0)$ in the sense of the above Proposition.  
        Moreover, for any $R>\frac{1}{2}$,
        \begin{equation*}
            \int_{B_R} |\nabla u_0|^2 = \lim_{j\to \infty}\int_{B_R \cap (\Omega_j' \setminus K_j')} |\nabla u_j'|^2 = R \lim_{j \to \infty} d(z_j, \alpha_j R) = 0, 
        \end{equation*}
        and hence $(K_0,u_0)$ is an elementary generalized minimizer, according to Definition \ref{def: elementary generalized minimizers}. 
        Moreover, $0 = \lim \varphi_j(z_j) \in K_0$ and hence, by Theorem \ref{prop: classification of elementary generalized minimizers}, $K_0 = \{x_2 = 0\}$. 
        On the other hand, since $|\varphi_j(z_j')|=1$, for a subsequence that we do not rename $\varphi_j(z_j') \to w \in K_0$, with $|w|=1$ and $|w_2| \geq |w_1|$. This implies that $|w_2|>0$ which is a contradiction. 
        \item Finally, assume that, $H = \{x_2 > -a\}$ for some $a>0$. As in the previous point, $0 \in K_0$. Therefore, by Proposition \ref{prop: classification of elementary generalized minimizers}, $K_0 = \{x_2 = 0\}$. Then the same argument of the previous point gives a contradiction. 
    \end{enumerate}
\end{proof}
\begin{proof}[Proof of Lemma \ref{lmm: Lipschitz approximation}]
    It follows from the same ideas of \cite[Proof of Proposition 3.3.1]{delellis-focardi}  
 \end{proof} 
\subsection{Proof of Lemma \ref{lmm: flatness improvement}}
\label{section: proof of flatness improvement}
\begin{proof}[Proof of Lemma \ref{lmm: flatness improvement}]
    \begin{enumerate}[leftmargin=*]
        \item Let us prove \ref{item: flatness improvement, case D/r < b^1/2}. 
        Suppose, by contradiction, that the statement is false. Let $\tau_1 >0$. Then there exist $\tau \leq \tau_1$, $\overline{\delta}>0$ such that, for every $j \in \N$, there exist \minimizer{K_j}{u_j}{\Omega}{g}, $x_j \in K_j$, $r_j \in (0, \overline{r})$, $y_j \in \partial \Omega$ such that $|y_j-x_j| = D_j = \dist(x_j,\partial\Omega)$, $\mathscr{V}_j = T_{y_j} \partial \Omega$, 
        such that 
        \begin{align}
            &\overline{\delta} \max\{d(x_j,r_j), r_j^\frac{1}{2}\} \leq \beta_{x_j+\mathscr{V}_j}(x_j,r_j) =: \beta_j \leq \frac{1}{j},  \label{eq: smallness assumptions in the proof of energy decay lemma} \\
            &0\leq \frac{D_j}{r_j} \leq \frac{1}{j} \beta_{x_j+\mathscr{V}_j}(x_j,r_j)^\frac{1}{2},  \label{eq: flatness decay, D/r = o(beta^1/2)} \\
            &\beta_{x_j+\mathscr{V}_j}(x_j,\tau r_j) > \tau \beta_j.
        \end{align}
        Now, for every $j \in \N$, translate and rotate the plane in such a way that, in the new coordinates, in which $\Omega$ and $g$ get renamed to $\Omega_j$ and $g_j$, while $K_j$, $u_j$, $x_j$, $y_j$ are not renamed, $x_j = 0$ and $T_{y_j} \partial \Omega_j = \mathscr{V}_0 = \{x_2=0\}$.

        Therefore, 
        \begin{equation}
            \label{eq: flatness improvement on smaller scale contradiction, case < b^1/2}
            \int_{B_{\tau r_j} \cap K_j}|x_2|^2 \, \dd{\HH^1}(x) > \tau^4 r_j^3\beta_j \quad \forall j \in \N.
        \end{equation}
        Let $\sigma >0$ be the constant of Lemma \ref{lmm: Lipschitz approximation}.
        For sufficiently large $j \in \N$, let $f_j: [-\sigma r_j, \sigma r_j] \to \R$ be the function given by Lemma \ref{lmm: Lipschitz approximation} \ref{item: Lipschitz approximation, D/r < 1} and let $\Gamma_j = \graph(f_j) \subset [-\sigma r_j, \sigma r_j] \times [-\frac{\sigma}{2}r_j, \frac{\sigma}{2}r_j]$. 
        
        Since $r_j \leq \overline{r}$ and $D_j \leq r_j$, we have, for every $j \in \N$, $\Omega_j \cap B_{r_j} = \{(x_1,x_2) \in B_{r_j} : x_2 > b_j(x_1) - D_j\}$
        where $b_j \in C^2(R)$, $b_j(0) = b_j'(0) = 0$ and $|b_j(t)| \leq C t^2$ for some constant $C$ and for every $t \in (-2r_j, 2r_j)$.
        Moreover, we have $f_j \geq b_j-D_j$.
    
        By Lemma \ref{lmm: Lipschitz approximation} \ref{item: Lipschitz approximation, D/r < 1} (iii) and (iv), we have, setting $C' = C(1+2 \overline{\delta}^{-1})$,
        \begin{align}
            \HH^1((\Gamma_j \triangle K_j) \cap [-\sigma r_j, \sigma r_j]^2) &\leq C' r_j \beta_j \label{eq: length of graph(f_j) Delta K_j}\\
            \int |f_j'|^2 &\leq C' r_j \beta_j. \label{eq: f_j' bound in L^2}
        \end{align}    
        On the other hand, since $0 \in K_j$, by the density lower bound (Theorem \ref{thm: density lower bound for K}) we have, for sufficiently large $j \in \N$,
        \begin{equation*}
            \HH^1(K_j \cap B_{2 \frac{C'}{\eps} r_j \beta_j}) \geq 2C' r_j \beta_j.
        \end{equation*}
        This and \eqref{eq: length of graph(f_j) Delta K_j} imply that, for sufficiently large $j$, there exists $z_j = (t_j, f_j(t_j)) \in K_j \cap B_{2 \frac{C'}{\eps} r_j \beta_j}$.
        For every $j \in \N$, we define $h_j: [-\sigma,\sigma] \to \R$ as follows. 
        \begin{equation*}
            h_j(t) = \frac{f_j(r_j t) - f_j(t_j)}{r_j \beta_j^\frac{1}{2}} \quad \forall t \in [-\sigma, \sigma].
        \end{equation*}
        Observe that $\{h_j'\}_j$ is uniformly bounded in $L^2((-\sigma, \sigma))$ by \eqref{eq: f_j' bound in L^2}. Moreover, the absolute continuity of $f_j$ and \eqref{eq: f_j' bound in L^2} imply that also $\{h_j\}_j$ is uniformly bounded in $L^2((-\sigma, \sigma))$.
        Then $\{h_j\}_j$ is uniformly bounded in $W^{1,2}((-\sigma,\sigma))$. Therefore there exists $h: [-\sigma, \sigma] \to \R$ such that, up to subsequences, $h_j \rightharpoonup h$ in $W^{1,2}((-\sigma,\sigma))$ and $h_j \to h$ uniformly in $[-\sigma, \sigma]$.
    
        Let $\zeta \in C^\infty_c((-\sigma, \sigma))$, $\zeta \geq 0$. Then 
        \begin{equation*}
            \int_{-\sigma}^\sigma h' \zeta ' = \lim_{j \to \infty} \frac{1}{r_j \beta_j^\frac{1}{2}} \int_{-\sigma r_j}^{\sigma r_j} f_j' \varphi_j'
        \end{equation*}
        where $\varphi_j(s):= r_j \zeta(\frac{s}{r_j})$ for every $s \in (-\sigma r_j, \sigma r_j)$.
        Moreover, an easy computation using \eqref{eq: f_j' bound in L^2} and $\lip(f_j) \leq 1$ shows that
        \begin{equation*}
            \lim_{j \to \infty} \frac{1}{r_j \beta_j^\frac{1}{2}} \int_{-\sigma r_j}^{\sigma r_j} f_j' \varphi_j' = \lim_{j \to \infty} \frac{1}{r_j \beta_j^\frac{1}{2}} \int_{-\sigma r_j}^{\sigma r_j} \frac{f_j' \varphi_j'}{\sqrt{1+|f_j'|^2}}.
        \end{equation*}
        Now observe that 
        \begin{equation*}
            \int_{-\sigma r_j}^{\sigma r_j} \frac{f_j' \varphi_j'}{\sqrt{1+|f_j'|^2}} 
            = \derivative{}{t} \HH^1(\Phi_t(\Gamma_j))|_{t=0} 
            = \int_{\Gamma_j \cap [-\sigma r_j, \sigma r_j]^2} e_j^T \, D \eta_j \, e_j \, \dd{\HH^1},
        \end{equation*}
        where $\Phi_t(x) = x+t\eta_j(x)$, $\eta_j(x) = (0,\varphi_j(x_1)\psi_j(x_2))$ for every $x \in \R^2$, $\psi_j(s) = \psi(\frac{s}{r_j})$ for every $s \in (-\sigma r_j, \sigma r_j)$, $\psi \in C^\infty_c((-\sigma,\sigma))$, $\psi \geq 0$, $\psi \equiv 1$ in $(-\frac{\sigma}{2}, \frac{\sigma}{2})$, and $e_j$ is a unit vector field tangent to $\Gamma_j$.
        By \eqref{eq: length of graph(f_j) Delta K_j} we have
        \begin{equation*}
            \int_{-\sigma}^\sigma h' \zeta ' = \lim_{j \to \infty} \frac{1}{r_j \beta_j^\frac{1}{2}} \int_{K_j \cap [-\sigma r_j, \sigma r_j]^2} e_j^T \, D \eta_j \, e_j \, \dd{\HH^1}.
        \end{equation*}
        For every $t\geq 0$, define $K_j^t = \Phi_t(K_j)$ and $u_j^t = u_j \circ \Phi_t^{-1}$.
        By construction, $\Phi_t|_{\partial [-\sigma r_j, \sigma r_j]^2}$ is the identity. 
        Moreover, we extend $u_j$ to $\{(x_1,x_2): x_1 \in [-\sigma r_j, \sigma r_j], \, x_2 < b_j(x_1)-D_j\}$ as follows. 
        \begin{equation*}
            u_j(x_1,x_2) := g_j(x_1,b_j(x_1)) \quad \forall x_1 \in [-\sigma r_j, \sigma r_j], \, x_2 < b_j(x_1)-D_j.
        \end{equation*}
        Then, by construction, $u_j^t = g_j$ on $\partial \Omega_j \cap [-\sigma r_j, \sigma r_j]^2$.
        Therefore $(K_j^t, u_j^t) \in \mathcal{A}(\Omega_j,g_j)$ for $t \geq 0$. 
    
        By the same computations of Lemma \ref{lmm: first derivative of the energy for generalized minimizers}, this implies that
        \begin{equation*}
            \int_{[-\sigma r_j, \sigma r_j]^2\cap \Omega_j \setminus K_j} (|\nabla u_j|^2 \diverg \eta_j - 2 \nabla u_j^T D\eta_j \, \nabla u_j) + \int_{K_j \cap [-\sigma r_j, \sigma r_j]^2} e_j^T D\eta_j \, e_j  \geq 0.
        \end{equation*}
        Since the absolute value of the first term above is controlled by $r_j \beta_j$, we conclude that
        \begin{equation*}
            \int_{-\sigma}^\sigma h' \zeta ' \geq 0 \quad \forall \zeta \in C^\infty_c((-\sigma, \sigma)), \, \zeta \geq 0.
        \end{equation*}
        Hence $h$ is weakly superharmonic and, as a consequence, $h$ satisfies the strong minimum principle: if there is $y \in (-\sigma, \sigma)$ such that $h(y) = \inf_{(-\sigma, \sigma)} h$, then $h$ is constant.
    
        First, we claim that $h \geq 0$. Indeed, $f_j \geq b_j-D_j$ by Lemma \ref{lmm: Lipschitz approximation}, and hence 
        \begin{equation*}
            h_j(t) \geq  \frac{b_j(r_j t)-D_j-f_j(t_j)}{r_j \beta_j^\frac{1}{2}}.
        \end{equation*}
        Recall that $|b_j(r_j t)| \leq C \sigma^2 r_j^2$ for every $t \in (-\sigma, \sigma)$. By \eqref{eq: smallness assumptions in the proof of energy decay lemma} this implies that $|b(r_j t)| \leq C \sigma \overline{\delta}^{-1} r_j \beta_j$.
        Furthermore, $|f_j(t_j)| \leq 2\frac{C'}{\eps} r_j \beta_j$ by construction.
        Finally, \eqref{eq: flatness decay, D/r = o(beta^1/2)} implies that $\frac{D_j}{r_j} = o(\beta_j^\frac{1}{2})$.
        We conclude that the right hand side of the above inequality goes to $0$ as $j\to \infty$, and hence $h \geq 0$. 
    
        Second, we claim that $h(0) = 0$. Indeed, $\frac{t_j}{r_j} \to 0$, $h_j(\frac{t_j}{r_j}) = 0$ and $h_j \to h$ uniformly.
       
        Then, by the strong minimum principle, we conclude that $h \equiv 0$. 
        As a conclusion, we get, by a simple computation, 
        \begin{equation*}
            \int_{B_{\tau r_j} \cap K_j} |x_2 - f_j(t_j)|^2 \, \dd{\HH^1}(x) = o(r_j^3 \beta_j)
        \end{equation*}
        and hence, since $|f_j(t_j)| \leq 2\frac{C'}{\eps}r_j\beta_j$, 
        \begin{equation*}
            \int_{B_{\tau r_j} \cap K_j} |x_2|^2 \, \dd{\HH^1}(x) = o(r_j^3 \beta_j)
        \end{equation*}
        which contradicts \eqref{eq: flatness improvement on smaller scale contradiction, case < b^1/2}.
        \item Let us prove \ref{item: flatness improvement, case D/r > b^1/3}. 
        Suppose, by contradiction, that the statement is false. Let $\tau_1 >0$. Then there exist $\tau \leq \tau_1$, $\overline{\delta}>0$ such that, for every $j \in \N$, there exist \minimizer{K_j}{u_j}{\Omega}{g}, $x_j \in K_j$, $r_j \in (0, \overline{r})$, $\theta_j \in [0,2\pi)$, $y_j \in \partial \Omega$ such that $|y_j-x_j| = D_j = \dist(x_j,\partial\Omega)$, $\mathscr{V}_j = T_{y_j} \partial \Omega$, 
        such that
        \begin{align*}
            &A \beta_{x_j+R_{\theta_j}(\mathscr{V}_j)}(x_j,r_j)^\frac{1}{3} < \frac{D_j}{r_j} < \eta, \\
            &\overline{\delta} \max\{d(x_j,r_j), r_j^\frac{1}{2}\} \leq \beta_{x_j+R_{\theta_j}(\mathscr{V}_j)}(x_j,r_j), \\
            &|\sin\theta_j| < E \frac{D_j}{r_j}, \quad |\cos\theta_j| > \frac{9}{10},\\
            &\beta_j:= \beta_{x_j+R_{\theta_j}(\mathscr{V}_j)}(x_j,r_j) \leq \frac{1}{j} \\
            &\min_{|\tan(\phi)|\leq b_j}\beta_{x_j+R_{\phi+\theta_j}(\mathscr{V}_j)}(x_j,\tau r_j) > \tau \beta_{x_j+R_{\theta_j}(\mathscr{V}_j)}(x_j,r_j),
        \end{align*}
        where $b_j = \sqrt{\frac{6(1+2\overline{\delta}^{-1})C}{\sigma}} \beta_j^\frac{1}{2}$.

        Now, for every $j \in \N$, translate and rotate the plane in such a way that, in the new coordinates, in which $\Omega$ and $g$ get renamed to $\Omega_j$ and $g_j$, while $K_j$, $u_j$, $x_j$, $y_j$ are not renamed, $x_j = 0$ and $R_{\theta_j}(\mathscr{V}_j) = \mathscr{V}_0 = \{x_2 = 0\}$.
        Therefore 
        \begin{equation}
            \label{eq: flatness improvement on smaller scale contradiction, case > b^1/3}
            \int_{B_{\tau r_j} \cap K_j} \dist^2(x,R_\phi(\mathscr{V}_0))\, \dd{\HH^1}(x) > \tau^4 r_j^3 \beta_j
        \end{equation}
        for every $\phi$ such that
        $|\tan\phi| \leq \sqrt{\frac{6(1+2\overline{\delta}^{-1})C}{\sigma}} \beta_j^\frac{1}{2}$
        and for every $j \in \N$. 
    
        For sufficiently large $j \in \N$, let $f_j: [-\sigma r_j, \sigma r_j] \to \R$ be the function given by Lemma \ref{lmm: Lipschitz approximation} \ref{item: Lipschitz approximation, D/r > b^1/3} and let $\Gamma_j = \graph(f_j) \subset [-\sigma r_j, \sigma r_j] \times [-\frac{\sigma}{2}r_j, \frac{\sigma}{2}r_j]$. 
        Let $h_j$, $h$ be as above. Let $\zeta \in C^\infty_c((-\sigma,\sigma))$. Then, as above, 
        \begin{equation*}
            \begin{split}
                \int_{-\sigma}^\sigma h' \zeta '  
                &= \lim_{j \to \infty} \frac{1}{r_j \beta_j^\frac{1}{2}} \derivative{}{t} \HH^1(\graph(f_j + t \varphi_j))|_{t=0}.
            \end{split}
        \end{equation*}
        Moreover, the above construction implies that $\norm{f_j}_\infty \lesssim r_j \beta_j^\frac{1}{2}$ and hence $\norm{f_j}_\infty \leq \frac{2}{\eps^\frac{1}{3}} \beta_j^\frac{1}{3}r_j  =: \rho_j$ for $j$ sufficiently large. 
        Let $\psi_j \in C^\infty_c((-2\rho_j, 2\rho_j); [0,1])$ be such that $\psi_j \equiv 1$ in $(-\rho_j, \rho_j)$ and $\norm{\psi_j'}_\infty \leq \frac{6}{\rho_j}$.
        We define 
        \begin{equation*}
            \eta_j(x) = (0,\varphi_j(x_1)\psi_j(x_2))\quad \forall x \in \R^2.
        \end{equation*}
        and $\Phi_t$, $K^t_j$, $u^t_j$ as above. 
        Then $\Phi_t(\Gamma_j) = \graph(f_j + t\varphi_j)$ and, by \eqref{eq: tilt lemma, strip does not intersect the boundary}, $\eta_j(x) = 0$ for every $x \in \partial\Omega_j \cap [-\sigma r_j, \sigma r_j]^2$.
        Therefore 
        \begin{equation*}
            \int_{K_j \cap [-\sigma r_j, \sigma r_j]^2} e_j^T D\eta_j \, e_j = - \int_{A_j} (|\nabla u_j|^2 \diverg \eta_j - 2 \nabla u_j^T D\eta_j \, \nabla u_j),
        \end{equation*} 
        where $A_j = [-\sigma r_j, \sigma r_j]^2\cap (\Omega_j \setminus K_j) \cap \{|x_2|\leq 2\rho_j\}$.
        With easy computations, this implies that
        \begin{equation*}
            \int_{-\sigma}^{\sigma} h'\zeta' = 0 \quad \forall \zeta \in C^\infty_c((-\sigma, \sigma)).
        \end{equation*}
        and hence $h(x_1) = ax_1$ for some $a \in \R$. Using the convergence $h_j \to h$ in $L^2((-\sigma, \sigma))$ and \eqref{eq: f_j' bound in L^2} it is easy to see that
        \begin{equation*}
            |a| \leq \sqrt{\frac{6C'}{\sigma}} = \sqrt{\frac{6(1+2\overline{\delta}^{-1})C}{\sigma}}.
        \end{equation*}
        Consider $\mathscr{W}_j = \{(x_1, \beta_j^\frac{1}{2} a x_1): x_1 \in \R\} \in \mathcal{L}.$
        Then it is easy to see that 
        \begin{equation*}
            \int_{B_{\tau r_j} \cap K_j} \dist^2(x,\mathscr{W}_j) \, \dd{\HH^1}(x) = o(r_j^3 \beta_j).
        \end{equation*}
        Finally, observe that $\mathscr{W}_j = R_{\phi_j}(\mathscr{V}_0)$
        with $|\tan\phi_j| = |a \beta_j^\frac{1}{2}| \leq \sqrt{\frac{6(1+2\overline{\delta}^{-1})C}{\sigma}} \beta_j^\frac{1}{2}$, contradicting \eqref{eq: flatness improvement on smaller scale contradiction, case > b^1/3}.
    \end{enumerate}
\end{proof}
\subsection{Proof of Lemma \ref{lmm: flatness preserved in the unfavourable case}}
\label{section: flatness in the unfavourable case}
\begin{proof}[Proof of Lemma \ref{lmm: flatness preserved in the unfavourable case}]
    Fix $\overline{\delta} > 0$ and $\iota_1 >0$. 
    Up to translations, rotations and dilations, assume that $x = 0$, $T_y \partial \Omega = \{x_2=0\}$ and $r = 1$. In this proof, let us use the notation $\beta(\rho):= \beta_{\mathscr{V}_0}(\rho)$. Then our assumption \eqref{eq: flatness unfavourable case} becomes  
    \begin{equation*}
        \iota_1 \beta(1)^\frac{1}{2} \leq D \leq A \beta(1)^\frac{1}{3},
    \end{equation*} 
    and we have to find $C>0$ and $\sigma \in (0,1)$ such that $\frac{D}{\sigma} > A \beta(\sigma)^\frac{1}{3}$, $\beta(\rho) \leq C \beta(1)^\frac{1}{8} \rho^\frac{1}{8}$ and $d(\rho) \leq C \beta(1)^\frac{1}{8} \rho^\frac{1}{8}$
    for any $\sigma \leq \rho \leq 1$.

    For sufficiently small $\eta_3 >0$, let $f: [-2\sigma' , 2\sigma' ] \to \R$ be the function given by Lemma \ref{lmm: Lipschitz approximation} \ref{item: Lipschitz approximation, D/r < 1} and let $\Gamma = \graph(f)$.
    We have 
    \begin{equation*}
        \HH^1((\Gamma \triangle K) \cap [-2\sigma', 2\sigma']^2) \leq C D^2, \quad  \int |f'|^2 \leq C D^2
    \end{equation*}
    for some constant $C$ depending on $\overline{\delta}$, on $\iota_1$, and on the constant $C$ of Lemma \ref{lmm: Lipschitz approximation}.

    Since $0 \in K$, the density lower bound (Theorem \ref{thm: density lower bound for K}) implies that there exists $z = (t, f(t)) \in B_{2\frac{C}{\eps} D^2}$.

    Let $0 < \sigma < \sigma'$ to be chosen later.
    By Hölder's inequality, we have 
    \begin{equation*}
        \norm{f}_{L^\infty(-2\sigma, 2\sigma)} \leq 2(C \sigma)^\frac{1}{2}D + 2\frac{C}{\eps}D^2.
    \end{equation*}
    Then, if 
    \begin{equation}
        \label{eq: flatness in the unfavourable case, condition for sigma}
        2 \frac{C}{\eps} D^2 < \sigma, 
    \end{equation}
    a simple contradiction argument using the density lower bound (Theorem \ref{thm: density lower bound for K}) implies that
    \begin{equation*}
        K \cap B_\sigma \subset \{x: |x_2| \leq 2(C \sigma)^\frac{1}{2}D + 4\frac{C}{\eps}D^2\}.
    \end{equation*}
    Now we estimate $\beta(\sigma)$. We have, by the energy upper bound (Proposition \ref{prop: Energy upper bound at the boundary}) and by the above estimates, 
    \begin{equation*}
        \begin{split}
            \beta(\sigma) &= \frac{1}{\sigma^3} \left(\int_{K \cap \Gamma \cap B_\sigma}|x_2|^2 \, \dd{\HH^1}(x) + \int_{(K \setminus \Gamma) \cap B_\sigma}|x_2|^2 \, \dd{\HH^1}(x) \right)\\
            &\leq \frac{C'}{\sigma^3} (\sigma^2 D^2 + \sigma D^4 + D^6) \leq \frac{3 C'}{\sigma^3} \max\{\sigma^2 D^2, \sigma D^4, D^6\}
        \end{split}
    \end{equation*}
    for some constant $C'$ that depends on $C$ and on the constant of the energy upper bound. 

    Now we define $\tilde{C} :=  \frac{1}{(3 C' A^3)^\frac{1}{2}}$ and we choose $\sigma := \tilde{C} D^\frac{1}{2}$.
    Recall that $D \leq A \beta(1)^\frac{1}{3} < A \eta_3^\frac{1}{3}$.
    Then condition \eqref{eq: flatness in the unfavourable case, condition for sigma} is true if we choose $\eta_3$ sufficiently small.

    \begin{enumerate}[leftmargin=*]
        \item Let us prove \ref{item: intermediate regime, D/sigma > beta(sigma)^1/3}. First, observe that
        \begin{equation*}
            \left(\frac{D}{\sigma}\right)^3 \geq \frac{D^3}{\sigma}\frac{3C'A^3}{D} = A^3 \frac{3C'}{\sigma^3} \sigma^2 D^2.
        \end{equation*}
        Second, for $\eta_3$ sufficiently small, we have 
        \begin{align*}
           &D^3 \geq 3 C' A^3 D^4 \geq 3 C' A^3 \sigma D^4, \\
           &D^3 \geq 3 C' A^3 D^6.
        \end{align*}
        We conclude by looking at each term appearing in the $\max$ in the above estimate of $\beta(\sigma)$.
        \item  Now let us prove \ref{item: intermediate regime, beta estimate} and \ref{item: intermediate regime, d estimate}. Assume that $\sigma \leq \rho \leq 1$. Observe that 
        \begin{equation*}
            \tilde{C} \iota_1^\frac{1}{2}\beta(1)^\frac{1}{4} \leq \tilde{C} D^\frac{1}{2} = \sigma \leq \rho.
        \end{equation*}
        Therefore
        \begin{equation*}
            \beta(\rho) \leq \frac{1}{\rho^3} \beta(1) \leq \frac{1}{\tilde{C}^3 \iota_1^\frac{3}{2}}\beta(1)^\frac{1}{4} \leq \frac{1}{\tilde{C}^\frac{7}{2} \iota_1^\frac{7}{4}}\beta(1)^\frac{1}{8} \rho^\frac{1}{2}.
        \end{equation*}
        If (d) holds, we have 
        \begin{equation*}
            d(\rho) \leq \frac{1}{\rho} d(1) \leq \frac{1}{\overline{\delta}\rho} \beta(1). 
        \end{equation*}
        If (e) holds, we have 
        \begin{equation*}
            d(\rho) \leq \frac{1}{\rho} d(1) \leq \frac{1}{\rho} D^2 \leq \frac{A^2}{\rho} \beta(1)^\frac{2}{3}. 
        \end{equation*}
        In both cases, 
        \begin{equation*}
            d(\rho) \lesssim \beta(1)^\frac{1}{4} \lesssim \beta(1)^\frac{1}{8} \rho^\frac{1}{2}.
        \end{equation*}
    \end{enumerate}
\end{proof}
\subsection{Proof of Lemma \ref{lmm: rescaled dirichlet energy decay}}
\label{section: rescaled dirichlet energy decay}
\begin{proof}[Proof of Lemma \ref{lmm: rescaled dirichlet energy decay}]
    \begin{enumerate}[leftmargin=*]
        \item The proof of \ref{item: Dirichlet, case 1} is based on \cite[Proof of Lemma 3.2.3]{delellis-focardi}.
        Suppose, by contradiction, that the claim is false. Let $\tau_2 >0$. Then there exist $\tau \leq \tau_2$, $\overline{\delta}>0$ such that, for every $j \in \N$, there exists \minimizer{K_j}{u_j}{\Omega}{g}, $x_j \in K_j$, $r_j \in (0, \overline{r})$, $y_j \in \partial \Omega$ such that $|y_j-x_j| = D_j = \dist(x_j,\partial\Omega)$, $\mathscr{V}_j = T_{y_j} \partial \Omega$, 
        such that $\frac{D_j}{r_j} \leq \frac{1}{j}$,
        \begin{align*}
            &0\leq \frac{D_j}{r_j} \leq A \beta_{x_j+\mathscr{V}_j}(x_j,r_j)^\frac{1}{2},\\
            &\overline{\delta}\max\{\beta_{x_j+\mathscr{V}_j}(x_j,r_j), r_j^\frac{1}{2}\} \leq d(x_j,r_j)=:d_j \leq \frac{1}{j},\\
            &d(x_j, \tau r_j) > \tau^\frac{1}{2} d_j. 
        \end{align*}
        Without renaming it, replace $(K_j, u_j)$ with $(K_j, u_j - g(x_j))$ (this does not affect its energy).
        
        Now, for every $j \in \N$, translate and rotate the plane in such a way that, in the new coordinates, in which $\Omega$ and $g - g(x_j)$ get renamed to $\Omega^j$ and $g^j$, while $K_j$, $u_j$, $x_j$, $y_j$ are not renamed, $x_j = 0$ and $T_{y_j} \partial \Omega^j = \mathscr{V}_0 = \{x_2=0\}$.
        Then we have 
        \begin{equation*}
            \int_{B_{\tau r_j}\cap \Omega^j \setminus K_j} |\nabla u_j|^2 > \tau^\frac{3}{2} r_j d_j.
        \end{equation*}
        Moreover, we define   
        \begin{equation*}
            \beta_j := \beta(r_j) = \frac{1}{r_j^3} \int_{B_{r_j} \cap K_j} |x_2|^2 \, \dd{\HH^1}(x).
        \end{equation*}

        Let $\overline{K}_j := r_j^{-1} K_j$ and $v_j(y) := (r_j d_j)^{-\frac{1}{2}}u_j(r_j y)$ for every $y \in r_j^{-1}\Omega^j =: \Omega_j$.
        By Lemma \ref{lmm: Lipschitz approximation} \ref{item: Lipschitz approximation, D/r < 1} (i) we have 
        \begin{equation*}
            K_j \cap B_\frac{r_j}{2} \subset \{|x_2| \leq C r_j \beta_j^\frac{1}{3} \}, 
        \end{equation*}
        and hence 
        \begin{equation*}
            \overline{K}_j \cap B_\frac{1}{2} \subset \{|x_2| \leq C \beta_j^\frac{1}{3}\}.
        \end{equation*}
        Moreover, by Lemma \ref{lmm: Lipschitz approximation} (v), for any $\eps >0$ 
        \begin{equation*}
            \lim_{j \to \infty} \HH^1([-\sigma, \sigma] \setminus \pi_{\mathscr{V}_0} (\overline{K}_j \cap \{|x_2| \leq \eps\}))  = 0.
        \end{equation*}
        Assuming $\sigma < \frac{1}{2}$, the two above properties imply that 
        $\overline{K_j} \cap [-\sigma, \sigma]^2$ converges to $\{x_2 = 0, |x_1|\leq\sigma \}$ in the Hausdorff distance. 
        Moreover, 
        \begin{equation*}
            \int_{B_1 \cap \Omega_j \setminus \overline{K}_j} |\nabla v_j|^2 = 1
        \end{equation*}
        for every $j \in \N$. 
        Observe that, for every $\eps >0$, $v_j$ is harmonic in $B_{\sigma} \cap \{x_2 > \eps\}$ for sufficiently large $j$. Therefore, it is easy to see that there exist $\{c_j\}_j \in \R$ such that, for a subsequence that we do not rename, 
        \begin{equation*}
            v_j - c_j \to v \quad \text{ in } W^{1,2}_\loc (B_\sigma^+), 
        \end{equation*}
        where $A^+ := A \cap \{x_2 > 0\}$.
    
        We claim that $|\nabla v_j|^2 \mathcal{L}^2 \mres (\Omega_j \setminus \overline{K}_j)$ does not concentrate on $[-\frac{\sigma}{2}, \frac{\sigma}{2}] \times \{0\}$ as $j \to \infty$. 
        Suppose by contradiction that this is false. Then there exists $\eps_j \downarrow 0$ and $\theta >0$ such that 
        \begin{equation}
            \label{eq: concentration of energy}
            \int_{(-\frac{\sigma}{2}, \frac{\sigma}{2}) \times (-\frac{\eps_j}{2}, \frac{\eps_j}{2}) \cap \Omega_j \setminus \overline{K}_j} |\nabla v_j|^2 \geq \theta.
        \end{equation}
        Without loss of generality, we can assume that 
        \begin{equation}
            \label{eq: choice of eps_j in decay lemma}
            \eps_j \geq C (\overline{\delta}^{-1} d_j)^\frac{1}{3}
        \end{equation}
        where $C$ is the constant of Lemma \ref{lmm: Lipschitz approximation}. As a consequence, $\eps_j \geq C \beta_j^\frac{1}{3}$. 
        We define $g_j(y) = (r_j d_j)^{-\frac{1}{2}}g^j(r_jy)$ for every $y \in \R^2$. 
        Furthermore, we define, for every $j \in \N$ and $(J,v) \in \mathcal{A}(\Omega_j, g_j)$, 
        \begin{equation*}
            F_j(J,v) = \int_{B_1 \cap \Omega_j \setminus J} |\nabla v|^2 + \frac{1}{d_j} \HH^1(B_1 \cap J).
        \end{equation*}
        Observe that 
        \begin{equation*}
            F_j(\overline{K}_j, v_j) = \frac{1}{r_j d_j} E(K_j, u_j, B_{r_j}\cap \Omega^j).
        \end{equation*}
        This implies that $(\overline{K}_j, v_j) \in \argmin\{F_j(J,v): (J,v) \in \mathcal{A}(\Omega_j, g_j)\}$.
        In what follows we want to construct a competitor for $(K_j,v_j)$ with strictly lower energy for sufficiently large $j$, to get a contradiction.
    
        Let $f_j: [-\sigma r_j, \sigma r_j] \to \R$ be the Lipschitz approximation of $K_j$ given by Lemma \ref{lmm: Lipschitz approximation} \ref{item: Lipschitz approximation, D/r < 1} and define 
        \begin{equation*}
            \overline{f}_j(t) = r_j^{-1}f_j(r_j t ) \quad \forall t \in [-\sigma, \sigma].
        \end{equation*}
        Since $0 \in K_j$, the density lower bound (Theorem \ref{thm: density lower bound for K}) and the properties of $f_j$ imply that 
        \begin{equation}
            \label{eq: bound on overline(f)_j}
            \norm{\overline{f}_j}_{L^\infty(-\sigma,\sigma)}  \lesssim d_j^\frac{1}{2}.
        \end{equation}
    
        Fix $\eps>0$ to be chosen later. By the pigeonhole principle, using 
        \begin{align*}
            &\int_{B_1 \cap \Omega_j \setminus \overline{K}_j} |\nabla v_j|^2 = 1, \\
            &\HH^1((\graph(\overline{f}_j) \triangle \overline{K}_j)\cap [-\sigma, \sigma ]^2) \leq C \overline{\delta}^{-1} d_j =: C' d_j,
        \end{align*}
        there exist $a_j' \in [-\sigma, -\frac{3}{4}\sigma]$, $b_j' \in [\frac{3}{4}\sigma, \sigma]$ such that 
        \begin{align*}
            &\int_{R_j \cap \Omega_j \setminus \overline{K}_j} |\nabla v_j|^2 \leq C \eps, \\
            &\HH^1((\graph(\overline{f}_j) \triangle \overline{K}_j)\cap R_j) \leq C \eps d_j, 
        \end{align*}
        where $C$ is a constant depending on $\sigma$ and $\delta$, and 
        \begin{equation*}
            R_j = ([a_j',a_j'+5\eps) \cup (b_j'-5\eps,b_j'])\times[-\sigma, \sigma].
        \end{equation*}
    
        Now, for every $j \in \N$, let $s_j$ be the lower convex envelope of $\overline{f}_j$ on $[a_j', b_j']$.
        We claim that, for every $j$, there exist $a_j \in [a_j', a_j'+4\eps)$, $b_j \in (b_j'-4\eps,b_j']$ such that 
        \begin{equation*}
            s_j(a_j) = \overline{f}_j(a_j), \quad s_j(b_j) = \overline{f}_j(b_j), \quad \lip(s_j) \leq \frac{\norm{\overline{f}_j}_\infty}{\eps}.
        \end{equation*}
        Indeed, let $j \in \N$. If there are $a_j \in (a_j'+2\eps, a_j'+4\eps)$ and $b_j \in (b_j'-4\eps, b_j'-2\eps)$ such that $s_j$ and $\overline{f}_j$ coincide at $a_j$ and $b_j$, then we are done, since the convexity of $s_j$ implies 
        \begin{equation*}
            \lip(s_j) \leq \frac{2 \norm{s_j}_\infty}{\min\{a_j-a_j', b_j'-b_j\}} \leq \frac{\norm{s_j}_\infty}{\eps}.
        \end{equation*}
        Assume that there is no $a_j$ as above.
        Then, since $s_j$ is the lower convex envelope of $\overline{f}_j$, $\overline{f}_j(t) > s_j(t)$ for every $t \in (a_j'+2\eps, a_j'+4\eps)$.
        We can take $a_j$ to be the first point before $a_j+2\eps$ where $s_j$ and $\overline{f}_j$ coincide. Observe that $a_j$ exists and $a_j \geq a_j'$.
        Then take $c_j$ as the first point after $a_j'+4\eps$ where $s_j$ and $\overline{f}_j$ coincide. In the interval $(a_j,c_j)$ the graph of the function $s_j$ is a segment with slope $|s_j'| \leq \frac{\norm{\overline{f}_j}_\infty}{\eps}$.
        The same argument holds for $b_j$, so we are done. 

        Now we have $s_j: [a_j,b_j] \to \R$, which is the lower convex envelope of $\overline{f}_j$ and with the above properties. 
        We extend it, without renaming it, to $[-\sigma, \sigma]$ by setting 
        \begin{align*}
            &s_j(t) = s_j(a_j) \quad \forall t \in [-\sigma, a_j), \\
            &s_j(t) = s_j(b_j) \quad \forall t \in (b_j,\sigma].
        \end{align*}
    
        By assumption there exists $h^j\in C^2(\R)$ such that $h^j(0)=(h^j)'(0)=0, \norm{(h^j)''}_\infty \leq C$ for some constant $C>0$, for every $j \in \N$, and  $\Omega^j \cap B_{r_j} = \{(x_1,x_2) \in B_{r_j}: x_2 > h^j(x_1) - D_j\}$
        and hence, setting $h_j(x_1) = r_j^{-1} h(r_j x_1)$,
        \begin{equation*}
            \textstyle\Omega_j \cap B_1 = \{(x_1,x_2) \in B_1: x_2 > h_j(x_1)- \frac{D_j}{r_j}\}, 
        \end{equation*}
        Moreover, it is easy to see that $\norm{h_j}_\infty \leq C r_j$ and $\norm{h_j'}_\infty \leq C r_j$ for some constant $C>0$. 

        To avoid technical details, we assume that $s_j \geq h_j-\frac{D_j}{r_j}$. If not, consider $\max\{s_j,h_j-\frac{D_j}{r_j}\}$ in place of $s_j$.
        We define the sets (see Figure \ref{fig: regions, competitor, lower convex envelope})
        \begin{figure}[ht]
            \begin{tikzpicture}[>=stealth,scale=1]
                \def\r{1.3}
                \draw[->](-3.8*\r,0)--(3.8*\r,0) node[right] {$x_1$};
                \draw[->](0,-1.2*\r)--(0,2.6*\r) node[above] {$x_2$};
                \def\aj{-2*\r}
                \def\bj{2.2*\r}
                \draw[very thin, gray, shift={(\aj,0)}] (0,0.1)--(0,-0.1);
                \draw[very thin, gray, shift={(\bj,0)}] (0,0.1)--(0,-0.1);
                \draw[dotted][domain=(-3*\r):(3*\r)] plot (\x, -0.7*\r);
                \def\a{0.02}
                \def\djoverrj{0.7*\r}
                \draw[thick][domain=-2.5*\r:2.5*\r] plot (\x,\a*\x*\x - \djoverrj);
                \draw[thick][dashed][domain=-3*\r:-2.5*\r] plot (\x, \a*\x*\x- \djoverrj);
                \draw[thick][dashed][domain=2.5*\r:3*\r] plot (\x, \a*\x*\x- \djoverrj) node[right] {$h_j-\frac{D_j}{r_j}$};
                \def\sjaj{0.2*\r}
                \def\sjbj{0.3*\r}
                \def\epsj{0.3*\r}
                \coordinate (A) at (\aj, \sjaj);
                \coordinate (B) at (\bj, \sjbj);
                \draw[thick, blue] (A) to[out=-10, in=200] (B) node[right] {$s_j$};
                \coordinate (C) at (\aj+\epsj, 2*\epsj);
                \coordinate (D) at (\bj-\epsj, 2*\epsj);
                \draw [thick] (A)--(C); 
                \draw [thick] (C)--(D);
                \draw [thick] (D)--(B);
                \def\s{6*\epsj}
                \draw (\aj,\s)--(\bj,\s);
                \draw (\aj,\s)--(\aj, \a*\aj*\aj - \djoverrj);
                \draw (\bj,\s)--(\bj, \a*\bj*\bj - \djoverrj);
                \node[above right] at (0,\s) {$\frac{\sigma}{2}$};
                \node[below left] at (\aj,0) {$a_j$};
                \node[below right] at (\bj,0) {$b_j$};
                \draw [dashed] (\aj+\epsj,\s) -- (\aj+\epsj,{\a*(\aj+\epsj)*(\aj+\epsj) - \djoverrj});
                \draw [dashed] (\bj-\epsj,\s) -- (\bj-\epsj,{\a*(\bj-\epsj)*(\bj-\epsj) - \djoverrj});
                \node [below right] at (\aj+\epsj,0) {\footnotesize $a_j+\eps_j$};
                \node [below left] at (\bj-\epsj,0) {\footnotesize $b_j-\eps_j$};
                \node [above] at ({(\bj-\epsj)/2}, 2*\epsj) {$l_j$};
                \node [right] at (0, {2*\epsj+(\s-2*\epsj)/2}) {$\Sigma_j^+$};
                \node [right] at (0, \epsj) {$\Sigma_j^0$};
                \node [right] at (0, {-\djoverrj/2}) {$\Sigma_j^-$};
            \end{tikzpicture}
            \caption{}
            \label{fig: regions, competitor, lower convex envelope}
        \end{figure}
        \begin{align*}
            &Q_j = [a_j,b_j] \times [-\frac{\sigma}{2}, \frac{\sigma}{2}] \cap \Omega_j, \\
            &\Sigma_j^0 = \{(x_1,x_2) \in Q_j: s_j(x_1) \leq x_2 \leq l_j(x_1)\}, \\
            &\Sigma_j^+ = \{(x_1,x_2) \in Q_j: l_j(x_1) < x_2 \leq \frac{\sigma}{2}\}, \\
            &\Sigma_j^- = \{(x_1,x_2) \in Q_j: h_j(x_1)-\frac{D_j}{r_j} \leq x_2 < s_j(x_1)\}, \\
            &L_j:= ([a_j, a_j+\eps_j) \cup (b_j -\eps_j, b_j]) \times [-\sigma,\sigma]
        \end{align*}
        where 
        \begin{equation*}
            l_j(x_1) = 
            \begin{cases}
                s_j(a_j) + \frac{2\eps_j - s_j(a_j)}{\eps_j}(x_1-a_j) &\text{ if } x_1 \in [a_j, a_j+\eps_j), \\
                2\eps_j &\text{ if } x_1 \in [a_j+\eps_j, b_j-\eps_j], \\
                s_j(b_j) - \frac{2\eps_j-s_j(b_j)}{\eps_j}(x_1-b_j) &\text{ if } x_1 \in (b_j-\eps_j, b_j].
            \end{cases}
        \end{equation*}
        We define $\Phi_j: Q_j \to Q_j$ as follows.
        \begin{equation*}
            \Phi_j(x_1,x_2) = (x_1, \varphi_j(x_1,x_2)), 
        \end{equation*}
        where 
        \begin{equation*}
            \varphi_j(x_1,x_2) = 
            \begin{cases}
               s_j(x_1) &\text{ if } (x_1,x_2) \in \Sigma_j^0, \\
               \frac{(\frac{\sigma}{2}-s_j(x_1))x_2+\frac{\sigma}{2}(s_j(x_1)-l_j(x_1))}{\frac{\sigma}{2}-l_j(x_1)} &\text{ if } (x_1,x_2) \in \Sigma_j^+, \\
               x_2 &\text{ if } (x_1,x_2) \in \Sigma_j^-.
            \end{cases}
        \end{equation*}
        \begin{claim} For sufficiently large $j$ there exists $\tilde{v}_j:Q_j \to \R$ such that
            \begin{enumerate}
                \item $\tilde{v}_j = v_j$ in $Q_j \setminus \Sigma_j^-$,
                \item $\tilde{v}_j = v_j$ su $\partial Q_j \cap \Sigma_j^-$,
                \item $\tilde{v}_j = g_j$ su $\partial \Omega_j \cap Q_j$,
                \item $\int_{\Sigma_j^-} |\nabla \tilde{v}_j|^2 - \int_{\Sigma_j^-} |\nabla v_j|^2 \leq -\int_{\Sigma_j^- \cap \{x_1 \in (-\frac{\sigma}{2}, \frac{\sigma}{2})\}}|\nabla v_j|^2 + o(1)$.
            \end{enumerate}
        \end{claim}
        \begin{proof}[Proof of Claim]
        For simplicity we assume that the boundary is flat and the boundary datum is zero, that is $h_j=0$ and $g_j=0$. The general case is not difficult, and it is handled as in the proof of Theorem \ref{thm: blow-up limit}(iii).
        Let us define $$I_j:=[a_j,b_j] \setminus \pi_{\mathscr{V}_0}(\overline{K}_j \cap \Sigma_j^-). $$
        By Fubini's Theorem and Chebyshev's inequality, there exist $\alpha_j \in (a_j,-\frac{\sigma}{2}) \cap I_j$ and $\beta_j \in (\frac{\sigma}{2},b_j) \cap I_j$ such that 
        \begin{equation}
            \label{eq: Dirichlet energy decay, integral on vertical segment bounded}
            \int_{A_j} |\nabla v_j|^2 + \int_{B_j} |\nabla v_j|^2 \leq C(\sigma) 
        \end{equation}
        where $C(\sigma)>0$ is a constant depending on $\sigma$, and
        \begin{equation*}
            \textstyle
            A_j = \{\alpha_j\} \times [-\frac{D_j}{r_j},s_j(\alpha_j)], \quad B_j =  \{\beta_j\} \times [-\frac{D_j}{r_j},s_j(\beta_j)].
        \end{equation*}
        Moreover, by construction, 
        \begin{equation*}
            \textstyle
            v_j(\alpha_j,-\frac{D_j}{r_j}) = v_j(\beta_j,-\frac{D_j}{r_j}) = 0.
        \end{equation*}
        Assume that $s_j(\alpha_j) > -\frac{D_j}{r_j}$, otherwise the construction is considerably easier.
    
        To make computations less cumbersome, translate the plane in such a way that, in the new coordinates, 
        \begin{equation*}
            \textstyle
            A_j = \{0\} \times [0,s_j(0)+\frac{D_j}{r_j}].
        \end{equation*}
        We define 
        \begin{align*}
            &\overline{s}_j(t) := s_j(t)+\frac{D_j}{r_j},\\
            &\tau_j:= \overline{s}_j(0), \\
            &\Gamma_j := \{(x_1,x_2) : 0<x_1<\tau_j,\,  0<x_2<\overline{s}_j(x_1)\}, \\
            &\tilde{A}_j := \tau_j^{-1} A_j = \{0\} \times [0,1],\\
            &\tilde{\Gamma}_j := \tau_j^{-1} \Gamma_j = \{(x_1,x_2): 0<x_1<1, \,  0<x_2< \tau_j^{-1} \overline{s}_j(\tau_j x_1)\},\\
            &\overline{v}_j(x) := v_j(\tau_jx) \quad \forall x \in \tilde{\Gamma}_j,\\
            &\overline{v}_j(0,t) := \overline{v}_j(0,1) = v_j(0,\tau_j) \quad \forall t > 1, \\
            &\tilde{w}_j(x_1,x_2) = (1-x_1) \overline{v}_j(0,x_2) \quad \forall x \in \tilde{\Gamma}_j, \\
            & \tilde{v}_j(x) = \tilde{w}_j(\tau_j^{-1}x) \quad \forall x \in \Gamma_j.
        \end{align*}
        Observe that, since $v_j(0,0) = 0$, by Hölder's inequality and by \eqref{eq: Dirichlet energy decay, integral on vertical segment bounded}, 
        \begin{equation*}
            v_j(0,x_2)^2 \leq C(\sigma) \tau_j \quad \forall x_2 \in (0,\tau_j) 
        \end{equation*}
        and hence 
        \begin{equation*}
            \int_{A_j} v_j^2 \lesssim \tau_j^2.
        \end{equation*}
        Moreover, 
        \begin{equation*}
            \int_{[0,1]^2} |\nabla \tilde{w}_j|^2 \leq \tau_j^{-1} \int_{A_j} v_j^2 + \tau_j \int_{A_j} \partial_2 v_j^2 \lesssim \tau_j.
        \end{equation*}
        Furthermore, if $\tilde{\Gamma_j} \setminus [0,1]^2 \neq \emptyset$ , by the convexity of $\overline{s}_j$,
        \begin{equation*}
            \int_{\tilde{\Gamma}_j \setminus [0,1]^2} \leq v_j(0,\tau_j)^2 \frac{1}{2}(\tau_j^{-1}\overline{s}_j(\tau_j)-1) \lesssim \overline{s}_j(\tau_j).
        \end{equation*}
        In conclusion, we have 
        \begin{equation*}
            \int_{\Gamma_j} |\nabla \tilde{v}_j|^2 = \int_{\tilde{\Gamma}_j} |\nabla \tilde{w}_j|^2 \lesssim \tau_j + \overline{s}_j(\tau_j).
        \end{equation*}
        Going back to the old coordinates, in which we had 
        \begin{equation*}
            \textstyle
            A_j = \{\alpha_j\} \times [-\frac{D_j}{r_j},s_j(\alpha_j)],
        \end{equation*}
        the above estimate becomes 
        \begin{equation*}
            \int_{\Sigma_j^- \cap \{x_1 \in (\alpha_j, \alpha_j+\tau_j) \}} |\nabla \tilde{v}_j|^2 \lesssim  s_j(\alpha_j) + s_j(\alpha_j+\tau_j) +2\frac{D_j}{r_j},
        \end{equation*}
        where $\tau_j = s_j(\alpha_j) + \frac{D_j}{r_j}$.
    
        Similarly, we can construct $\tilde{v}_j$ on $\Sigma_j^- \cap \{x_1 \in (\beta_j - \tau_j', \beta_j)\}$ such that 
        \begin{equation*}
            \int_{\Sigma_j^- \cap \{x_1 \in (\beta_j-\tau_j'), \beta_j\}} |\nabla \tilde{v}_j|^2 \lesssim  s_j(\beta_j) + s_j(\beta_j-\tau_j') +2\frac{D_j}{r_j},
        \end{equation*}
        where $\tau_j' = s_j(\beta_j)+ \frac{D_j}{r_j}$.
    
        Finally, we set $\tilde{v}_j = 0$ in $\Sigma_j^- \cap \{x_1 \in (\alpha_j+\tau_j, \beta_j-\tau_j')$, and $\tilde{v}_j = v_j$ outside $\Sigma_j^- \cap \{x_1 \in (\alpha_j,\beta_j)\}$. We conclude by observing that $\norm{s_j}_{L^\infty(a_j,b_j)} \lesssim d_j^\frac{1}{2} \to 0$ and $\frac{D_j}{r_j} \to 0$.
        The first estimate follows from the fact that $s_j \leq \overline{f}_j$ and from estimate \eqref{eq: bound on overline(f)_j}.
        \end{proof}
    
        We define the competitor $(K_j', w_j)$ as follows. 
        \begin{align*}
            &K_j' = (\overline{K}_j \setminus Q_j) \cup (\graph(s_j) \cap Q_j) \cup \Phi_j(\overline{K}_j \cap \Sigma_j^+), \\
            &w_j = \tilde{v}_j \circ \Phi_j^{-1}.
        \end{align*}
        \begin{claim} For sufficiently large $j$, taking a subsequence if necessary, 
            \begin{equation*}
                \int_{Q_j} |\nabla w_j|^2 - \int_{Q_j} |\nabla v_j|^2 \leq -\frac{3}{4}\theta.
            \end{equation*}
        \end{claim}
        \begin{proof}[Proof of claim]
            Observe that 
            \begin{equation*}
                \int_{Q_j} |\nabla w_j|^2 = \int_{Q_j \setminus \Sigma_j^-} |\nabla (v_j \circ \Phi_j^{-1})|^2 + \int_{\Sigma_j^-} |\nabla \tilde{v}_j|^2
            \end{equation*}
            and
            \begin{equation*}
                \begin{split}
                    \int_{Q_j \setminus \Sigma_j^-} |\nabla (v_j \circ \Phi_j^{-1})|^2  &\leq C \int_{R_j \cap \Omega_j \setminus \overline{K}_j} |\nabla v_j|^2 + (1+o(1))\int_{\Sigma_j^+ \setminus R_j} |\nabla v_j|^2 \\ 
                    &\leq C\eps + (1+o(1))\int_{\Sigma_j^+ \setminus R_j} |\nabla v_j|^2.
                \end{split}
            \end{equation*}
            Therefore 
            \begin{equation*}
                \begin{split}
                    \int_{Q_j} |\nabla w_j|^2 - \int_{Q_j} |\nabla v_j|^2 \leq C\eps + o(1) + \int_{\Sigma_j^-} |\nabla \tilde{v}_j|^2 - \int_{\Sigma_j^-} |\nabla v_j|^2 - \int_{\Sigma_j^0} |\nabla v_j|^2
                \end{split}
            \end{equation*}
            for some constant $C$. Using the previous claim, we get
            \begin{equation*}
                \int_{Q_j} |\nabla w_j|^2 - \int_{Q_j} |\nabla v_j|^2 \leq C\eps + o(1) -  \int_{(-\frac{\sigma}{2}, \frac{\sigma}{2}) \times (-\frac{\eps_j}{2}, \frac{\eps_j}{2}) \cap \Omega_j \setminus \overline{K}_j} |\nabla v_j|^2.
            \end{equation*}
            We conclude by choosing $\eps$ sufficiently small and by \eqref{eq: concentration of energy}.
        \end{proof}
        \begin{claim} For sufficiently large $j$, taking a subsequence if necessary,
          \begin{equation*}
            \HH^1(K_j'\cap Q_j) - \HH^1(\overline{K}_j \cap Q_j) \leq \frac{1}{4}\theta d_j
          \end{equation*}
        \end{claim}
        \begin{proof}[Proof of Claim]
            Recall that, by construction, we have $\lip(s_j) \leq \frac{\norm{\overline{f}_j}_\infty}{\eps}$.
            Then, by \eqref{eq: bound on overline(f)_j}, we have 
            \begin{equation*}
                \lip(s_j) \lesssim \frac{d_j^{\frac{1}{2}}}{\eps}.
            \end{equation*}
            We define 
            \begin{align*}
                &I_j^1 = \{t \in [a_j,b_j]: (t,\overline{f}_j(t))\in\overline{K}_j, \, (\{t\}\times\R)\cap\overline{K}_j \subset \Gamma(\overline{f}_j) \}, \\
                &I_j^2 = [a_j,b_j]\cap \pi_{\mathscr{V}_0}(\overline{K}_j\cap Q_j) \setminus I_j^1, \\
                &I_j^3 = [a_j,b_j] \setminus \pi_{\mathscr{V}_0}(\overline{K}_j\cap Q_j). 
            \end{align*}
            Since $\lip(\overline{f}_j) \leq 1$ and $\overline{f}_j = s_j$ on $\{a_j, b_j\}$, then  $\Gamma(\overline{f}_j) \cap Q_j \subset \Sigma_j^0$ for sufficiently large $j$.
            Therefore, since $L_j \subset R_j$, we have, for some constant $C$ and choosing $\eps$ small enough,  
            \begin{equation*}
                \begin{split}
                    \HH^1(\Phi_j(\overline{K}_j \cap \Sigma_j^+)) &\leq C\HH^1(\overline{K}_j \cap \Sigma_j^+) \\
                    &= C\HH^1(\overline{K}_j \cap \Sigma_j^+ \cap L_j) \\
                    &\leq C\HH^1((\Gamma(\overline{f}_j) \triangle \overline{K}_j )\cap L_j)\\
                    &\leq C\HH^1((\Gamma(\overline{f}_j) \triangle \overline{K}_j )\cap R_j)\\
                    &\leq C \eps d_j < \frac{1}{8} \theta d_j.
                \end{split}
            \end{equation*}
            Now we claim that
            \begin{equation*}
                \frac{\HH^1(\Gamma(s_j))-\HH^1(\overline{K}_j \cap Q_j) }{d_j} \leq \frac{1}{8}\theta
            \end{equation*}
            for sufficiently large $j$ and taking a subsequence if necessary.
            We have 
            \begin{equation*}
                \begin{split}
                    &\HH^1(\Gamma(s_j))-\HH^1(\overline{K}_j \cap Q_j) \\
                    &\leq \int_{a_j}^{b_j}\left(\sqrt{1+s_j'^2} -1\right)- \int_{I_j^1}\left(\sqrt{1+\overline{f}_j'^2}-1\right) + \mathcal{L}^1(I_j^3)
                \end{split}
            \end{equation*}
            By Lemma \ref{lmm: Lipschitz approximation}\ref{item: Lipschitz approximation, D/r < 1}(v), for sufficiently large $j$, $\mathcal{L}^1(I_j^3) \leq \frac{1}{8}\theta d_j$. Then it suffices to check that 
            \begin{equation*}
                \limsup_{j\to\infty}
                \left(  \int_{a_j}^{b_j}\frac{\sqrt{1+s_j'^2} -1}{d_j}- \int_{I_j^1}\frac{\sqrt{1+\overline{f}_j'^2}-1}{d_j}\right) \leq 0.
            \end{equation*}
            Taking a subsequence if necessary, we can assume that $a_j \to a$ and $b_j \to b$
            with $a \in [-\sigma+2\eps, -\frac{3}{4}\sigma+4\eps]$, $b \in [\frac{3}{4}\sigma-4\eps, \sigma-2\eps]$.
            Since $s_j$ is the lower convex envelope of $\overline{f}_j$, we have 
            \begin{align*}
                \int_{a_j}^{b_j} s_j'^2 \leq \int_{a_j}^{b_j} \overline{f}_j'^2.
            \end{align*}
            We define $\tilde{s}_j = d_j^{-\frac{1}{2}}s_j$ and $\tilde{f}_j = d_j^{-\frac{1}{2}} \overline{f}_j$.
            Then, by estimates on $\overline{f}_j$ given by Lemma \ref{lmm: Lipschitz approximation}, we have
            \begin{align*}
                &\norm{\tilde{s}_j}_{L^\infty(-\sigma,\sigma)} \leq C, \quad  \norm{\tilde{s}_j'}_{L^2(-\sigma,\sigma)}\leq C, \\
                &\norm{\tilde{f}_j}_{L^\infty(-\sigma,\sigma)} \leq C, \quad  \norm{\tilde{f}_j'}_{L^2(-\sigma,\sigma)}\leq C.
            \end{align*}
            Then, taking a subsequence if necessary, $\tilde{s}_j \rightharpoonup \tilde{s}$, $\tilde{f}_j \rightharpoonup \tilde{f}$ in $W^{1,2}((-\sigma,\sigma))$ and hence $\tilde{s}_j \to \tilde{s}$, $\tilde{f}_j \to \tilde{f}$
            uniformly in $[-\sigma,\sigma]$.
            Moreover, $\tilde{s}$ is the lower convex envelope of $\tilde{f}$, and hence 
            \begin{equation*}
                \int_a^b \tilde{s}'^2 \leq \int_a^b \tilde{f}'^2.
            \end{equation*}
            \begin{claim}
                \begin{equation*}
                    \liminf_{j\to\infty}\int_{I_j^1}\frac{\sqrt{1+\overline{f}_j'^2}-1}{d_j} \geq \frac{1}{2} \int_a^b \tilde{f}'^2.
                \end{equation*}
            \end{claim}
            \begin{proof}[Proof of Claim]
                \begin{equation*}
                    \begin{split}
                        \int_{I_j^1}\frac{\sqrt{1+\overline{f}_j'^2}-1}{d_j} &\geq \int_{I_j^1 \cap \{\overline{f}_j'\leq d_j^\frac{1}{4}\}}\frac{\sqrt{1+\overline{f}_j'^2}-1}{d_j} \\
                        &\geq \frac{1}{2}\int_{I_j^1 \cap \{\overline{f}_j'\leq d_j^\frac{1}{4}\}} \tilde{f}_j'^2 + o(1)
                    \end{split}
                \end{equation*}
                By Chebyshev's inequality we have 
                \begin{equation*}
                    \mathcal{L}^1([a_j,b_j]\setminus( I_j^1 \cap \{\overline{f}_j'\leq d_j^\frac{1}{4}\})) = o(1).
                \end{equation*} 
                and hence 
                \begin{equation*}
                   \tilde{f}_j' \chi_{I_j^1 \cap \{\overline{f}_j'\leq d_j^\frac{1}{4}\}} \rightharpoonup \tilde{f}_j' \quad \text{in } L^2
                \end{equation*}
                and the result follows from the lower semicontinuity of the $L^2$ norm with respect to weak convergence.
            \end{proof}
            \begin{claim}
                \begin{equation*}
                    \lim_{j \to \infty} \int_{a_j}^{b_j} \tilde{s}_j'^2 = \int_{a}^{b} \tilde{s}'^2
                \end{equation*}
            \end{claim}
            \begin{proof}[Proof of Claim]
                Recall that $\tilde{s}_j$ is convex on $(a_j,b_j)$ and $\tilde{s}_j ' = 0$ in $(-\sigma, a_j) \cup (b_j,\sigma)$. Then, approximating each $\tilde{s}_j$ by smooth and convex functions, it is easy to see that $\tilde{s}_j'$ is uniformly bounded in $BV(-\sigma,\sigma)$. Therefore, for a subsequence that we do not rename, $\tilde{s}_j' \to \tilde{s}'$ in $L^1((-\sigma,\sigma))$ and $a.e.$. Since $\tilde{s}_j'$ is uniformly bounded, so is $\tilde{s}'$. Then 
                \begin{equation*}
                    \left|\int_{-\sigma}^\sigma \tilde{s}_j'^2 - \int_{-\sigma}^\sigma \tilde{s}'^2\right| = \left|\int_{-\sigma}^\sigma (\tilde{s}_j'-\tilde{s}')(\tilde{s}_j'+\tilde{s}') \right| \leq C \int_{-\sigma}^{\sigma}|\tilde{s}_j' - \tilde{s}'| \to 0. 
                \end{equation*}
                Moreover $\tilde{s}' = 0$ a.e. in $(-\sigma,a)\cup(b,\sigma)$, and hence the above computation is exactly what we wanted to prove.
            \end{proof}
            \begin{claim}
                \begin{equation*}
                    \lim_{j\to\infty}  \int_{a_j}^{b_j}\frac{\sqrt{1+s_j'^2} -1}{d_j} = \frac{1}{2} \int_a^b \tilde{s}'^2
                \end{equation*}
            \end{claim}
            \begin{proof}[Proof of Claim]
                Recall that, by construction, $\norm{s_j'}_\infty \lesssim d_j^{\frac{1}{2}}$.
                This implies that
                \begin{equation*}
                    \int_{a_j}^{b_j}\frac{\sqrt{1+s_j'^2} -1}{d_j} = \frac{1}{2} \int_{a_j}^{b_j} \tilde{s}_j'^2 + o(1).
                \end{equation*}
                We conclude by the previous claim.
            \end{proof}
        \end{proof}
        By the above claims, we get 
        \begin{equation*}
            F(K_j', w_j) - F(\overline{K}_j, v_j) \leq -\frac{1}{2}\theta < 0
        \end{equation*}
        for sufficiently large $j$.
        \item The proof of \ref{item: Dirichlet, case 2} is the same as the proof of \ref{item: Dirichlet, case 1}
        \item The proof of \ref{item: Dirichlet, case 3} is the same as the proof of \ref{item: Dirichlet, case 1}, using Lemma \ref{lmm: tilt lemma} \ref{item: tilt lemma, case D/r < 1} and the fact that 
        \begin{equation*}
            \frac{D}{r} \leq B r
        \end{equation*}
        \item The proof of \ref{item: Dirichlet, case 4} is analogous to the proof of \ref{item: Dirichlet, case 1}, with some technical complications due to the fact that we have $R_\theta(\mathscr{V})$ instead of $\mathscr{V}$, using Lemma \ref{lmm: tilt lemma} \ref{item: tilt lemma, case > b^1/3}.
    \end{enumerate}
\end{proof}
\section{Generalized minimizers on the upper half-plane}
\label{section: generalized minimizers on the upper half-plane}
In this section, we study generalized minimizers on the open upper half-plane $U = \{x_2 >0\}$.

First, let us give a brief proof of Theorem \ref{thm: compactness for generalized minimizers}.
\begin{proof}[Proof of Theorem \ref{thm: compactness for generalized minimizers}]
    The proof is the same as the proof of Theorem \ref{thm: blow-up limit}, slightly easier since $\partial U$ is flat and the boundary datum is zero. Moreover, since the boundary datum is zero, we do not need $r_j \to 0$ in the proof, and we can also take $r_j \to \infty$.
\end{proof}

In the following proposition we obtain the analogous for our problem of the famous David-Léger-Maddalena-Solimini identity \cite[Proposition 3.5]{david-leger}.
\begin{prop}[David-Léger-Maddalena-Solimini for the boundary problem]
    \label{prop: David-Legér-Maddalena-Solimini}
    Let $(K,u)$ be a generalized minimizer on $U$ as in Definition \ref{def: admissible competitor and minimizer on a half-plane}. Then for every $x \in \partial U$ and for a.e. $r>0$ we have 
    \begin{equation*}
        \frac{1}{r}\HH^1(K \cap B_r(x)) = \int_{\partial B_r(x) \cap U \setminus K} \left( \left(\frac{\partial u}{\partial \tau}\right)^2 - \left(\frac{\partial u}{\partial \nu}\right)^2\right) + \sum_{p \in \partial B_r(x)\cap K} e(p)\cdot n(p),
    \end{equation*}
    where $n(p)$ is the unit outer normal vector to $\partial B_r(x)$ at point $p$, and $e(p)$  is the unit tangent vector to $K$ at point $p$ such that $e(p)\cdot n(p)>0$.
\end{prop}
\begin{proof}
   The proof follows from the ideas of \cite[Proposition 2.5.4]{delellis-focardi}.
\end{proof}

Below we prove a Lemma of harmonic extension similar to \cite[Lemma 4.10]{david-leger}. We will need it in the following. 
\begin{lmm}[Harmonic extension on circular sector]
    \label{lmm: harmonic extension circular sector}
    Let $r>0$, $\alpha \in (0,\pi)$, $\gamma = \{ (r \cos\theta, r\sin\theta): \theta \in [0,\alpha]\}$, $U = \{ (\rho\cos\theta,\rho\sin\theta): \rho \in (0,1), \theta \in (0,\alpha) \}$. Let $g \in W^{1,2}(\gamma)$ be such that $g(r,0)=0$. Then there exists $u: \overline{U} \to \R$ such that $u$ is harmonic in $U$, $u=g$ on $\gamma$, $u(t,0)=0$ for every $t \in [0,r]$ and 
    \begin{equation*}
        \int_U |\nabla u|^2 \leq \frac{2\alpha r}{\pi} \int_\gamma \left(\frac{\partial g}{\partial \tau} \right)^2
    \end{equation*}
\end{lmm}
\begin{proof}
    Without loss of generality we can assume $r=1$.

    Let us define $\hat{g}: [-2\alpha, 2\alpha] \to \R$ as follows. For $\theta \in [0,\alpha]$ we set $\hat{g}(\theta) = g(\cos\theta,\sin\theta)$, $\hat{g}(-\theta) = -\hat{g}(\theta)$, $\hat{g}(\theta+\alpha) = \hat{g}(-\theta + \alpha)$, $\hat{g}(-\theta-\alpha) = \hat{g}(\theta - \alpha)$.
    Observe that $g$ is well defined since $g(1,0)=0$. Now we extend $\hat{g}$ to $\R$ as a $4\alpha$-periodic function. By construction $\hat{g} \in W^{1,2}(\R)$ and $\hat{g}$ is odd. 
    We write the Fourier series of $\hat{g}$
    \begin{equation*}
        \hat{g}(\theta) = \sum_{k=1}^{\infty} b_k \sin k \frac{\pi}{2\alpha} \theta
    \end{equation*}
    and we define, for $(\rho,\theta) \in [0,1] \times [0,\alpha]$, 
    \begin{equation*}
        v(\rho,\theta) = \sum_{k=1}^{\infty} b_k \rho^{k\frac{\pi}{2\alpha}} \sin k \frac{\pi}{2\alpha} \theta
    \end{equation*}
    and $u(\rho\cos\theta,\rho\sin\theta) = v(\rho,\theta)$. Then $u$ is harmonic in $U$, $u=g$ on $\gamma$ and $u(t,0) = 0$ for every $t \in [0,1]$. It remains to verify the energy estimate.

    First observe that by the symmetries of $\hat{g}$ we have $b_{2n} = 0$ for every $n \in \N$.
    Finally, we compute
    \begin{equation*}
        \begin{split}
            \int_U |\nabla u|^2 
            &= \frac{\pi}{2\alpha} \sum_{\myatop{n,m=1}{n,m \text{ odd}}}^{\infty} b_n b_m \frac{nm}{n+m} \int_0^\alpha \cos\left[(n-m)\frac{\pi}{2\alpha}\theta \right] d\theta \\ 
            &= \frac{\pi}{4} \sum_{n=1}^\infty n b_n^2 \leq \frac{\pi}{4} \sum_{n=1}^\infty n^2 b_n^2 = \frac{2\alpha}{\pi}  \int_\gamma \left(\frac{\partial g}{\partial \tau} \right)^2.
        \end{split}
    \end{equation*}    
\end{proof}
In the case of an half-disk we need a slightly different extension, which we obtain in the following Lemma.
\begin{lmm}[Harmonic extension on half-disk]
    \label{lmm: harmonic extension on half-disk}
    Let $r>0$, $\gamma = \{ (r \cos\theta, r\sin\theta): \theta \in [0,\pi]\}$, $U = \{ (\rho\cos\theta,\rho\sin\theta): \rho \in (0,1), \theta \in (0,\pi) \}$. Let $g \in W^{1,2}(\gamma)$ be such that $g(r,0)=g(-r,0)=0$. Then there exists $u: \overline{U} \to \R$ such that $u$ is harmonic in $U$, $u=g$ on $\gamma$, $u(t,0)=0$ for every $t \in [-r,r]$ and 
    \begin{equation*}
        \int_U |\nabla u|^2 \leq r \int_\gamma \left(\frac{\partial g}{\partial \tau} \right)^2
    \end{equation*}
\end{lmm}
\begin{proof}
    Without loss of generality we can assume $r=1$.
    Let us define $\hat{g}: [-\pi, \pi] \to \R$ as follows. For $\theta \in [0,\pi]$ we set $\hat{g}(\theta) = g(\cos\theta,\sin\theta)$ and $\hat{g}(-\theta) = -\hat{g}(\theta)$.
    Observe that $g$ is well defined since $g(1,0)=0$. Now we extend $\hat{g}$ to $\R$ as a $2\pi$-periodic function. By construction $\hat{g} \in W^{1,2}(\R)$ and $\hat{g}$ is odd. 
    The result follows by computations analogous to those in the proof of the previous Lemma, starting from the Fourier series for $\hat{g}$.
\end{proof}

The following proposition will be used in the proof of Proposition \ref{prop: Bonnet monotonicity} to exclude possibility of a cracktip (see \cite[Definition 1.3.2]{delellis-focardi}) as a generalized minimizer on $U$.
\begin{prop}
    \label{prop: generalized minimizer with K in partial Omega containing an interval}
    Let $(K,u)$ be a generalized minimizer on $U$ as in Definition \ref{def: admissible competitor and minimizer on a half-plane}.
    Suppose that $K\subset \partial U$ and $K$ contains a non empty interval $ I = (a,b) \times \{0\}$. Then $K = \partial U$.
\end{prop}
\begin{proof}
    Without loss of generality assume that $I = (-R,R)$ for some $R>0$. Let $\varphi \in C^\infty_c(B_R)$ be such that $\varphi \geq 0$. We define
    \begin{equation*}
        \eta(x) = (0,\varphi(x)) \quad \forall x \in \R^2.
    \end{equation*}
    Let $\Phi_t$ be defined as in Lemma \ref{lmm: first derivative of the energy for generalized minimizers}. It is easy to see that 
    \begin{equation*}
        \Phi_t^{-1}(s,0) \in \{x_2 \leq 0\} \quad \forall s \in \R, \, \forall t\geq0
    \end{equation*}
    Since $u = 0$ in $\{x_2 \leq 0\} \setminus K$, this implies that $u_t = u \circ \Phi_t^{-1} = 0$ in $\{x_2=0\}\setminus \Phi_t(K)$.
    In particular, for sufficiently small $t\geq0$, $(K_t, u_t) = (\Phi_t(K), u \circ \Phi_t^{-1})$ is a admissible competitor for $(K,u)$ in $B_R$ according to Definition \ref{def: admissible competitor and minimizer on a half-plane}. Since $(K,u)$ is a generalized minimizer, this implies that $\derivative{}{t} E(K_t,u_t, B_R)|_{t=0} \geq 0$ and, by Lemma \ref{lmm: first derivative of the energy for generalized minimizers}, this means that
    \begin{equation*}
        \int_{B_R \cap U \setminus K} (|\nabla u|^2 \diverg \eta - 2 \nabla u^T D\eta \, \nabla u) + \int_{K \cap B_R} e^T D\eta \, e \, \dd{\HH^1} \geq 0.
    \end{equation*}
    Observe that $e(x) = e_1 = (1,0)$ for every $x \in K$ since $K \subset \partial U = \{x_2=0\}$. 
    This implies that $e^T D\eta \, e = 0$ in $K \cap B_R$, so that the second integral in the above formula vanishes.
    Moreover, in $U \setminus K $ we have, since $u$ is harmonic by \eqref{eq: Euler-Lagrange },
    \begin{equation*}
        |\nabla u|^2 \diverg \eta - 2 \nabla u^T D\eta \, \nabla u =\diverg(|\nabla u|^2 \varphi e_2 -2 \varphi \partial_2 u  \nabla u ).
    \end{equation*}
    Then, by the Divergence Theorem, 
    \begin{equation*}
        \int_{\partial (B_R \cap U \setminus K)} \left(|\nabla u|^2 \varphi \nu_2 -2 \varphi \partial_2 u  \frac{\partial u}{\partial \nu} \right) \geq 0,
    \end{equation*}
    where $\nu$ is the outward pointing unit vector field normal to the boundary.
    Now recall that $K \subset \partial U$. Therefore, $B_R \cap U \setminus K = B_R \cap U $ and $\partial (B_R \cap U \setminus K) = (\partial B_R \cap U) \cup I $.
    By the Euler-Lagrange equations \eqref{eq: Euler-Lagrange }, we have $\frac{\partial u}{\partial \nu} = 0$ on $K$. Moreover, $\varphi = 0$ on $\partial B_R \cap U$. In conclusion, we get 
    \begin{equation*}
       - \int_I |\nabla u|^2\varphi \geq 0 
    \end{equation*}
    for every $\varphi \in C^\infty_c(B_R)$, $\varphi \geq 0$. This implies that $\nabla u = 0$ on $I$. 
    \begin{claim}
        $u$ is constant in $B_R \cap U$.
    \end{claim}
    \begin{proof}
        We have $u$ harmonic in $B_R \cap U = B_R \cap \{x_2 >0\}$. 
        Translating $u$, and without renaming it, we can assume that $u=0$ on $I$.
        We define $v$ to be the Schwarz's reflection of $u$, that is 
        \begin{equation*}
            v(x_1,x_2) = \begin{cases}
                u(x_1,x_2) &\text{ if } x_2 \geq 0 \\ 
                -u(x_1,-x_2) &\text{ if } x_2 < 0.
            \end{cases}
        \end{equation*}
        Then $v$ is harmonic in $B_R$ (see \cite[Problem 2.4]{gilbarg-trudinger}). Therefore there exists a function $f$ that is holomorphic on $B_R$ and such that $v = \operatorname{Re}f $.
        Moreover, $|f'(x_1+ix_2)| = |\nabla v(x_1,x_2)|$ for every $(x_1,x_2) \in B_R$. In particular, $f' = 0$ on $I$. 
        Since the zero of $f'$ at $0$ is not isolated, $f' = 0$ in $B_R$ by the Identity Theorem \cite[Theorem 15.7]{priestley}. Therefore $f$ is constant in $B_R$ and so is $u$. 
    \end{proof}
        
    Since $K \subset \partial U$, we know that $u$ is harmonic in $U$.
    Then the above claim implies that that $u$ is constant in $U$. 
    Since $(K,u)$ is a generalized minimizer, either $u=0$ (and $K = \emptyset$) or $K=\partial U$. Since we know that $I \subset K$, we conclude that $K = \partial U$. 
\end{proof}
\subsection{Monotonicity formulae}
\label{section: monotonicity formuale}
In this section, we prove Propositions \ref{prop: david-leger monotonicity} and \ref{prop: Bonnet monotonicity}.
We start with the following proposition, which is a preliminary step in proving Proposition \ref{prop: david-leger monotonicity}. 
\begin{prop}
    \label{prop: david-leger monotonicity with angle}
    Assume that $(K,u)$, $x$ and $F$ are as Proposition \ref{prop: david-leger monotonicity}.

    For a.e. $r>0$ such that $\HH^0(K \cap \partial B_r(x)) \in [1,\infty)$ we have
    \begin{equation*}
        F'(r) \geq \min\left(1,3-\frac{2\alpha}{\pi}\right) \frac{1}{r}\int_{\partial B_r(x) \cap U \setminus K} |\nabla u|^2, 
    \end{equation*}
    where $\alpha$ is such that $\HH^1(I) \leq \alpha r$ for every $I$ connected component of $\partial B_r(x) \cap U \setminus K$ not contained in $I^*$, and $\HH^1(I) \leq \frac{\alpha}{2} r$ for every connected component $I \subset I^*$, with $I^*$ being the connected component of $\partial B_r(x) \setminus K$ that contains the point $(0,-r)$.

    In particular, $F'(r) \geq 0$ if $\HH^1(I) \leq \frac{3}{4}\pi r$ for every $I \subset I^*$. 
\end{prop}
\begin{proof}
    Assume without loss of generality that $x=0$.
    Let us define, for $r>0$,
    \begin{equation*}
        D(r) = \int_{B_r\cap U \setminus K} |\nabla u|^2, \quad
        l(r) = \HH^1(K \cap B_r), \quad
        N(r) = \HH^0(K \cap \partial B_r).        
    \end{equation*}
    Take $r>0$ such that $N(r)<\infty$, and observe that this holds for a.e. $r>0$ by the by the coarea formula \cite[Theorem 2.93, p. 101]{AFP}. By a consequence of the coarea formula \cite[Theorem 3.12, p. 140]{evans-gariepy} we have, for a.e. $r>0$, 
    \begin{equation*}
        D'(r) = \int_{\partial B_r \cap U \setminus K} |\nabla u|^2
    \end{equation*}
    and by the coarea formula \cite[Theorem 2.93, p. 101]{AFP} we have, for a.e. $r>0$, 
    \begin{equation*}
        l'(r) = \sum_{p \in \partial B_r \cap K} \frac{1}{e(p)\cdot n(p)},
    \end{equation*}
    with the same notation of Proposition \ref{prop: David-Legér-Maddalena-Solimini}.
    Therefore, for a.e. $r>0$, 
    \begin{equation*}
        r^2 F'(r) = 2r \int_{\partial B_r \cap U \setminus K} |\nabla u|^2 + r \sum_{p \in \partial B_r \cap K} \frac{1}{e(p)\cdot n(p)} -(2 D(r)+l(r)).
    \end{equation*}
    Then, by Proposition \ref{prop: David-Legér-Maddalena-Solimini}, 
    \begin{equation*}
        r^2 F'(r) = 3r E_\tau(r) + r E_\nu(r) + r \sum_{p \in \partial B_r \cap K} \left( \frac{1}{e(p)\cdot n(p)} + e(p)\cdot n(p) \right) - 2(D(r)+l(r)),
    \end{equation*}
    with 
    \begin{equation*}
        E_\tau(r) = \int_{\partial B_r \cap U \setminus K} \left(\frac{\partial u}{\partial \tau}\right)^2, \quad
        E_\nu(r) = \int_{\partial B_r \cap U \setminus K} \left(\frac{\partial u}{\partial \nu}\right)^2.
    \end{equation*}
    Since $|e(p)\cdot n(p)| \leq 1$, for a.e. $r>0$
    \begin{equation}
        \label{eq: proof of david-leger, main inequality}
        r^2 F'(r) \geq 3r E_\tau(r) + r E_\nu(r) + 2rN(r)- 2 E(K,u,B_r \cap \overline{U}).
    \end{equation}
    Let us fix some notation (see Figure \ref{fig: monotonicity, angles and arcs notation}). Let $N = N(r)+1$. Then $\partial B_r \cap U \setminus K = \cup_{j=1}^N I_j$ with 
    \begin{equation*}
        I_j = \{(r\cos\theta,r\sin\theta): \theta \in (\sum_{i=1}^{j-1} \alpha_i, \sum_{i=1}^{j} \alpha_i)\},
    \end{equation*}
    $\alpha_0:=0$, with the convention that $\alpha_1=0$ if $(r,0) \in K$ and $\alpha_N = 0$ if $(-r,0) \in K$.

    \begin{figure}[ht]
        \begin{tikzpicture}[>=stealth, scale=1]
            \draw[->](-3,0)--(3,0) node[right] {$x_1$};
            \draw[->](0,-2.6)--(0,2.6) node[above] {$x_2$};
            \def\r{2}
            \def\a{30}
            \def\b{70}
            \def\c{150}
            \draw (0,0) circle [radius=\r];
            \foreach \x in {\a,\b,\c} 
            {
                \coordinate (A) at (\x:\r);
                \filldraw[black] (A) circle(0.7pt);
                \draw (0,0)--(A);
            }
            \draw ({\r/4},0) arc [start angle=0, end angle=70, radius={\r/4}];
            \draw ({\r/4*cos(150)}, {\r/4*sin(150)}) arc [start angle=150, end angle=180, radius={\r/4}];
            \node at ({\a/2}:{1.5*\r/4}) {\small $\alpha_1$};
            \node at ({\a+(\b-\a)/2}:{1.4*\r/4}) {\small $\alpha_2$};
            \node at ({180-(180-\c)/2}:{1.5*\r/4}) {\small $\alpha_N$};
            \node at ({\a/2}:{1.1*\r}) {$I_1$};
            \node at ({\a+(\b-\a)/2}:{1.1*\r}) {$I_2$};
            \node at ({180-(180-\c)/2}:{1.1*\r}) {$I_N$};
            \node[below right] at (\r,0) {$r$};
        \end{tikzpicture}
        \caption{}
        \label{fig: monotonicity, angles and arcs notation}
    \end{figure}
    In other words, $\alpha_i = \frac{\HH^1(I_i)}{r}$ is the angle corresponding to the arc $I_i$, $\alpha_i \in (0,\pi)$ for $i = 2,\ldots, N-1$, while $\alpha_1, \alpha_N \in [0,\pi]$. Let $\alpha = \max(2\alpha_1, \alpha_2, \ldots, \alpha_{N-1}, 2\alpha_N)$ and fix $r>0$ such that \eqref{eq: proof of david-leger, main inequality} holds.

    We define $(J,v) = (K,u)$ outside $B_r \cap \overline{U}$ and  
    \begin{equation*}
        J \cap B_r = \bigcup_{p \in \partial B_r \cap K} [0,p],
    \end{equation*}
    where $[0,p]$ denotes the segment joining points $0$ and $p$.
    The connected components of $B_r \cap U \setminus J$ are circular sectors $\{A_j: j=1,\ldots,N\}$ corresponding to the arcs $\{I_j: j=1,\ldots, N\}$.
    For every $j\neq 1,N$ we define $v$ to be the harmonic extension of $u \restriction_{I_j}$ to $A_j$ given by \cite[Lemma 4.11]{david-leger}. In particular, we get 
    \begin{equation*}
        \int_{A_j} |\nabla v|^2 \leq \frac{\alpha_j r}{\pi} \int_{I_j} \left(\frac{\partial u}{\partial \tau} \right)^2 \quad \text{ for } j\neq 1,N.
    \end{equation*}
    For $j=1,N$, instead, we define $v$ to be the harmonic extension given by Lemma \ref{lmm: harmonic extension circular sector}. Therefore we obtain 
    \begin{equation*}
        \int_{A_j} |\nabla v|^2 \leq \frac{2\alpha_j r}{\pi} \int_{I_j}  \left(\frac{\partial u}{\partial \tau} \right)^2 \quad \text{ for } j = 1,N.
    \end{equation*}
    By construction $(J,v)$ is a admissible competitor for $(K,u)$ according to Definition \ref{def: admissible competitor and minimizer on a half-plane}, and hence 
    \begin{equation*}
       \begin{split}
            E(K,u,B_r \cap \overline{U}) 
            &\leq \sum_{j=1,N} \int_{A_j}|\nabla v|^2 + \sum_{j=2}^{N-1}\int_{A_j}|\nabla v|^2 + r N(r)\\
            &\leq \sum_{j=1,N} \frac{2\alpha_j r}{\pi} \int_{I_j} \left(\frac{\partial u}{\partial \tau} \right)^2 + \sum_{j=2}^{N-1}\frac{\alpha_j r}{\pi} \int_{I_j} \left(\frac{\partial u}{\partial \tau} \right)^2 + rN(r) \\ 
            &\leq \frac{\alpha r}{\pi} E_\tau(r) + rN(r).
       \end{split}
    \end{equation*}
    Substituting this into the above inequality, we get 
    \begin{equation*}
        r^2 F'(r) \geq \left(3-\frac{2\alpha}{\pi} \right)r E_\tau(r) + r E_\nu(r) 
        \geq \min \left( 3-\frac{2\alpha}{\pi},1\right) r \int_{\partial B_r \cap U \setminus K} |\nabla u|^2.
    \end{equation*}

    Finally, if $F'(r) = 0$, then the penultimate inequality implies that $E_\nu(r)=0$ and hence, by the Euler Lagrange equations \eqref{eq: Euler-Lagrange }, $D(r)=0$.
\end{proof}

Now we are ready to prove Proposition \ref{prop: david-leger monotonicity}.
\begin{proof}[Proof of Proposition \ref{prop: david-leger monotonicity}]   
    Assume without loss of generality that $x=0$.
    Let us use the notation of the proof of Proposition \ref{prop: david-leger monotonicity with angle} and fix $r>0$ such that \eqref{eq: proof of david-leger, main inequality} holds.
    \begin{enumerate}[leftmargin=*]
        \item  Assume that $N(r) = 0$.
        Then
        \begin{equation*}
            r^2 F'(r) \geq 3r E_\tau(r) + r E_\nu(r) - 2 E(K,u,B_r \cap \overline{U}).
        \end{equation*}
        We set $(J,v) = (K,u)$ outside $B_r \cap \overline{U}$, and we define $J \cap B_r = \emptyset$ and we let $v$ be the harmonic extension of $u \restriction_{\partial B_r \cap \overline{U}}$ to $B_r \cap \overline{U}$ given by Lemma \ref{lmm: harmonic extension on half-disk}. In particular, we have 
        \begin{equation*}
            \int_{B_r \cap U} |\nabla v|^2 \leq r \int_{\partial B_r \cap U} \left(\frac{\partial u}{\partial \tau} \right)^2.
        \end{equation*}
        Therefore 
        \begin{equation*}
            E(K,u,B_r \cap \overline{U}) \leq r \int_{\partial B_r \cap U} \left(\frac{\partial u}{\partial \tau} \right)^2
        \end{equation*}
        and
        \begin{equation*}
            r^2 F'(r) \geq r \int_{\partial B_r \cap U} |\nabla u|^2.
        \end{equation*}

        Finally, if $F'(r)=0$, then $\int_{\partial B_r\cap U} |\nabla u|^2= 0$, and, by the Euler Lagrange equations \eqref{eq: Euler-Lagrange }, $D(r) = 0$. 
        \item Now assume that $N(r) \in [2,\infty)$. 

        If $\alpha_1,\alpha_N \leq \frac{3}{4}\pi$ there is nothing to prove, by Proposition \ref{prop: david-leger monotonicity with angle}. So assume that this is false. Without loss of generality we assume that $\alpha_1 > \frac{3}{4}\pi$. 

        We define $(J,v) = (K,u)$ outside $B_r \cap \overline{U}$.
        Then we define (see Figure \ref{fig: competitor, monotonicity, angle bigger than 3/4pi})
        \begin{equation*}
            J\cap B_r \cap \overline{U} = [0,p] \bigcup \gamma \bigcup_{j=3}^{N(r)}[0,p_j],
        \end{equation*}
        where, identifying $\mathbb{C}$ with $\R^2$, $p = re^{i \frac{3}{4}\pi}$, $\partial B_r \cap K = \{p_1, \ldots, p_{N(r)}\}$ with $p_j = r e^{i\sum_{k=1}^{j}\alpha_k}$, $\gamma = \{r e^{i\theta} : \theta \in [\frac{3}{4}\pi, \alpha_1+\alpha_2] \}$. 
        \begin{figure}[ht]
            \begin{tikzpicture}[>=stealth, scale=1.1]
                \draw[->](-3,0)--(3,0) node[right] {$x_1$};
                \draw[->](0,-2.6)--(0,2.6) node[above] {$x_2$};
                \def\r{2}
                \draw (0,0) circle [radius=\r];
                \coordinate (P) at ({45*3}:\r);
                \draw[thick, blue] (0,0)--(P) node[above right, pos=0.5] {$J$};
                \filldraw[black] (P) circle(0.7pt) node[left] {$p$};
                \coordinate (A) at ({45*3+15+10+8}:\r);
                \draw[thick, blue] (0,0)--(A);
                \filldraw[black] (A) circle(0.7pt) node[left] {$p_3$};
                \coordinate (A) at ({45*3+15+10+8+8}:\r);
                \draw[thick, blue] (0,0)--(A);
                \filldraw[black] (A) circle(0.7pt) node[left] {$p_{N(r)}$};
                \draw[thick, blue] (P) arc [start angle = {45*3}, end angle = {45*3+15+10}, radius = \r];
                \coordinate (A) at ({45*3+15}:\r);
                \filldraw[black] (A) circle(0.7pt) node[left] {$p_1$};
                \coordinate (A) at ({45*3+15+10}:\r);
                \filldraw[black] (A) circle(0.7pt) node[left] {$p_2$};
                \node at (60:{\r/2}) {$A$};
                \node at ({45*3+(3+15+10+8)/2}:{1.1*\r/2}) {$B$};
                \node[below right] at (\r,0) {$r$};
            \end{tikzpicture}
            \caption{}
            \label{fig: competitor, monotonicity, angle bigger than 3/4pi}
        \end{figure}
        Let $A_j$ be the circular sector corresponding to the arc $I_j$. 
        On the circular sector $A= \{\rho e^{i\theta}: \rho \in (0,r), \theta \in \left(0,\frac{3}{4}\pi \right)\}$, we let $v$ be the harmonic extension of $u\restriction_{\partial A \cap \partial B_r}$ to $A$ given by Lemma \ref{lmm: harmonic extension circular sector}.
    
        Let $\beta = \{r e^{i \theta} : \theta \in (\frac{3}{4}\pi, \alpha_1+\alpha_2+\alpha_3)\}$ and define $h: \beta \to \R$ as follows. 
        \begin{equation*}
            h(r e^{i\theta}) = 
            \begin{cases}
                u(p_2) \quad &\text{ if } \theta \in (\frac{3}{4}\pi, \alpha_1+\alpha_2] \\
                u(re^{i\theta}) \quad &\text{ if } \theta \in (\alpha_1+\alpha_2,\alpha_1+\alpha_2+\alpha_3)
            \end{cases}
        \end{equation*}
        On $B = A_1 \cup A_2 \cup A_3 \setminus A$, we define $v$ to be the harmonic extension of $h$ to $B$ given by Lemma \ref{lmm: harmonic extension circular sector} if $N(r)=2$ and by Lemma \cite[Lemma 4.11]{david-leger} if $N(r)>2$.
        If $N(r)>2$, on $A_j$ for $j = 4,\ldots, N-1$ we define $v$ to be the harmonic extension of $u \restriction_{I_j}$ to $A_j$ given by Lemma \cite[Lemma 4.11]{david-leger}, while on $A_N$ we use the harmonic extension of Lemma \ref{lmm: harmonic extension circular sector}.
        By construction $(J,v)$ is an admissible competitor. Therefore
        \begin{equation*}
            \begin{split}
                E(K,u,B_r \cap \overline{U}) &\leq \frac{3}{2}r \int_{\partial B_r \cap U \setminus K} \left(\frac{\partial u}{\partial \tau}\right)^2 + r(N(r)-1) + r \left( \theta_2-\frac{3}{4}\pi\right) \\
                &\leq \frac{3}{2}r E_\tau(r) + rN(r)
            \end{split}
        \end{equation*}
        and
        \begin{equation*}
            r^2 F'(r) \geq r \int_{\partial B_r \cap U \setminus K} \left(\frac{\partial u}{\partial \nu}\right)^2.
        \end{equation*}

        Finally, if $F'(r) = 0$, then the last inequality combined with the Euler Lagrange equations \eqref{eq: Euler-Lagrange } implies that $D(r) = 0$. 

        \item Finally, assume that $N(r) =1$ and $K \cap B_r = \{r e^{i \theta}\}$ with $\theta \in (\frac{1}{4}\pi, \frac{3}{4}\pi)$.
        Then the result follows by Proposition \ref{prop: david-leger monotonicity with angle}.
    \end{enumerate}
\end{proof}

Below we state a useful monotonicity formula for elementary generalized minimizers. This is a particular case of the monotonicity formula for stationary varifolds (see \cite{allard-interior,allard-boundary}). 
\begin{prop}
    \label{prop: monotonocicity formula for stationary varifolds}
    Let $(K,u)$ be an elementary generalized minimizer according to Definition \ref{def: elementary generalized minimizers}. Assume that either $K \cap \partial U = \emptyset$ and $x \in K$, or $x \in \partial U \cap K$. Define, for $r>0$, 
    \begin{equation*}
        F(r) = \frac{1}{r} \HH^1(K \cap B_r(x)).
    \end{equation*}
    Then $F'(r) \geq 0$ for a.e. $r>0$. Moreover, if $F$ is constant, then $K$ is the union of finitely many half-lines with origin at $x$. 
\end{prop}

In the following we prove Proposition \ref{prop: Bonnet monotonicity}.
\begin{proof}[Proof of Proposition \ref{prop: Bonnet monotonicity}]
\begin{enumerate}[leftmargin=*]
    \item Let us prove (i). Assume that $K$ is connected. Let $D(r)$ be defined as in the proof of \ref{prop: david-leger monotonicity with angle}.
    Our goal is to show that $D(r) \leq rD'(r)$ for a.e. $r$. 
    Integrating by parts and using the Euler-Lagrange equations \eqref{eq: Euler-Lagrange }, we obtain $D(r) = \int_{\partial B_r(x) \cap U \setminus K} u \frac{\partial u}{\partial \nu}$.
    
    By the coarea formula \cite[Theorem 2.93, p. 101]{AFP} $\HH^0(K \cap \partial B_r (x)) < \infty $ for a.e. $r>0$. Take $r$ such that this is true.
    
    Up to translations, we can assume $x=0$. 
    Then $S_r := \partial B_r \cap U \setminus K = \cup_{j=1}^N I_j$ with $I_j \cap I_k = \emptyset$ if $j\neq k$, $I_j$ circular arcs, $I_1 = \{(r\cos\theta, r\sin\theta): \theta \in (0,\alpha_1)\}$, $I_N = \{(r\cos\theta, r\sin\theta): \theta \in (\pi - \alpha_N, \pi)\}$, with the convention that $\alpha_1 = 0$ if $(r,0) \in K$, and $\alpha_N = 0$ if $(-r,0) \in K$.
    Since $K$ is connected, integrating by parts and using the Euler-Lagrange equation \eqref{eq: Euler-Lagrange }, we have $\int_{I_j} \frac{\partial u}{\partial \nu} = 0$
    for every $j \neq 1,N$. 
    We define $m_j = \frac{1}{|I_j|} \int_{I_j} u$.
    Then, for every $j \neq 1,N$ we have, by Wirtinger's inequality \cite[Lemma 47.18]{david},
    \begin{equation*}
        \int_{I_j} (u-m_j)^2 \leq \left(\frac{|I_j|}{\pi}\right)^2 \int_{I_j} \left(\frac{\partial u}{\partial \tau}\right)^2.
    \end{equation*} 
    On the arcs $I_1$ and $I_N$ we proceed a little bit differently. 
    Consider $I_1$ and suppose that $\alpha_1>0$, otherwise the following construction is not needed. Observe that $u(r,0) = 0$ since $u=0$ on $\partial U \setminus K$.
    Therefore the odd extension $\tilde{u}$ of $u$ to $\tilde{I}_1 = \{(r\cos\theta, r\sin\theta): \theta \in (-\alpha_1, \alpha_1)\}$ is well defined.
    Since $\tilde{u}$ is odd, we have $\int_{\tilde{I}_1} \tilde{u} = 0$. Using Wirtinger's inequality \cite[Lemma 47.18]{david} for $\tilde{u}$ on $\tilde{I}_1$, we conclude that 
    \begin{equation}
        \label{eq: Wirtinger's inequality for odd extension}
        \int_{I_1} u^2 \leq \left(\frac{2|I_1|}{\pi}\right)^2 \int_{I_1} \left(\frac{\partial u}{\partial \tau}\right)^2
    \end{equation}
    The same argument holds for $I_N$. Hence, we have
    \begin{equation*}
        \begin{split}
            D(r) &= \sum_{j=1}^N \int_{I_j} u \frac{\partial u}{\partial \nu} \\
            &\hintedrel{\leq} \sum_{j=1,N} \left(\int_{I_j} u^2\right)^\frac{1}{2}\left(\int_{I_j} \left(\frac{\partial u}{\partial \nu} \right)^2\right)^\frac{1}{2} + \sum_{j=2}^{N-1} \left( \int_{I_j} (u-m_j)^2 \right)^\frac{1}{2} \left( \int_{I_j} \left(\frac{\partial u}{\partial \nu}\right)^2 \right)^\frac{1}{2} \\ 
            &\hintedrel{\leq} \sum_{j=1,N} \left(\frac{1}{2\mu} \int_{I_j} u^2 + \frac{\mu}{2} \int_{I_j} \left(\frac{\partial u}{\partial \nu}\right)^2 \right) + \sum_{j=2}^{N-1} \left(\frac{1}{2\lambda} \int_{I_j} (u-m_j)^2 + \frac{\lambda}{2} \int_{I_j} \left(\frac{\partial u}{\partial \nu}\right)^2 \right) \\
            &\hintedrel{\leq} \sum_{j=1,N} \left(\frac{2r^2}{\mu} \int_{I_j} \left(\frac{\partial u}{\partial \tau} \right)^2+ \frac{\mu}{2} \int_{I_j} \left(\frac{\partial u}{\partial \nu}\right)^2 \right) + \sum_{j=2}^{N-1} \left(\frac{r^2}{2\lambda} \int_{I_j} \left(\frac{\partial u}{\partial \tau} \right)^2 + \frac{\lambda}{2} \int_{I_j} \left(\frac{\partial u}{\partial \nu}\right)^2 \right) \\
            &= r \sum_{j=1,N} \int_{I_j} |\nabla u|^2 + \frac{r}{2}\sum_{j=2}^{N-1}\int_{I_j} |\nabla u|^2 \\
            &\hintedrel{\leq} r \int_{\partial B_r \cap U \setminus K} |\nabla u|^2 = r D'(r),
        \end{split}
    \end{equation*}
    where in (1) we used Cauchy-Schwarz inequality, in (2) we used Cauchy's inequality with $\mu, \lambda >0$ \cite[Appendix B.2, p. 706]{evans}, in (3) we used the above estimates, and in the equality before (4) we chose $\mu = 2r$ and $\lambda = r$.
    
    In conclusion, $\omega'(r) \geq 0$ for a.e. $r >0$.
    
    Now suppose that $\omega$ is constant, and suppose by contradiction that $\omega \not \equiv 0$. Then for a.e. $r>0$ we have $D(r) = r D'(r)$, and hence all the above inequalities are equalities. Take $r$ such that the above holds and $\omega(r) > 0$. 
    Then we conclude that 
    \begin{equation*}
        \int_{I_j} |\nabla u|^2 = 0 
    \end{equation*}
    for every $j\neq 1,N$, otherwise inequality (4) would be strict. Inequality (3) then becomes 
    \begin{equation*}
        \sum_{j=1,N} \int_{I_j} u^2 = \sum_{j=1,N} 4r^2 \int_{I_j} \left(\frac{\partial u}{\partial \tau}\right)^2.
    \end{equation*}
    Moreover, by Wirtinger's inequality for the odd extension \eqref{eq: Wirtinger's inequality for odd extension}, we conclude that there is $j \in \{1,N\}$ such that $\int_{I_j} |\nabla u|^2 >0$,
    for if they are both equal to zero this contradicts inequality (1), given that we assumed $D(r)=r\omega(r)>0$. Let $j \in \{1,N\}$ such that the above holds. Then it must be $|I_j| = \pi r$. Indeed, if $|I_j| < \pi r$, again by Wirtinger's inequality for the odd extension \eqref{eq: Wirtinger's inequality for odd extension}, inequality (3) would be strict.
    In conclusion, for a.e. $r>0$ either $K\cap \partial B_r = (r,0)$ or $K\cap \partial B_r = (-r,0)$.

    At this point, using Wirtinger's inequality, it is easy to show that $K \subset \partial U$ and that $K$ contains an interval $\{(t,0): t \in (a,b)\}$. Then, by Proposition \ref{prop: generalized minimizer with K in partial Omega containing an interval}, we conclude that $K = \partial U$. But this is a contradiction since then we would have $K \cap \partial B_r = \{(-r,0), (r,0)\}$ for every $r>0$. 
    \item The proof of (ii) follows from the same construction above on $I_1$ and $I_N$, since for those arcs we did not use the connectedness of $K$.
\end{enumerate}
\end{proof}

\subsection{Elementary generalized minimizers}
\label{section: elementary generalized minimizers}
In this section, we prove Theorem \ref{prop: classification of elementary generalized minimizers} and, as an easy consequence, Theorem \ref{thm: classification of elementary blow-up limits}.
\begin{proof}[Proof of Theorem \ref{prop: classification of elementary generalized minimizers}]
    \begin{enumerate}[leftmargin=*]
        \item  If $K = \emptyset$ then it is trivial to see that $u=0$ using Definition \ref{def: elementary generalized minimizers}.
        \item Suppose that $K \neq \emptyset$ and $K \cap \partial U \neq \emptyset$.
        Let $x \in \partial U \cap K$. Then, by Proposition \ref{prop: monotonocicity formula for stationary varifolds}, $F$ is monotonic. In particular there exist $l_0(x) = \lim_{r \downarrow 0} F(r)$ and $l_\infty(x) = \lim_{r \uparrow \infty} F(r)$ and $0\leq l_0(x) \leq l_\infty(x) < \infty$ by the energy upper bound (Proposition \ref{prop: Energy upper bound at the boundary}).
    
    Let $r_j \downarrow 0$ as $j \to \infty$ and let $K_j = K_{x,r_j}$ , $U_j = U_{x, r_j} = \{x_2 > 0\}$ and $u_j = u_{x,r_j}$ for every $j \in \N$. 
    By Theorem \ref{thm: compactness for generalized minimizers}, there exists $(K_0,u_0)$ elementary generalized minimizer such that $(K_j, u_j)$ converges up to subsequences to $(K_0,u_0)$ in the sense of the theorem, and 
    \begin{equation*}
        \HH^1(B_r\cap K_0) = \lim_{j \to \infty} \HH^1(B_r \cap K_j)
    \end{equation*}
    for a.e. $r>0$. Moreover, observe that 
    \begin{equation*}
        \HH^1(B_r \cap K_j) = r_j^{-1} \HH^1(B_{r r_j}(x) \cap K) = r F(r r_j).
    \end{equation*}
    This implies that 
    \begin{equation*}
        l_0(x) = \frac{1}{r}\HH^1(B_r \cap K_0)
    \end{equation*}
    for a.e. $r>0$. As a consequence, the right hand side of the above equation is constant. Hence Proposition \ref{prop: monotonocicity formula for stationary varifolds} implies that $K_0$ is the union of $N \in \N$ half-lines with origin in $0$. 

    To fix the notation, we call these half-lines $l_1, \ldots, l_N$, with 
    \begin{equation*}
        l_i=\{(t \cos \theta_i, t \sin \theta_i) : t \in [0,+\infty)\}
    \end{equation*}
    and $0\leq \theta_1 < \cdots < \theta_n \leq \pi $. 
    Moreover, we define $\alpha_i = \theta_{i+1}-\theta_i$ for $i=1, \ldots, N-1$.

    We claim that $\alpha_i \geq \frac{2}{3}\pi$ for every $i = 1,\ldots, N-1$. 
    Indeed, suppose by contradiction that $\alpha_i<\frac{2}{3}\pi$ for some $i$. 
    Fix $r>0$ and call $x_i$ and $x_{i+1}$ the points of intersection of $\partial B_r$ with $l_i$ and $l_{i+1}$ respectively. Then the triangle $T$ with vertices $0,x_i,x_{i+1}$ is isosceles. Therefore the angles with vertex $x_i$ and $x_{i+1}$ are smaller than $\frac{1}{2}\pi < \frac{2}{3} \pi$.
    Then the Fermat-Torricelli point $p$ of the triangle $T$ lies in the interior of $T$. 
    We define $S = [0,p] \cup [p,x_i] \cup [p,x_{i+1}] $ and $J$ as follows. 
    \begin{align*}
        J \cap B_r &= B_r \cap \left(\bigcup_{\myatop{j=1,\ldots, N}{j \neq i,i+1}} l_j \cup S \right), \\
        J \setminus B_r &= K_0 \setminus B_r. 
    \end{align*}
    Moreover, we define $v=u_0$ in $U \setminus B_r$. On each connected component of $B_r \setminus J$, we define $v$ to be equal to the constant value of $u_0$ on the corresponding connected component of $\partial B_r \setminus K_0$. 
    Then $(J,v)$ is a admissible competitor for $(K_0,u_0)$ according to Definition \ref{def: admissible competitor and minimizer on a half-plane}. Moreover, by definition of Fermat-Torricelli point, 
    \begin{equation*}
        \HH^1(J \cap B_r) = (N-2)r + \HH^1(S) < Nr = \HH^1(K_0 \cap B_r),
    \end{equation*}
    which is a contradiction since $(K_0, u_0)$ is a generalized minimizer on $U$ according to Definition \ref{def: admissible competitor and minimizer on a half-plane}.

    As a consequence of the previous claim, we conclude that $N\leq 2$. 
    We exclude the case $N=0$, since that would imply that $K_0 = \emptyset$.
    We claim that $N=1$ is not possible. Indeed, suppose that $N=1$. Since $(K_0,u_0)$ is an elementary generalized minimizer, we have $u_0 = 0$ in $\overline{U} \setminus K_0$. 
    For any $r>0$, $(K_0\setminus B_r, 0)$ is a competitor that contradicts the minimality of $(K_0,u_0)$.

    Assume that $N=2$. Then we claim that $K_0 = \partial U$.
    Indeed, using the above notation, suppose by contradiction that $\alpha_1 \in (0,\pi)$.
    Fix any $r>0$. Let $x_1$ and $x_2$ be the intersections of $\partial B_r$ with $l_1$ and $l_2$ respectively. We define $J$ as follows. 
    \begin{align*}
        J \cap B_r &= [x_1,x_2] \\
        J \setminus B_r &= K_0 \setminus B_r
    \end{align*}
    We define $v = u_0$ in $U \setminus B_r$. On each connected component of $B_r \setminus J$ we define $v$ to be equal to the constant value of $u_0$ on the corresponding connected component of $\partial B_r \setminus K_0$. 
    Then $(J,v)$ is a admissible competitor for $(K_0, u_0)$ according to Definition \ref{def: admissible competitor and minimizer on a half-plane}. Moreover, $\HH^1([x_1,x_2]) < 2r$ by the triangle inequality, and this contradicts the minimality of $(K_0,u_0)$. This implies that $K_0=\partial U$. 
    
    In particular, we conclude that $l_0(x) = 2$. By Theorem \ref{thm: compactness for generalized minimizers}, taking $r_j \uparrow \infty$ and by an analogous argument, we also conclude that $l_\infty(x) = 2$. 
    In particular, since $F$ is non-decreasing, we have $F(r) = 2$ for every $r>0$. In particular, $F$ is constant. Therefore, by Proposition \ref{prop: monotonocicity formula for stationary varifolds} and by the above argument, we also conclude that $K = \partial U$. 
        \item Suppose that $K \neq \emptyset$ and $K\cap \partial U = \emptyset$. Let $x \in K \cap U$.
        By Proposition \ref{prop: monotonocicity formula for stationary varifolds}, there exist $l_0(x)$ and $l_\infty(x)$ defined as in the previous point, with $0\leq l_0(x) \leq l_\infty(x) < \infty$.

        Let $r_j \uparrow +\infty$ as $j\to \infty$ and define $K_j$, $U_j$, $u_j$ as in the previous point. 
        Observe that $U_j \to \{x_2 >0\}$ locally in the Hausdorff distance. 
        Then, by Theorem \ref{thm: compactness for generalized minimizers}, $(K_j,u_j)$ converge, in the sense of the theorem, to an elementary generalized minimizer $(K_\infty,u_\infty)$ and, as in the previous point, 
        \begin{equation*}
            l_\infty(x) = \frac{1}{r} \HH^1(B_r \cap K_\infty).
        \end{equation*}
        Then, by the previous point, $K_\infty = \{x_2 = 0\}$, and hence $l_\infty(x) = 2$. In particular, $l_0(x) \leq 2$. 

        Let $r_j \downarrow 0$ as $j\to \infty$ and define $K_j$, $U_j$, $u_j$ as in the previous point. 
        Then $U_j \to \R^2$ locally in the Hausdorff distance, and $(K_j,u_j)$ converges to a generalized minimizer $(K_0, u_0)$ by \cite[Theorem 2.3.2]{delellis-focardi}. Moreover, by \cite[Theorem 2.4.1]{delellis-focardi}, $K_0$ is either a straight line or the union of three lines originating at a common point with equal angles. Since $l_0(x) \leq 2$, we conclude that $K_0$ is a line, and $l_0(x) = 2$.

        Then we conclude that 
        \begin{equation*}
            \frac{1}{r} \HH^1(B_r(x) \cap K) = 2 \quad \forall r >0.
        \end{equation*}
        In particular, by Proposition \ref{prop: monotonocicity formula for stationary varifolds}, $K$ is the union of 2 half-lines originating at $x$. Since we know that $K_\infty = \{x_2 = 0\}$, we conclude that $K$ is the line through $x$ parallel to $\partial U$. 
    \end{enumerate}
\end{proof}
\begin{proof}[Proof of Theorem \ref{thm: classification of elementary blow-up limits}]
    By Theorem \ref{thm: blow-up limit}, $(K,u)$ is an elementary generalized minimizer. Translating the plane, we can then use the classification of Proposition \ref{prop: classification of elementary generalized minimizers} and, translating back, we get the desired classification. It remains to show that $\lim_{n \to \infty} |u_n(z^1)| = \infty$ in cases (i) and (ii). This follows from the same argument contained in \cite[Proof of (b3), Theorem 2.4.1]{delellis-focardi}
\end{proof}
\subsection{Partial classification of generalized minimizers}
\label{section: partial classification of generalized minimizers}
In this section, we prove propositions \ref{prop: classification of connected generalized minimizers}, \ref{prop: exclude 1 intersection} and \ref{prop: david-leger classification of generalized minimizers}.
\begin{proof}[Proof of Proposition \ref{prop: classification of connected generalized minimizers}]
    Let $x \in K \cap \partial U$ and let $\omega$ be defined as in Proposition \ref{prop: Bonnet monotonicity}. Then there exist $l_0(x) = \lim_{r \downarrow 0} \omega(r)$ and $l_\infty(x) = \lim_{r \uparrow \infty} \omega(r)$ and $0 \leq l_0(x) \leq l_\infty(x) < \infty$ by the energy upper bound (Proposition \ref{prop: Energy upper bound at the boundary}).

    Let $r_j \downarrow 0$ as $j \to \infty$ and let $K_j = K_{x,r_j}$, $u_j = u_{x,r_j}$ and $U_j = U_{x,r_j} = U$ for every $j \in \N$. 
    By Theorem \ref{thm: compactness for generalized minimizers}, there exists $(K_0,u_0)$ generalized minimizer on $U$ such that $(K_j, u_j)$ converges up to subsequences to $(K_0,u_0)$ in the sense of the theorem, and  
    \begin{equation*}
        \int_{B_r \setminus K_0} |\nabla u_0|^2  = \lim_{j \to \infty} \int_{B_r \setminus K_j} |\nabla u_j|^2
    \end{equation*}
    for a.e. $r>0$.
    Moreover, observe that
    \begin{equation*}
        \int_{B_r \setminus K_j} |\nabla u_j|^2 = r_j^{-1} \int_{B_{r r_j}(x)\setminus K} |\nabla u|^2 = r \omega(rr_j).
    \end{equation*}
    This implies that 
    \begin{equation*}
        l_0(x) = \frac{1}{r} \int_{B_r \setminus K_0} |\nabla u_0|^2 
    \end{equation*}
    for a.e. $r>0$. As a consequence, the right hand side of the above equation is constant. 
    Hence Proposition \ref{prop: Bonnet monotonicity} implies that $l_0(x) = 0$. 
    A similar argument with $r_j \uparrow \infty$ shows that $l_\infty(x) = 0$.

    Since $\omega$ is non-decreasing, we conclude that $\omega \equiv 0$. This implies that $(K,u)$ is an elementary generalized minimizer according to Definition \ref{def: elementary generalized minimizers}.
\end{proof}
\begin{proof}[Proof of Proposition \ref{prop: exclude 1 intersection}]
    Assume by contradiction that $\HH^0(K \cap \partial B_r(x)) = 1$ for a.e. $r>0$.
    Arguing as in the proof of Proposition \ref{prop: classification of connected generalized minimizers}, we conclude that $(K,u)$ is elementary. This implies that $K = \partial U$, whence $\HH^0(K \cap \partial B_r(x)) = 2$ for every $r>0$.
\end{proof}

\begin{proof}[Proof of Proposition \ref{prop: david-leger classification of generalized minimizers}]
    By Proposition \ref{prop: david-leger monotonicity}, $F'(r) \geq 0$ for a.e. $r>0$. Therefore, 
    there exist $l_0(x) = \lim_{r \downarrow 0} F(r)$  and $l_\infty(x) = \lim_{r \uparrow \infty} F(r)$ and $0\leq l_0(x) \leq l_\infty(x) < \infty$ by the energy upper bound (Proposition \ref{prop: Energy upper bound at the boundary}).

    Let $r_j \downarrow 0$ as $j \to \infty$ and let $K_j = K_{x,r_j}$, $u_j = u_{x,r_j}$ and $U_j = U_{x,r_j} = U$ for every $j \in \N$. 
    By Theorem \ref{thm: compactness for generalized minimizers}, there exists $(K_0,u_0)$ generalized minimizer on $U$ such that $(K_j, u_j)$ converges up to subsequences to $(K_0,u_0)$ in the sense of the theorem. Moreover, for every $r >0$, 
    \begin{equation*}
            \frac{2}{r} \int_{B_r \setminus K_0} |\nabla u_0|^2 + \frac{1}{r} \HH^1(K_0 \cap B_r) = \lim_{j \to \infty} F(r r_j) = l_0(x).
    \end{equation*}
    Therefore, the first term of the above equation is constant with respect to $r$ and, by Proposition \ref{prop: david-leger monotonicity}, $(K_0, u_0)$ is elementary. Since $0 \in K_0 \cap \{x_2 = 0\}$, Theorem \ref{prop: classification of elementary generalized minimizers} implies that $K_0 = \{x_2 = 0\}$ and hence $l_0(x) = 2$. 

    Repeating the same argument with $r_j \uparrow +\infty$, we conclude that also $l_\infty(x) = 2$. Since $F$ is nondecreasing, this implies that $F(r) = 2$ for every $r>0$. In particular, $F$ is constant and, by Proposition \ref{prop: david-leger monotonicity}, $(K,u)$ is elementary. Since $x \in K$, this implies that $K = \partial U = \{x_2 =0\}$. 
\end{proof}
\section{Proof of Theorem \ref{thm: blow-up limit}}
\label{section: proof of blow-up limit theorem}
In this section, we prove Theorem \ref{thm: blow-up limit}.
\begin{proof}[Proof of Theorem \ref{thm: blow-up limit} (i)]
    Observe that $K_0 = (K_0 \cap \Omega_0) \cup (K_0 \cap \partial\Omega_0)$.
    Moreover, $K_0 \cap \partial\Omega_0$ is rectifiable, since it is a subset of the rectifiable set $\partial\Omega_0$. 
    Finally, by \cite[Theorem 2.2.3]{delellis-focardi}, $K_0 \cap \Omega_0$ is rectifiable. 
    Therefore $K_0$ is rectifiable, being the union of two rectifiable sets. 
\end{proof}
\begin{proof}[Proof of Theorem \ref{thm: blow-up limit} (ii)] Observe that, by \cite[Theorem 2.2.3]{delellis-focardi}, our thesis holds under the stricter assumption $U \subset \subset \Omega_0$. It remains to deal with the case that $\overline{U} \cap \partial\Omega_0 \neq \emptyset$.
\begin{claim}
    \label{claim: H^1 of K_0 and K_j on balls}
    There is a constant $C>0$ such that 
    \begin{equation*}
       C^{-1} \limsup_{n \to \infty}\HH^1(B_r(z) \cap K_n) \leq \HH^1(\overline{B}_r(z) \cap K_0) \leq C \liminf_{n \to \infty} \HH^1(B_{2r}(z) \cap K_n) 
    \end{equation*}
    for every $B_r(z) \subset \R^2$. 
\end{claim}
\begin{proof}[Proof of Claim]
    The proof is the same as the proof of Step 1 in \cite[Theorem 2.2.3]{delellis-focardi}, with the density lower bound replaced by Theorem \ref{thm: density lower bound for K}.
\end{proof}
\begin{claim}[Density bounds for blow-up limit] 
    \label{claim: density bounds for K_0}
    There exist $\rho_0, C_1, C_2, C_3 > 0$ such that 
    \begin{equation*}
        \HH^1(K_0 \cap B_\rho(z)) \geq C_1 \rho \quad \forall z \in K_0, \forall \rho \in (0, \rho_0)
    \end{equation*}
    and 
    \begin{equation*}
        \HH^1(K_0 \cap B_\rho(z)) \leq C_2 \rho^2 + C_3 \rho \quad \forall z \in \R^2, \forall \rho > 0.
    \end{equation*}
\end{claim}
\begin{proof}[Proof of Claim]
    First observe that the density upper and lower bounds (Proposition \ref{prop: Energy upper bound at the boundary} and Theorem \ref{thm: density lower bound for K}) for $K$ hold, with the same constants, for the blow-up sequence $K_n$. This follows from the fact that, for $n \in \N$, $z \in K_n$ and $\rho >0$, we have
    \begin{equation*}
        \HH^1(K_n \cap B_\rho(z)) = r_n^{-1} \HH^1(K^n \cap B_{r_n \rho}(x_n+r_n z)), 
    \end{equation*}
    and $x_n+r_n z \in K^n$. 
    The claim follows from this information combined with Claim \ref{claim: H^1 of K_0 and K_j on balls} and local Hausdorff convergence of $K_n$ to $K_0$.
\end{proof}
\begin{claim}
    \begin{equation*}
        \int_{U \setminus K_0} |\nabla u_0|^2 \leq \liminf_{n \to \infty} \int_{U \cap (\Omega_n \setminus K_n)}|\nabla u_n|^2
    \end{equation*}
\end{claim}
\begin{proof}[Proof of Claim]
    By Proposition \ref{prop: blow-up, convergence of functions}, we have 
    \begin{equation*}
        \lim_{n\to \infty} \int_A |\nabla u_n|^2 = \int_A |\nabla u_0|^2
    \end{equation*}
    for every open set $A \subset \subset \Omega_0 \setminus K_0$.
Let $\{A_m\}_{m \in \N}$ be a compact exhaustion of $U \setminus K_0$ (i.e. $A_m$ compact, $A_m \subset \mathrm{int}(A_{m+1})$ for every $m \in \N$, and $A_m \uparrow U \setminus K_0$). Then, by the monotone convergence theorem and by the above, 
\begin{equation*}
    \int_{U \setminus K_0}  |\nabla u_0|^2 = \lim_{m \to \infty} \int_{A_m} |\nabla u_0|^2 \leq \lim_{m \to \infty} \liminf_{n \to \infty} \int_{A_m} |\nabla u_n|^2.
\end{equation*}
Since $K_n \to K_0$ locally in the Hausdorff distance, and since $A_m \uparrow U \setminus K_0$, for every $m \in \N$ we have $A_m \subset U \cap (\Omega_n \setminus K_n)$ for sufficiently large $n$. This concludes the proof of the claim. 
\end{proof}
\begin{claim}
    \label{claim: blow-up lower semicontinuity of H^1}
    \begin{equation*}
       \HH^1(\overline{U} \cap K_0) \leq \liminf_{n \to \infty} \HH^1(\overline{U} \cap K_n)
    \end{equation*}
    for every open set $U \subset \subset \overline{\Omega}_0$ such that $\HH^1(\partial U \cap K_0 \cap \Omega_0) = 0$.
\end{claim}
\begin{proof}[Proof of Claim]
    For every $n \in \N$, define the Radon measure $\mu_n = \HH^1 \mres {K_n}$. By Proposition \ref{prop: Energy upper bound at the boundary}, for every compact set $J \subset \R^2$ we have $\sup_{n \in \N} \mu_n(J) < \infty$. Therefore, by \cite[Theorem 1.41]{evans-gariepy}, there is a Radon measure $\mu$ on $\R^2$ such that, up to subsequences, $\mu_n \rightharpoonup \mu$.
    Now we claim that
    \begin{equation*}
        C^{-1} \HH^1 \mres K_0 \leq \mu \leq C \HH^1 \mres K_0
    \end{equation*}
    for some constant $C>0$. Indeed, by Claims \ref{claim: H^1 of K_0 and K_j on balls} and \ref{claim: density bounds for K_0}, we have 
    \begin{equation*}
        \begin{split}
            \theta^*_1(\mu,z) &= \limsup_{\rho \downarrow 0} \lim_{n\to\infty} \frac{\HH^1(K_n \cap B_\rho(z))}{2\rho} \\
            &\geq C^{-1} \limsup_{\rho \downarrow 0}\frac{\HH^1(K_0 \cap B_\rho(z))}{2\rho} \geq \frac{1}{2} C^{-1} C_1
        \end{split}
    \end{equation*}
    for every $z \in K_0$, and in a similar way
    \begin{equation*}
        \begin{split}
            \theta_{*1} (\mu,z) &= \liminf_{\rho \downarrow 0} \lim_{n\to\infty}\frac{\HH^1(K_n \cap B_\rho(z))}{2\rho} \leq C
        \end{split}
    \end{equation*}
    for some constant $C>0$. Then, by \cite[Theorem 2.56, p.78]{AFP}, there is a positive constant $C$ such that  
    \begin{equation*}
        C^{-1} \HH^1 \mres K_0 \leq \mu \mres K_0 \leq C \HH^1 \mres K_0.
    \end{equation*}
    Finally, since $K_n \to K_0$ locally in the Hausdorff distance and since $K_0$ is closed, for every $z \in \R^2 \setminus K_0$ we have 
    \begin{equation*}
        \mu(B_\rho(z)) = \lim_{n \to \infty} \HH^1(K_n \cap B_\rho(z)) = 0 
    \end{equation*}
    for sufficiently small $\rho>0$. This implies that $\supp{\mu} \subset K_0$ and hence $\mu = \mu \mres K_0$. 

    The above inequality then allows us to conclude that $\mu \ll \HH^1 \mres K_0$. Then, by \cite[Theorem 1.30]{evans-gariepy}, 
    \begin{equation*}
        \mu = \theta \HH^1 \mres K_0,
    \end{equation*}
    where $\theta$ is the density of $\mu$ with respect to $\HH^1 \mres {K_0}$, and $\theta(y) \in [C^{-1},C]$ for $\HH^1$-a.e. $y \in \R^2$. 
    
    Let $U$ be a bounded open subset of $\R^2$ such that $\HH^1(\partial U \cap K_0) = 0$. 
    Then $\mu(\partial U) = 0$ and, by \cite[Theorem 1.40]{evans-gariepy}, we have 
    \begin{equation*}
        \mu(U) = \lim_{n \to \infty} \mu_n(U) = \lim_{n \to \infty} \HH^1(U \cap K_n).
    \end{equation*}
    To prove the Claim, it remains to show that $\theta(y) \geq 1$ for $\HH^1$-a.e. $y \in K_0$.
    The blow-up argument of Step 3 in the proof of \cite[Theorem 2.2.3]{delellis-focardi} shows that $\theta(y) \geq 1$ for $\HH^1$-a.e. $y \in K_0 \cap \Omega_0$. We are left to consider points $y \in K_0 \cap\partial\Omega_0$. 
    
    If $\HH^1(K_0 \cap\partial\Omega_0) = 0$ there is nothing to prove. Hence, we consider the case
    $\HH^1(K_0 \cap\partial\Omega_0) > 0$. 
    Suppose by contradiction that there is a set $A \subset K_0 \cap\partial\Omega_0 $ with $\HH^1(A)>0$ and $\theta(y) < 1$ for every $y \in A$. Since, by point (i), $K_0$ is $\HH^1$-rectifiable, we know that there exists the approximate tangent line $T_y K_0$ for $\HH^1$-a.e. $y \in K_0$. Moreover, since $\HH^1(K_0 \cap\partial\Omega_0) > 0$ and $\partial\Omega_0 = \{x_2 = -a\} = T_x \partial \Omega - a$ for some $a \in [0,+\infty)$, we have 
    \begin{equation*}
        T_y K_0 = T_x \partial \Omega = \{x_2 = 0\}
    \end{equation*}
    for $\HH^1$-a.e. $y \in K_0 \cap\partial\Omega_0$. 
    Therefore, there is a point $y \in A$ such that 
    $y$ is a Lebesgue point of $\theta$ with respect to $\HH^1 \mres K_0$, $0 < \theta(y) < 1$,
    the approximate tangent line $T_y K_0$ exists and $T_y K_0 = T_x \partial \Omega $.

    The existence of the approximate tangent line $T_y K_0$ implies 
    \begin{equation*}
        \HH^1\mres (K_0)_{y,\rho} = \rho^{-1} (\HH^1 \mres {K_0})_{y,\rho} \rightharpoonup \HH^1 \mres {T_y K_0}
    \end{equation*}
    as $\rho \downarrow 0$. 
    Since $K_0$ has the density lower bound (Claim \ref{claim: density bounds for K_0}), we conclude that $(K_0)_{y,\rho} = \rho^{-1}(K_0 - y) \to T_y K_0$ locally in the Hausdorff distance as $\rho \downarrow 0$.
    Let $\{\rho_n\}_{n \in \N} \subset \R$ be such that $0 < \rho_n \to 0$ as $n \to \infty$.
    Then, by a diagonal argument, for a subsequence of $\{r_n\}_n$ which do not rename we have 
    \begin{equation*}
        \tilde{K}_n = (K_n)_{y,\rho_n} = \rho_n^{-1}[r_n^{-1}(K^n-x_n)-y] \to T_y K_0
    \end{equation*}
    locally in the Hausdorff distance.

    Moreover, since $y$ is a Lebesgue point of $\theta$ with respect to $\HH^1 \mres K_0$, by \cite[Proposition 3.12, p.17]{delellis-tangent-measures}, we know that 
    \begin{equation*}
        \rho^{-1} (\theta \HH^1 \mres K_0)_{y,\rho} \rightharpoonup \theta(y) \HH^1 \mres T_y K_0
    \end{equation*}
    as $\rho \downarrow 0$.
    Now recall that 
    \begin{equation*}
        \mu_n = \HH^1 \mres K_n \rightharpoonup \theta \HH^1 \mres K_0
    \end{equation*}
    as $n\to \infty$. Therefore, by a diagonal argument, we can choose the subsequence of  $\{r_n\}_n$ above in such a way that we also have 
    \begin{equation}
        \label{eq: blow-up limit: convergence of measures double blow-up}
        \HH^1 \mres \tilde{K}_n = \HH^1 \mres (K_n)_{y,\rho_n}  = \rho_n^{-1} (\HH^1 \mres K_n)_{y,\rho_n} \rightharpoonup \theta(y) \HH^1 \mres T_y K_0.
    \end{equation}

    We define $\tilde{u}_n = (u_n)_{y,\rho_n}$, $\tilde{g}_n = (g_n)_{y,\rho_n}$ and $\tilde{\Omega}_n = (\Omega_n)_{y,\rho_n}$. 
    Then it is clear that \minimizer{\tilde{K}_n}{\tilde{u}_n}{\tilde{\Omega}_n}{\tilde{g}_n}.
    Then, arguing as in the proof of Theorem \ref{thm: density lower bound for K}, we can extend $u$ to $\R^2$ by setting $u=g$ in $\R^2 \setminus \overline{\Omega}$ and conclude that $u \in SBV(\R^2)$, $S_u \subset K$, and, in the notation of Section \ref{section: notation, density lower bound SBV}, $u \in \argmin\{F(v,\overline{\Omega}): v\in SBV(\R^2), v=g \text{ in } \R^2\setminus \overline{\Omega}\}$. 
    This also implies that $\tilde{u}_n \in SBV(\R^2)$, $S_{\tilde{u}_n} \subset \tilde{K}_n$, and $\tilde{u}_n \in \argmin\{F(v,\overline{\tilde{\Omega}}_n): v \in SBV(\R^2), v=\tilde{g}_n \text{ in } \R^2 \setminus \overline{\tilde{\Omega}}_n\}$.

    Moreover, observe that \eqref{eq: blow-up limit: convergence of measures double blow-up} implies 
    \begin{equation*}
        \lim_{n \to \infty} \HH^1(\tilde{K}_n \cap B_1) = 2 \theta(y) < 2
    \end{equation*}
    and hence, for sufficiently large $n$ and for some $\theta <1$, 
    \begin{equation*}
        \HH^1(\tilde{K}_n \cap B_1) < 2\theta. 
    \end{equation*}
    Furthermore, \cite[Theorem 2.56]{AFP} implies that 
    \begin{equation*}
        \theta_1^*(|\nabla u_n|^2 \mathcal{L}^2, y) = 0 \quad \forall n \in \N,
    \end{equation*}
    for $\HH^1$-a.e. $y\in K_0$, since $\mathcal{L}^2(K_0) = 0$. Therefore we can also assume that $y$ is such that the above property is satisfied. 
    As a consequence, since 
    \begin{equation*}
        |\nabla \tilde{u}_n|^2 \mathcal{L}^2 
        = \rho_n^{-1} \left[r_n^{-1}(|\nabla u^n|^2 \mathcal{L}^2)_{x_n,r_n}\right]_{y, \rho_n} = \rho_n^{-1}(|\nabla u_n|^2 \mathcal{L}^2)_{y, \rho_n},
    \end{equation*}
    by a diagonal argument, we can choose a subsequence of $\{\rho_n\}_n$, that we do not rename, such that
    \begin{equation}
        \label{eq: nabla tilde u_n Leb converges weakly to 0}
        |\nabla \tilde{u}_n|^2 \mathcal{L}^2 \rightharpoonup 0. 
    \end{equation}
    
    Finally, recall that $g_n \to 0$ 
    uniformly as $n \to \infty$ on $B_R$ for every $R>0$. 
    Then, by a diagonal argument, we can choose a subsequence of $\{r_n\}_n$, that we do not rename, such that $\tilde{g}_n = (g_n)_{y,\rho_n} \to 0$
    uniformly in $B_1$ as $n \to \infty$. 
    \begin{claim}
        \label{claim: existence of vertical segment that does not interesect the singular set}
        For any $0 < \gamma < 1-\theta$ there is $t \in (-(1-\gamma), 1-\gamma)$ such that, up to subsequences,  
        \begin{equation*}
            \lim_{n \to \infty} \int_{-1}^1 |\nabla \tilde{u}_n (t,x_2)|^2 \dd{x_2} = 0, \quad
            \tilde{K}_n \cap (\{t\} \times \R) = \emptyset.
        \end{equation*}
    \end{claim}
    \begin{proof}[Proof of Claim]
        Suppose that $0 < \gamma < 1-\theta$. 
        By Fubini's Theorem, by \eqref{eq: nabla tilde u_n Leb converges weakly to 0} and by Fatou's Lemma, we conclude that
        \begin{equation*}
            \liminf_{n \to \infty} \int_{-1}^1 |\nabla \tilde{u}_n(x_1,x_2)|^2 \dd{x_2} = 0
        \end{equation*}
        for $\mathcal{L}^1$-a.e. $x_1 \in (-(1-\gamma), 1-\gamma)$. Now let us define, for every $n \in \N$, 
        \begin{equation*}
            A_n = \{x_1 \in -((1-\gamma), 1-\gamma): \tilde{K}_n \cap (\{x_1\} \times \R) = \emptyset\}.
        \end{equation*}
        Then we must have $\HH^1(A_n) > 1- \gamma - \theta > 0$. Indeed, suppose by contradiction that this is not true. Then for sufficiently large $n$ we have
        \begin{equation*}
            \HH^1(\tilde{K}_n \cap B_1) \geq \HH^1((-(1-\gamma), 1-\gamma) \setminus A_n)= 2(1-\gamma) - \HH^1(A_n) \geq 1-\gamma+\theta > 2\theta
        \end{equation*}
        which is a contradiction.
        Combining this information with the above and taking a subsequence we find a point $t \in (-(1-\gamma), 1-\gamma)$ with the desired properties. 
    \end{proof}
    \begin{claim}
        \label{claim: rho such that Dirichlet integral on the circle is uniformly bounded and the singular set intersects at most two points}
        For any $\gamma \in (0,1)$ there is $\rho \in (1-\gamma, 1)$ such that, up to subsequences
        \begin{equation*}
            \lim_{n \to \infty} \int_{\partial B_\rho} |\nabla \tilde{u}_n|^2 < \infty, \quad \lim_{n \to \infty} \HH^0(\tilde{K}_n \cap \partial B_\rho) \leq 2.
        \end{equation*}
    \end{claim}
    \begin{proof}[Proof of Claim]
        By \eqref{eq: blow-up limit: convergence of measures double blow-up}, since $\HH^1(\partial B_\rho \cap \{x_2=0\}) = 0$ for every $\rho$, we have 
        \begin{equation*}
            \lim_{n \to \infty} \HH^1(\tilde{K}_n \cap (B_1 \setminus B_{1-\gamma})) = 2\gamma \theta(y) < 2\gamma 
        \end{equation*}
        for every $\gamma \in (0,1)$. 
        Then, by the coarea formula \cite[Theorem 2.93, p. 101]{AFP}, by a consequence of the coarea formula \cite[Theorem 3.12, p. 140]{evans-gariepy} and by \eqref{eq: nabla tilde u_n Leb converges weakly to 0}, we have, up to subsequences,
        \begin{equation*}
            \lim_{n \to \infty} \int_{1-\gamma}^1 \left(\HH^0(\tilde{K}_n \cap \partial B_\rho) + \int_{\partial B_\rho} |\nabla \tilde{u}_n(z)|^2 \dd{\HH^1(z)} \right) \dd{\rho} \leq 2 \gamma, 
        \end{equation*}
        and hence, if we define for every $\rho>0$
        \begin{equation*}
            h(\rho) = \liminf_{n \to \infty}  \HH^0(\tilde{K}_n \cap \partial B_\rho) +  \liminf_{n \to \infty} \int_{\partial B_\rho} |\nabla \tilde{u}_n(z)|^2 \dd{\HH^1(z)} , 
        \end{equation*}
        by Fatou's Lemma we have
        \begin{equation*}
            \int_{1-\gamma}^1 h(\rho) \, \dd{\rho} \leq 2\gamma. 
        \end{equation*}
        Then, by Chebyshev's inequality, there exists $\rho \in (1-\gamma, 1)$ such that, for a subsequence for which the two limits exist,
        \begin{equation*}
            \lim_{n \to \infty}  \HH^0(\tilde{K}_n \cap \partial B_\rho) +  \lim_{n \to \infty}\int_{\partial B_\rho} |\nabla \tilde{u}_n(z)|^2 \dd{\HH^1(z)} < 3.
        \end{equation*} 
    \end{proof}
    \begin{claim} 
        \label{claim: blow-up limit: tilde(u)_n uniformly bounded on partial B_rho}
        Let $\gamma \in (0, 1-\theta)$ and let $\rho \in (1-\gamma,1)$ be given by Claim \ref{claim: rho such that Dirichlet integral on the circle is uniformly bounded and the singular set intersects at most two points}. Then, up to subsequences, there is a constant $C>0$ such that
        \begin{equation*}
            \norm{\tilde{u}_n}_{L^\infty(\partial B_\rho)} \leq C 
        \end{equation*}
        for every $n \in \N$.
    \end{claim}
    \begin{proof}[Proof of Claim]
        By Claim \ref{claim: rho such that Dirichlet integral on the circle is uniformly bounded and the singular set intersects at most two points}, for sufficiently large $n$, $\tilde{K}_n$ intersects $\partial B_\rho$ in at most two points. 
        To fix ideas, suppose that $\HH^0(\tilde{K}_n \cap \partial B_\rho) =2$.
        Let $A^1_n$ and $A^2_n$ be the two connected components (two open circular arcs) of $\partial B_\rho \setminus \tilde{K}_n$.
        
        Let $t \in (-(1-\gamma), 1-\gamma)$ be given by Claim \ref{claim: existence of vertical segment that does not interesect the singular set}.
        Let $z \in \partial B_\rho \setminus \tilde{K}_n$. Without loss of generality, assume that $(t, -\sqrt{\rho^2-t^2}) \in A^2_n$. Then, if $z \in A^2_n$, we can connect it to $(t, -\sqrt{\rho^2-t^2})$ with a circular arc contained in $\partial B_\rho$.
        If, instead, $z \in A^1_n$, then we can connect it to $(t, -\sqrt{\rho^2-t^2})$ with a path made of a circular arc contained in $\partial B_\rho$ and the vertical segment $(\{t\} \times \R ) \cap B_\rho$. In both cases, we call $\gamma_{z,t}$ such path. Observe that $\gamma_{z,t} \cap \tilde{K}_n = \emptyset$. Therefore, 
        by Claims \ref{claim: rho such that Dirichlet integral on the circle is uniformly bounded and the singular set intersects at most two points} and \ref{claim: existence of vertical segment that does not interesect the singular set}, 
        \begin{equation*}
            |\tilde{u}_n(z) - \tilde{g}_n(t,-\sqrt{\rho^2-t^2})| \leq \HH^1(\gamma_{z,t})^{1/2}\left(\int_{\gamma_{z,t}} |\nabla \tilde{u}_n|^2\right)^{1/2} \leq C. 
        \end{equation*}
        for some constant $C>0$. We conclude by observing that $\tilde{g}_n(t,-\sqrt{\rho^2-t^2}) \to 0$ as $n \to \infty$.
    \end{proof}
    \begin{claim}
        \label{Claim: blow-up limit: maximum principle}
        There is a constant $C>0$ such that 
        \begin{equation*}
            \norm{\tilde{u}_n}_{L^\infty(B_\rho)} \leq C \quad \forall n \in \N.
        \end{equation*}
    \end{claim}
    \begin{proof}[Proof of Claim]
        Observe that $\tilde{u}_n = \tilde{g}_n$ in $B_\rho \setminus \tilde{\Omega}_n$. 
        Let $C$ be the constant given by Claim \ref{claim: blow-up limit: tilde(u)_n uniformly bounded on partial B_rho}. 
        Since $\tilde{g}_n \to 0$ uniformly in $B_1$, for sufficiently large $n$ we have 
        $\norm{\tilde{g}_n}_{L^\infty(B_\rho)} \leq C$.
        Now the thesis follows by a simple truncation argument, considering $v_n = (\tilde{u}_n \wedge C) \vee (-C)$ as a competitor, in a similar way as in the proof of Proposition \ref{prop: maximum principle}. 
    \end{proof}
    
    For every $\eta>0$ we know that $\tilde{K}_n \cap B_1 \subset B_1 \cap \{|x_2| \leq \eta \}$ for sufficiently large $n$. As a consequence, by Remark \ref{Remark: regularity of u}, we know that $\tilde{u}_n \in C^\infty(B_\rho \cap \{x_2 > \eta\} )$ for sufficiently large $n$. 
    Furthermore, by the energy upper bound (Proposition \ref{prop: Energy upper bound at the boundary}) we know that $\nabla \tilde{u}_n$ is uniformly bounded in $L^2(B_{\rho} \cap \{x_2 > \eta\})$. 
    The above uniform bound (Claim \ref{Claim: blow-up limit: maximum principle}) on the $L^\infty$ norm of $\tilde{u}_n$ also implies that $\tilde{u}_n$ is uniformly bounded in $L^2(B_{\rho})$.
    As a consequence, by a diagonal argument, for a subsequence that we do not rename, for every $\eta >0$
    \begin{equation*}
        \tilde{u}_n \rightharpoonup \tilde{u} \quad \text{ in } W^{1,2}(B_{\rho} \cap \{x_2 > \eta\})
    \end{equation*}
    and 
    \begin{equation*}
        \tilde{u}_n \to \tilde{u} \quad \text {a.e. in } B_\rho \cap \{x_2 > \eta\}
    \end{equation*}
    as $n \to \infty$, for some function $\tilde{u} \in W^{1,2} (B_{\rho} \cap \tilde{\Omega}_0)$, where $\tilde{\Omega}_0 = \{x_2 > 0\}$. Moreover, this implies that $\tilde{u}_n \to \tilde{u}$ a.e. in $B_\rho \cap \tilde{\Omega}_0$.
    \begin{claim}
        \label{claim: blow-up limit: tilde u is zero} 
        $\tilde{u} = 0$ a.e. in $B_\rho \cap \tilde{\Omega}_0$.
    \end{claim}
    \begin{proof}[Proof of Claim]
        By Fubini's Theorem, \eqref{eq: nabla tilde u_n Leb converges weakly to 0} and by Fatou's Lemma, if we define 
        \begin{equation*}
            A:= \left\{ \alpha \in (0,\rho): \liminf_{n \to \infty} \int_{\{x_2=\alpha\} \cap B_\rho } |\nabla \tilde{u}_n|^2 \dd{\HH^1} = 0 \right\},
        \end{equation*}
        then $\mathcal{L}^1((0,\rho) \setminus A) = 0$.
        Moreover, if we define $B = \{z=(z_1,z_2) \in B_\rho: z_2 \in A\}$,
        then $\mathcal{L}^2(B_\rho \cap \tilde{\Omega}_0 \setminus B) = 0$.
        Let us define $C = \{z \in B_\rho \cap \tilde{\Omega}_0: \tilde{u}(z) = \lim_{n \to \infty} \tilde{u}_n(z)\}$.
        Then $\mathcal{L}^2(B_\rho \cap \tilde{\Omega}_0 \setminus (B \cap C)) = 0$.
        We conclude the proof by showing that 
        \begin{equation*}
            \tilde{u}(z) = 0 \quad \forall z \in B_\rho \cap \tilde{\Omega}_0 \cap B \cap C.
        \end{equation*}
        Indeed, let $z=(z_1,z_2) \in B_\rho \cap \tilde{\Omega}_0 \cap B \cap C$.
        Since $\tilde{K}_n \cap B_1 \to \{x_2=0\} \cap B_1$ in the Hausdorff distance, for sufficiently large $n$ we have $\tilde{K}_n \cap B_1 \subset B_1 \cap \{|x_2| \leq \tfrac{1}{2} z_2\}$. 
        This implies, by Remark \ref{Remark: regularity of u}, that $\tilde{u}_n \in C^{\infty} (B_1 \cap \{x_2 > \tfrac{1}{2} z_2 \})$, and hence, by the fundamental theorem of calculus, we have, up to subsequences,
        \begin{equation*}
            |\tilde{u}_n(z_1,z_2) - \tilde{u}_n(t,z_2)| \leq (2\rho)^{1/2} \left(\int_{\{x_2=z_2\} \cap B_\rho } |\nabla \tilde{u}_n|^2 \dd{\HH^1}\right)^{1/2} \to 0,
        \end{equation*}
        where $t$ is given by Claim \ref{claim: existence of vertical segment that does not interesect the singular set}.
        We conclude by observing that 
        \begin{equation*}
            |\tilde{u}_n(t,z_2) - \tilde{g}_n(t, -\sqrt{\rho^2-t^2})| \leq \sqrt{2} \left(\int_{-1}^1 |\nabla \tilde{u}_n (t,x_2)|^2 \, \dd{x_2}\right)^{1/2} \to 0.
        \end{equation*}
    \end{proof}
    We define $\rho' = \tfrac{9}{10}\rho$. By \eqref{eq: blow-up limit: convergence of measures double blow-up}, we have
    \begin{equation*}
        \lim_{n\to \infty} \HH^1(\tilde{K}_n \cap B_{\rho'}) = 2\rho' \theta(y) > 0.
    \end{equation*}
    Let $\eta\in(0,\tfrac{1}{20}\rho)$ be such that 
    \begin{equation*}
        \HH^1(\partial B_{\rho'} \cap \{|x_2| \leq \eta \}) < \rho' \theta(y)
    \end{equation*}
    and let $\varphi \in C^\infty_c(B_{\rho} \cap \{x_2 > \tfrac{1}{2}\eta \})$ be such that $0 \leq \varphi \leq 1$ in $B_{\rho}$ and $\varphi \equiv 1$ in $B_{\rho'} \cap \{x_2 > \eta\}$ (see Figure \ref{fig: comparison argument, proof of blow-up, to show theta>1}). As usual, for sufficiently large $n$ we can assume that $\tilde{K}_n \cap B_1 \subset B_1 \cap \{|x_2| < \tfrac{1}{2} \eta\}$. We define, for every $n \in \N$, $w_n: \tilde{\Omega}_n \to \R$ by
    \begin{equation*}
        w_n = 
        \begin{cases}
            (1-\varphi) \tilde{u}_n \quad &\text{ in } \tilde{\Omega}_n \setminus B_{\rho'}, \\
            (1-\varphi) \tilde{g}_n \quad &\text{ in }  \tilde{\Omega}_n \cap B_{\rho'}, 
        \end{cases}
    \end{equation*}
    and $J_n = [\tilde{K}_n \cap (\tilde{\Omega}_n \setminus B_{\rho'})] \cup (\partial B_{\rho'} \cap \{|x_2| \leq \eta\})$.
    Then $w_n \in W^{1,2}(\tilde{\Omega}_n \setminus J_n)$, $w_n = \tilde{g}_n$ on $\partial \tilde{\Omega}_n \setminus J_n$, and $w_n = \tilde{u}_n$ in $\tilde{\Omega}_n \setminus B_{\rho}$. Observe that 
    \begin{figure}[ht]
        \begin{tikzpicture}[>=stealth, scale=1.1]
            \draw[->](-3.5,0)--(3.5,0) node[right] {$x_1$};
            \draw[->](0,-2.6)--(0,2.6) node[above] {$x_2$};
            \def\r{2}
            \draw (0,0) circle [radius=\r];
            \draw (0,0) circle [radius=0.9*\r];
            \def\e{0.2*\r}
            \draw[dotted] (-2.8, {\e/2}) -- (2.8,{\e/2}) node[right] {\tiny $\tfrac{1}{2}\eta$};
            \draw[dotted] (-2.8, {-\e/2}) -- (2.8,{-\e/2}) node[right] {\tiny $-\tfrac{1}{2}\eta$};
            \draw[dotted] (-2.8, \e) -- (2.8,\e) node[right] {\tiny $\eta$};
            \draw[thick, blue] (-0.91*\r, 0.3*\r) to[out=250, in=130] (-0.9*\r,{0.8*\e}) to [out=310, in=180] (-0.6*\r,{0.8*\e}) to[out=0, in=173] (0.6*\r, {0.8*\e}) to[out=353, in=240] (0.95*\r, 0.18*\r) to [out=60, in=-70] (0.91*\r, 0.3*\r) to [out=110, in=-5] (0,0.95*\r) to[out=175, in=70] (-0.91*\r, 0.3*\r);
            \node[blue] at (45:{1.2*\r}) {\small$\supp \varphi$};
            \node[below right] at (0,0.9*\r) {\small$\rho'$};
            \node[above right] at (0,\r) {\small$\rho$};
            \coordinate (A) at (-\r,{1/4*\e});
            \coordinate (B) at (\r,{-1/6*\e});
            \draw[thick] (A) to[out=5, in=200] (0,{-1/10*\e}) to[out=20, in=200] (B);
            \draw[thick,densely dashed] (B) to[out=20, in=230] (1.2*\r,0.1*\r) node[right] {$h_n$};
            \draw[thick, densely dashed] (-1.2*\r, -0.03*\r) to[out=10, in=185] (A); 
        \end{tikzpicture}
        \caption{}
        \label{fig: comparison argument, proof of blow-up, to show theta>1}
    \end{figure}
    \begin{equation*}
        \begin{split}
            &\int_{B_{\rho} \cap \tilde{\Omega}_n} |\nabla w_n|^2 = \int_{B_{\rho'} \cap \tilde{\Omega}_n} |\nabla w_n|^2 +  \int_{(B_\rho \setminus B_{\rho'}) \cap \tilde{\Omega}_n}|\nabla w_n|^2 \\
            & = \int_{B_{\rho'} \cap \tilde{\Omega}_n} |\nabla [(1-\varphi)\tilde{g}_n]|^2 +  \int_{(B_\rho \setminus B_{\rho'}) \cap \tilde{\Omega}_n \cap \supp{\varphi}}|\nabla[(1-\varphi)\tilde{u}_n]|^2 \\ 
            &+ \int_{(B_\rho \setminus B_{\rho'}) \cap \tilde{\Omega}_n \setminus \supp{\varphi}}|\nabla \tilde{u}_n|^2 =: I_n^1 + I_n^2 + I_n^3 \to 0
        \end{split}
    \end{equation*}
    as $n \to \infty$. 
    Indeed $I_n^1 \to 0$ since $\tilde{g}_n \to 0$ in $C^1(B_1)$. 
    Moreover, by previous computations, $\tilde{u}_n \rightharpoonup 0$ in $W^{1,2}(B_\rho \cap \{x_2 > \tfrac{1}{2} \eta\})$, and hence $\tilde{u}_n \to 0$ in $L^2(B_\rho \cap \{x_2 > \tfrac{1}{2} \eta\})$. This, with \eqref{eq: nabla tilde u_n Leb converges weakly to 0}, implies that $I_n^2 \to 0$. Also $I_n^3 \to 0$ by \eqref{eq: nabla tilde u_n Leb converges weakly to 0}.
    Furthermore, observe that $\HH^1(J_n \cap B_\rho \cap \overline{\tilde{\Omega}}_n) =\HH^1(\tilde{K}_n \cap (B_\rho \setminus \overline{B}_{\rho'})\cap \overline{\tilde{\Omega}}_n) +  \HH^1(\partial B_\rho \cap \{|x_2|\leq \eta\})$.
    Therefore
    \begin{equation*}
        \begin{split}
            &E(J_n, w_n, B_\rho \cap \overline{\tilde{\Omega}}_n) - E(\tilde{K}_n, \tilde{u}_n, B_\rho \cap \overline{\tilde{\Omega}}_n) \\
            &\leq \HH^1(\partial B_\rho \cap \{|x_2|\leq \eta\}) - \HH^1(\tilde{K}_n \cap B_{\rho'}) + o(1) \\ 
            &\leq \rho' \theta(y) - \HH^1(\tilde{K}_n \cap B_{\rho'}) + o(1) \xrightarrow{(n \to \infty)} -\rho'\theta(y) < 0, 
        \end{split}
    \end{equation*}
    and this, for sufficiently large $n$, contradicts the minimality of $(\tilde{K}_n, \tilde{u}_n)$ among the admissible pairs $\mathcal{A}(\tilde{\Omega}_n, \tilde{g}_n)$.
\end{proof}
\begin{claim}
    \begin{equation*}
        \int_{U \setminus K_0} |\nabla u_0|^2 + \HH^1(U \cap K_0) \geq \limsup_{n \to \infty} \left( \int_{U \cap (\Omega_n \setminus K_n)} |\nabla u_n|^2 + \HH^1(U \cap K_n) \right)
    \end{equation*}
    for every open set $U \subset \subset \overline{\Omega}_0$ such that $\HH^1(\partial U \cap K_0 \cap \Omega_0) = 0$.
\end{claim}
\begin{proof}[Proof of Claim]
    The proof is a slight modification of \cite[Proof of Theorem 2.2.3, Step 4]{delellis-focardi}.
    Assume, by contradiction, that there exists an open set $U \subset \subset \overline{\Omega}_0$ such that $\HH^1(\partial U \cap K_0 \cap \Omega_0) = 0$ and such that the above inequality fails. Then, for a subsequence for which the limit exists, we have 
    \begin{equation*}
        \int_{U \setminus K_0} |\nabla u_0|^2 + \HH^1(U \cap K_0) < \lim_{n \to \infty} \left( \int_{U \cap (\Omega_n \setminus K_n)} |\nabla u_n|^2 + \HH^1(U \cap K_n) \right).
    \end{equation*} 
    For every $n \in \N$, let us define the Radon measure 
    \begin{equation*}
        \mu_n = |\nabla u_n|^2 \mathcal{L}^2\mres (\Omega_n \setminus K_n) + \HH^1\mres K_n.
    \end{equation*}
    Then, by Proposition \ref{prop: Energy upper bound at the boundary}, $\sup_n \mu_n(J) < \infty$ for every compact subset $J$ of $\R^2$. Therefore, by \cite[Theorem 1.41]{evans-gariepy}, there is a Radon measure $\mu$ such that, up to subsequences, $\mu_n \rightharpoonup \mu$. 
    Let $B$ be a bounded Borel subset of $\R^2$ such that $\mu(\partial B) = 0$. Then $\mu(B) = \lim_{n \to \infty} \mu_n(B)$.
    \begin{claim}
        If $B \cap K_0 = \emptyset$ then 
    \begin{equation*}
        \mu(B) = \int_{B \cap (\Omega_0 \setminus K_0)} |\nabla u_0|^2.
    \end{equation*}
    \end{claim}
    \begin{proof}[Proof of Claim]
        First, suppose that $\dist(B, K_0) >0$. Then $B \cap K_n = \emptyset$ for sufficiently large $n$, and hence 
        \begin{equation*}
            \mu(B) = \lim_{n \to \infty} \int_{B \cap \Omega_n} |\nabla u_n|^2.
        \end{equation*}
        By Proposition \ref{prop: global convergence of functions, blow-up}, we know that, for an extension of $u_n$ given by the proposition, $\nabla u_n \to \nabla u_0$ uniformly in $\overline{B}$, and hence there is a constant $C>0$ such that $|\nabla u_n|^2 \leq C$ in $B$. 
        Since $\partial \Omega_n \to \partial\Omega_0$ locally in the Hausdorff distance, we know that $\chi_{\Omega_n} \to \chi_{\Omega_0}$ $\mathcal{L}^2$-a.e. in $B$. 
        Then, by the dominated convergence theorem, 
        \begin{equation*}
            \mu(B) = \int_{B \cap \Omega_0} |\nabla u_0|^2. 
        \end{equation*}
        Now assume only that $B \cap K_0 = \emptyset$. We define, for $t>0$, $B_t = \{y \in B: \dist(y, \partial B) > t\}$.
        Then $B_t \uparrow \mathrm{int}(B)$ as $t \downarrow 0$. 
        Since we assumed $\mu(\partial B) = 0$, this implies that $\mu(B) = \lim_{t \downarrow 0} \mu(B_t)$.
        Since $\mu$ is a Radon measure on $\R^2$ and $\{\partial B_t\}_{t>0}$ is a disjoint family of Borel sets, by \cite[Proposition 2.16]{maggi} we know that $\mu(\partial B_t) > 0$ for at most countably many $t >0$. Therefore, there is a sequence $\{t_n\}_{n \in \N} \subset (0,+\infty)$ such that $t_n \downarrow 0$ as $n \to \infty$, $\mu(\partial B_{t_n}) = 0$ for every $n$, and $\mu(B) = \lim_{n \to \infty} \mu(B_{t_n})$.
        Observe that $\dist(B_{t_n}, K_0) >0$ for every $n$, and hence, by the above computations, 
        \begin{equation*}
            \mu(B_{t_m}) = \int_{B_{t_m} \cap \Omega_0} |\nabla u_0|^2
        \end{equation*}
        for every $m \in \N$. Finally, by the monotone convergence theorem,
        \begin{equation*}
            \mu(B) = \lim_{m \to \infty}  \mu(B_{t_m}) = \int_{B \cap \Omega_0} |\nabla u_0|^2.
        \end{equation*}
    \end{proof}
    If $z \in K_0$, $\rho >0 $ and $\mu(\partial B_\rho(z)) = 0$, we have $\mu(B_\rho(z)) \leq C\pi \rho^2 + 2\pi\rho$ by Proposition \ref{prop: Energy upper bound at the boundary}. 
    This implies that $\theta^*_1(\mu,z) \leq \pi$, 
    and hence, by \cite[Theorem 2.56]{AFP}, we have $\mu \mres K_0 \leq 2 \pi \HH^1\mres K_0$.
    In particular, we deduce that $\mu \mres K_0 \ll \HH^1 \mres K_0$, and hence 
    \begin{equation*}
        \mu = |\nabla u_0|^2 \mathcal{L}^2 \mres (\Omega_0 \setminus K_0) + \theta \HH^1 \mres K_0
    \end{equation*}
    with $\theta$ being the density of $\mu \mres K_0$ with respect to $\HH^1 \mres K_0$. 
    Observe that if $z \in K_0$, $\rho>0$, $\mu(\partial B_\rho(z)) = 0$, by Claim \ref{claim: blow-up lower semicontinuity of H^1} we have 
    \begin{equation*}
        \mu(B_\rho(z)) = \lim_{n \to \infty} \mu_n (B_\rho(z)) \geq \lim_{n \to \infty} \HH^1(B_\rho(z) \cap K_n) \geq \HH^1(B_\rho(z) \cap K_0).
    \end{equation*} 
    We conclude that $1 \leq \theta(z) \leq 2\pi $ for $\HH^1$-a.e. $z \in K_0$.
    Now we claim that $\theta > 1$ on a subset of $U \cap K_0$ with positive $\HH^1$ measure. Otherwise we would have 
    \begin{equation*}
        \begin{split}
            \mu(U) = \int_{U \setminus K_0} |\nabla u_0|^2 + \HH^1(U \cap K_0) 
             < \lim_{n \to \infty} \mu_n(U) = \mu(U).
        \end{split}
    \end{equation*}
    Furthermore, by \cite[Theorem 2.56]{AFP}, we have,  for $\HH^1$-a.e. $y \in K_0$,
    \begin{equation*}
    \theta_1^*(|\nabla u_0|^2 \mathcal{L}^2 \mres (\Omega_0 \setminus K_0), y) = 0.
    \end{equation*}
    
    Now let $y \in K_0$ be such that $y$ is a Lebesgue point of $\theta$ with respect to $\HH^1\mres K_0$, there exists the approximate tangent line $T_y K_0$, $\theta_1^*(|\nabla u_0|^2 \mathcal{L}^2 \mres (\Omega_0 \setminus K_0), y) = 0$, and $\theta(y) > 1$.
    The proof of Step 4 in \cite[Proof of Theorem 2.2.3]{delellis-focardi} gives a contradiction in the case $y \in K_0 \cap \Omega_0$. 
    We are left to consider the case $y \in K_0 \cap \partial\Omega_0$, under the assumption that $\HH^1(K_0 \cap \partial\Omega_0) > 0$. In this case we can also require that $T_y K_0 = T_x \partial \Omega = \{x_2=0\}$. 

    Let $\{\rho_n\}_n \subset \R$ such that $0< \rho_n \to 0$ as $n \to \infty$.
    As we did in the proof of Claim \ref{claim: blow-up lower semicontinuity of H^1}, by a diagonal argument we can find a subsequence of $\{r_n\}_n$, which do not rename, such that $\tilde{K}_n := (K_n)_{y, \rho_n} \to T_y K_0$ locally in the Hausdorff distance. 
    We also define, in the same way, $\tilde{u}_n = (u_n)_{y, \rho_n}$.
    
    Now we claim that
    \begin{equation*}
        \rho^{-1} \mu_{y,\rho} \rightharpoonup \theta(y) \HH^1 \mres T_y K_0
    \end{equation*}
    as $\rho \downarrow 0$. 
    Indeed, by definition of approximate tangent line and since $y$ is a Lebesgue point of $\theta$ with respect to $\HH^1 \mres K_0$, by \cite[Proposition 3.12, p.17]{delellis-tangent-measures} we have 
    \begin{equation*}
        \rho^{-1} (\theta \HH^1 \mres K_0)_{y, \rho} \rightharpoonup \theta(y) \HH^1 \mres T_y K_0.
    \end{equation*}
    Moreover, since $\theta_1^*(|\nabla u_0|^2 \mathcal{L}^2 \mres (\Omega_0 \setminus K_0), y) = 0$, we have
    \begin{equation*}
        \rho^{-1} (|\nabla u_0|^2 \mathcal{L}^2 \mres (\Omega_0 \setminus K_0))_{y,\rho} \rightharpoonup 0.
    \end{equation*}
    
    At this point, by a diagonal argument, we can choose the subsequence of $\{r_n\}_n$ above in such a way that we also have 
    \begin{equation*}
        \tilde{\mu}_n := \rho_n^{-1} (\mu_n)_{y,\rho_n} \rightharpoonup \theta(y) \HH^1 \mres T_y K_0.
    \end{equation*}
    Observe that 
    \begin{equation*}
        \tilde{\mu}_n = |\nabla \tilde{u}_n|^2 \mathcal{L}^2 \mres (\tilde{\Omega}_n \setminus \tilde{K}_n) + \HH^1 \mres \tilde{K}_n.
    \end{equation*}
    Let $\{\eps_n\}_{n \in N} \subset \R$ be such that $\eps_n \downarrow 0$ as $n \to \infty$.  
    By a diagonal argument, we can choose the subsequence of $\{r_n\}_n$ above in such a way that we also have 
    \begin{equation*}
        \begin{split}
            &\tilde{K}_n \cap B_1 \subset \{|x_2| \leq \tfrac{1}{2}\eps_n\}, \\
            &\partial \tilde{\Omega}_n \subset \{|x_2| \leq \tfrac{1}{2}\eps_n\}, \\
            &\HH^1 \mres \partial \tilde{\Omega}_n \rightharpoonup \HH^1 \mres T_x \partial \Omega, \\
            &\lim_{n \to \infty} \tilde{\mu}_n ([-1,1] \times (-\eps_n, \eps_n)) = 2\theta(y).
        \end{split}
    \end{equation*}
    Now, if we define $L_n = ([-1, -1+\eps_n) \cup (1-\eps_n, 1]) \times [-1,1]$, by a diagonal argument we can choose the subsequence of $\{r_n\}_n$ above in such a way that we also have $\lim_{n \to \infty} \tilde{\mu}_n (L_n) = 0$.
    Now let $Q = [-1,1] \times [-1,1]$ and $Q_n^+ = Q \cap \tilde{\Omega}_n$. 
    For sufficiently large $n \in \N$, we have  $\partial \tilde{\Omega}_n \cap Q = \{(x_1,x_2) \in Q: x_2 = h_n(x_1)\}$
    for some function $h_n \in C^2([-1,1])$. Moreover, the sequence $\{h_n\}_n$ is such that 
    $h_n$ and $h_n'$ converge uniformly to $0$ in $[-1,1]$.

    We define, for every $n \in \N$, a map $\Phi_n: \overline{Q_n^+} \to \overline{Q_n^+}$ such that 
    \begin{equation*}
        \Phi_n(x_1,x_2) = (x_1, \varphi_n(x_1,x_2)), 
    \end{equation*}
    where (see Figure \ref{fig: proof of blow-up, regions, competitor to show theta<1})
    \begin{equation*}
        \varphi_n(x_1,x_2) = 
        \begin{cases}
           h_n(x_1) &\text{ if } (x_1,x_2) \in \Sigma_n^0, \\
           \frac{(1-h_n(x_1))x_2+h_n(x_1)-f_n(x_1)}{1-f_n(x_1)} &\text{ if } (x_1,x_2) \in \Sigma_n^+, 
        \end{cases}
    \end{equation*}
    with 
    \begin{equation*}
        \Sigma_n^0 = \{(x_1,x_2): h_n(x_1) \leq x_2 \leq f_n(x_1)\}, \quad
        \Sigma_n^+ = \{(x_1,x_2): f_n(x_1) < x_2 \leq 1\},
    \end{equation*} 
    and 
    \begin{equation*}
        f_n(x_1) = 
        \begin{cases}
            h_n(-1) + \frac{\eps_n - h_n(-1)}{\eps_n}(x_1+1) &\text{ if } x_1 \in [-1, -1+\eps_n), \\
            \eps_n &\text{ if } x_1 \in [-1+\eps_n, 1-\eps_n], \\
            h_n(1) - \frac{\eps_n-h_n(1)}{\eps_n}(x_1-1) &\text{ if } x_1 \in (1-\eps_n, 1].
        \end{cases}
    \end{equation*}
    \begin{figure}[ht]
        \begin{tikzpicture}[>=stealth, scale=2.2]
            \draw[->](-1.5,0)--(1.5,0) node[right] {$x_1$};
            \draw[->](0,-1.5)--(0,1.5) node[above] {$x_2$};
            \def\epsn{0.2}
            \draw[dotted] (-1, {\epsn/2}) -- (1, {\epsn/2});
            \draw[dotted] (-1, {-\epsn/2}) -- (1, {-\epsn/2});
            \draw (-1,-1)--(1,-1)--(1,1)--(-1,1)--cycle;
            \node[below right] at (1,0) {$1$};
            \node[below left] at (-1,0) {$-1$};
            \node[above right] at (0,1) {$1$};
            \node[below right] at (0,-1) {$-1$};
            \draw[dashed] (-1+\epsn, -1) -- (-1+\epsn,1);
            \draw[dashed] (1-\epsn, -1) -- (1-\epsn,1);
            \coordinate (A) at (-1,{1/4*\epsn});
            \coordinate (B) at (1,{-1/6*\epsn});
            \draw[thick, blue] (A) to[out=5, in=200] (0,{-1/10*\epsn}) to[out=20, in=200] (B);
            \draw[thick,densely dashed, blue] (B) to[out=20, in=230] (1.2,0.1) node[right] {$h_n$};
            \draw[thick, densely dashed, blue] (-1.2, -0.03) to[out=10, in=185] (A); 
            \draw[thick] (A) -- (-1+\epsn,\epsn) -- node[above, pos=0.75] {$f_n$} (1-\epsn,\epsn) -- (B);
            \node at (-0.3,0.5) {$\Sigma_n^+$};
            \node at (-0.3,0.1) {$\Sigma_n^0$};
            \draw[<->] (-1,0.5) -- (-1+\epsn,0.5) node[above, pos=0.5]{\footnotesize$\eps_n$};
        \end{tikzpicture}
        \caption{}
        \label{fig: proof of blow-up, regions, competitor to show theta<1}
    \end{figure}
    By construction, $\Phi_n(\Sigma_n^+) = Q_n^+$ and $\Phi_n|_{\partial Q_n^+}$ is the identity. 
    Moreover, observe that $\Phi_n \in C^1(\Sigma_n^+ \setminus L_n)$, $\Phi_n(\Sigma_n^+ \setminus L_n) = Q_n^+ \setminus L_n$ and, if $(x_1,x_2) \in \Sigma_n^+ \setminus L_n$, we have 
    $\det(J\Phi_n(x_1,x_2)) = \frac{1-h_n(x_1)}{1-\eps_n}$.
    This also implies that, if $(x_1,x_2) = \Phi_n^{-1}(x_1,y_2)$ for some $(x_1,y_2) \in Q_n^+ \setminus L_n$, we have $\det(J\Phi_n^{-1}(x_1,y_2)) = \frac{1-\eps_n}{1-h_n(x_1)}$.
    Furthermore, it is easy to see that $\norm{J \Phi_n^{-1}}_2 = 1+o(1)$ in $Q_n^+ \setminus L_n$ uniformly as $n \to \infty$. 
    Similarly, $\Phi_n \in C^1(\Sigma_n^+ \cap L_n)$, $\Phi_n(\Sigma_n^+ \cap L_n) = Q_n^+ \cap L_n$ and, if $(x_1, x_2) \in \Sigma_n^+ \cap L_n$ and $(x_1,x_2) = \Phi_n^{-1}(x_1,y_2)$ for some $(x_1, y_2) \in Q_n^+ \cap L_n$, we have $\det(J\Phi_n^{-1}(x_1,y_2)) = \frac{1-f_n(x_1)}{1-h_n(x_1)}$. Finally, $\norm{J\Phi_n^{-1}}_2 \leq C$ in $Q_n^+ \cap L_n$ for some constant $C>0$.
    
    We define the pair $(J_n, v_n) \in \mathcal{A}(\tilde{\Omega}_n, \tilde{g}_n)$ by 
    \begin{align*}
        J_n &= (\partial \tilde{\Omega}_n \cap Q) \cup \Phi_n(\tilde{K}_n) \cup (\tilde{K} \setminus Q), \\
        v_n &= (\tilde{u}_n \circ \Phi_n^{-1}) \chi_{Q\setminus J_n} + \tilde{u}_n \chi_{\tilde{\Omega}_n \setminus Q}.
    \end{align*}
    Then, by the above properties, it is easy to see that 
    \begin{equation*}
        \begin{split}
            &E(J_n, v_n, \overline{Q_n^+}) - E(\tilde{K}_n, \tilde{u}_n, \overline{Q_n^+}) 
            \\ 
            &\leq (1+o(1)) \int_{\Sigma_n^+ \setminus (L_n \cup \tilde{K}_n)} |\nabla \tilde{u}_n|^2 + C \tilde{\mu}_n(L_n) \\ 
            &+ \HH^1(\partial \tilde{\Omega}_n \cap Q) - \int_{\Sigma_n^+ \setminus L_n} |\nabla \tilde{u}_n|^2 - \tilde{\mu}_n(Q \cap \{|x_2| \leq \eps_n\}) \\ 
            &\xrightarrow{(n \to \infty)} -2(\theta(y)-1) < 0, 
        \end{split}
    \end{equation*}
    and this contradicts the minimality of $(\tilde{K}_n, \tilde{u}_n)$ for sufficiently large $n$.
\end{proof}
This concludes the proof of (ii).
\end{proof}
\begin{proof}[Proof of Theorem \ref{thm: blow-up limit} (iii)]
For simplicity, change coordinates so that $\Omega_0 = \{x_2 > 0\}$.
Let $O$ be an open bounded subset of $\R^2$ and $(J,v)$ be a admissible competitor for the blow-up limit $(K_0, u_0)$ according to Definition \ref{def: admissible competitor and minimizer on a half-plane}, such that $(J,v)$ coincides with $(K_0,u_0)$ in $\overline{\Omega}_0 \setminus O$. 

If $\dist(O, \partial \Omega_0) > 0$ then, by Step 4 in \cite[Proof of Theorem 2.2.3]{delellis-focardi}, we conclude that $E(K_0, u_0) \leq E(J,v)$.

We are left to consider the case $\dist(O, \partial \Omega_0) = 0$.
We consider the pairs $(K_0,u_0)$ and $(J,v)$ defined on the entire plane by extending them to $(\emptyset,0)$ in $\{x_2<0\}$.

It is elementary see that if $O'$ is a bounded open subset of $\R^2$ such that $O \subset O'$, then $(J,v)$ is still admissible competitor for $(K_0,u_0)$ in $O'$ according to Definition \ref{def: admissible competitor and minimizer on a half-plane}, and $E(J,v,O') < E(K_0,u_0, O')$.

Since $O$ is bounded, there exists $R>0$ such that $O \subset \subset B_{R/2}$.
Since $\partial \Omega \in C^2$, there exists $\overline{n} \in \N$ such that for every $n \in \N$, $n \geq \overline{n}$, there exists a function $h_n \in C^2(\R)$ such that $\partial \Omega_n \cap B_R = \{(x_1,x_2) \in B_R : x_2 = h_n(x_1)\}$.
Moreover $h_n \to 0$ in $C^1(B_R)$ as $n\to \infty$.

We define, for every $n \geq \overline{n}$, a map $\Phi_n: \R^2 \to \R^2$ as follows. 
\begin{equation*}
    \Phi_n(x_1,x_2) = (x_1,x_2-h_n(x_1)).
\end{equation*}
Observe that $\Phi_n$ is invertible, with inverse $\Phi_n^{-1}(x_1,x_2)=(x_1,x_2+h_n(x_1))$.
Moreover, for every $n$, $\det D\Phi_n=\det D\Phi_n^{-1} =1$.
Finally, $\sup_{x \in B_R} \norm{D\Phi_n(x)}_2 \leq (1+o(1))$, $\sup_{x \in B_R} \norm{D\Phi_n^{-1}(x)}_2 \leq (1+o(1))$
as $n \to \infty$.
These properties will be useful during various changes of variables. 
Recall that $(K_n,u_n)$ is the blow-up sequence. We define, for $n \in \N$ sufficiently large, $K_n' = \Phi_n(K_n)$, $\Omega_n' = \Phi_n(\Omega_n)$, $u_n' = u_n \circ \Phi_n^{-1}$, $g_n' = g_n \circ \Phi_n^{-1}$.
Observe that, for sufficiently large $n$, $\partial \Omega_n' \cap B_{\frac{3}{4}R} = \{x_2 = 0\} \cap B_{\frac{3}{4}R}$.

By a Fubini-type argument we can select $\rho \in (\frac{1}{2}R, \frac{3}{4}R)$ such that, up to subsequences, 
\begin{align*}
    &\sup_{n\in \N} \left(\int_{\partial B_\rho}|\nabla u_n'|^2 + \HH^0(K_n' \cap \partial B_\rho) \right) < \infty,\\
    &\int_{\partial B_\rho} |\nabla v|^2 + \HH^0(J \cap \partial B_\rho) < \infty.
\end{align*}
Now, for every $\delta \in (0,\rho)$, we define a map $\Psi_\delta: \R^2 \to \R^2$ as follows. $\Psi_\delta (x) = \frac{\rho}{\rho-\delta}x$
for every $x \in \R^2$. Observe that $\Phi_\delta$ is a diffeomorphism of $B_{\rho-\delta}$ onto $B_\rho$. Moreover, $\norm{D\Psi_\delta-I}_{C^0}$ and $\norm{D\Psi_\delta^{-1}-I}_{C^0}$ are infinitesimal as $\delta \to 0$. 

Then, by the same construction used in Step 4 of \cite[Proof of Theorem 2.2.3]{delellis-focardi}, there exist $\delta_n \downarrow 0$ as $n\to \infty$ and $\{(J_n',z_n')\}_{n \in \N}$ such that
\begin{itemize}
    \item $(J_n', z_n') = (K_n',u_n'-g_n')$ in $\Omega_n' \setminus B_\rho$,
    \item $E(J_n',z_n',\Omega_n' \cap (B_\rho \setminus B_{\rho-\delta_n})) \to 0$ as $n\to\infty$,
    \item $E(J_n',z_n',\Omega_n' \cap B_{\rho-\delta_n}) \to E(J,v,B_\rho)$ as $n\to\infty$,
    \item $J_n' \cap B_{\rho-\delta_n} = \Psi_{\delta_n}^{-1}(J\cap B_\rho)$
    \item $z_n' = 0$ on $\{x_2=0\} \cap B_{\rho-\delta_n} \setminus J_n'$ 
\end{itemize}
We define, for every $n\in\N$, $(J_n'',z_n'') = (J_n' \cup (\{x_2=0\} \cap (B_\rho \setminus B_{\rho-\delta_n})),z_n'+g_n')$.
Then $z_n'' = g_n'$ on $\{x_2=0\} \cap B_{\rho-\delta_n} \setminus J_n'$ and $z_n'' = u_n'$ in $\Omega_n' \setminus B_\rho$.
Finally, we define, for every $n \in \N$, $(J_n, z_n) = (\Phi_n^{-1}(J_n''), z_n'' \circ \Phi_n)$.
Then $z_n = g_n$ on $\partial \Omega_n \cap \Phi_n^{-1}(B_\rho)$ and $z_n = u_n$ in $\Omega_n \setminus \Phi_n^{-1}(B_\rho)$.
Moreover, by changing variables, using the properties of $\Phi_n$, $\Phi_n^{-1}$, $(J_n', z_n')$, and since $g_n \to 0$ in $C^1(B_R)$, it is easy to see that
\begin{equation*}
    E(J_n, z_n, \Phi_n^{-1}(B_\rho) \cap \Omega_n) \leq (1+o(1)) E(J,v,B_\rho)
\end{equation*}
as $n\to \infty$. Therefore, by point (ii) in Theorem \ref{thm: blow-up limit},
\begin{align*}
    \lim_{n\to\infty} ( E(J_n, z_n, \Phi_n^{-1}(B_\rho) \cap \Omega_n) -  E(K_n, u_n, \Phi_n^{-1}(B_\rho) \cap \Omega_n) ) \\
    \leq E(J,v,B_\rho)-E(K_0,u_0,B_\rho) <0. 
\end{align*}
Since $(J_n,z_n) \in \mathcal{A}(\Omega_n, g_n)$, this contradicts the minimality of $(K_n,u_n)$ for sufficiently large $n \in \N$.
\end{proof}

\section{Energy and density bounds}
In this section, we prove energy and density bounds for minimizers of our functional.
\subsection{Energy upper bound}
\begin{prop}[Energy upper bound]
    \label{prop: Energy upper bound at the boundary}
    Suppose $\partial \Omega \in C^1$, $g \in C^1(\partial \Omega)$, $(K,u) \in \argmin\{E(J,v): (J,v) \in \mathcal{A}(\Omega,g)\}$.  Then there exists $C = C(\Omega, \norm{g}_{C^1(\partial\Omega)})$ such that, for any $x\in \overline{\Omega}$ and $r>0$,
    \begin{equation*}
        E(K,u,B_r(x)\cap \Omega) \leq C\pi r^2 + 2\pi r
    \end{equation*}
\end{prop}
\begin{proof}
    It follows from a simple comparison argument. 
\end{proof}
\subsection{Density lower bound}
\label{section: density lower bound}
\subsubsection{Preliminaries}
\label{section: notation, density lower bound SBV}
We first need to introduce some notation. If $u \in SBV(\R^2)$ we define 
\begin{equation*}
    F(u,E) = \int_E |\nabla u|^2 + \HH^1(S_u \cap E)
\end{equation*}
for every Borel set $E \subset \R^2$. 
Following the approach in \cite{carriero-leaci}, we want to investigate 
\begin{equation*}
    u \in \argmin\{F(v,\overline{\Omega}): v\in SBV(\R^2), v=g \text{ in } \R^2\setminus \overline{\Omega}\}.
\end{equation*}
For any $\varphi \in C^1(\R)$ with $\varphi(0)=0=\varphi'(0)$, $\lip \varphi \leq 1$, we define 
\begin{equation*}
    \Omega_\varphi = \{(x_1,x_2) \in B_1: x_2 > \varphi(x_1)\}.
\end{equation*}
For the above $u$ we can obtain an energy upper bound.
\begin{prop}[Energy upper bound]
    \label{prop: energy upper bound in SBV}
    Let $u \in \argmin\{F(v,\overline{\Omega}): v\in SBV(\R^2), v=g_e \text{ in } \R^2\setminus \overline{\Omega}\}$ and $x\in \R^2$. Then 
    \begin{equation*}
        F(u, \overline{B}_\rho (x)) \leq C\pi \rho^2 + 2\pi\rho. 
    \end{equation*}
    for some constant $C$ that depends on $\Omega$ and $g$. 
\end{prop}
\begin{proof}
   It follows from a simple comparison argument.
\end{proof}
\subsubsection{Decay lemma}
For the reader's convenience, we report here Lemma 3.9 in \cite{carriero-leaci}.
\begin{lmm}[Decay]
    \label{lmm: Decay Lemma}
    For every $0 < \alpha < 1$, $0 < \beta < 1$ and $L>0$, there exist 
    $\eps = \eps(\alpha,\beta,L)$ and $\theta= \theta(\alpha, \beta, L)$ such that:
    for every $\varphi \in C^1(\R)$ with $\varphi(0)=0=\varphi'(0)$ and $\lip \varphi \leq 1$ and for every $w \in C^1(B_2)$ with $\lip w < L$, 
    if $u \in SBV(\R^2)$, $u=w$ in $B_1 \setminus \overline{\Omega}_\varphi$ and 
    \begin{equation*}
        F(u,\overline{\Omega}_\varphi) = \inf\{F(v,\overline{\Omega}_\varphi): v\in SBV(\R^2), \, v=u \text{ in } \R^2 \setminus \overline{\Omega}_\varphi \},
    \end{equation*}
    and if 
    \begin{equation*}
        \HH^1(S_u \cap \overline{B}_1) \leq \eps, 
    \end{equation*}
    then 
    \begin{equation*}
        F(u,\overline{B}_\alpha)\leq \alpha^{2-\beta} \max\{F(u, \overline{B}_1), \, \theta [(\lip \varphi)^2 + (\lip w)^2 ] \}.
    \end{equation*}
\end{lmm}

\subsubsection{Energy lower bound in SBV}
\begin{prop}[Energy lower bound at the boundary in $SBV$]
    \label{prop: density lower bound at the boundary, SBV}
    There exist $\eps, r >0$ depending on $\Omega$ and $g$, such that:
    if 
    \begin{equation*}
        u \in \argmin\{F(v,\overline{\Omega}): v\in SBV(\R^2), v=g \text{ in } \R^2\setminus \overline{\Omega}\}, 
    \end{equation*}
    then 
    \begin{equation*}
      F(u,B_\rho(x)) \geq \eps \rho \quad \forall x \in S_u\cap \partial \Omega, \, \forall \rho \in (0,r).
    \end{equation*}
\end{prop}
\begin{proof}
    The result follows from a slight modification of \cite[Proof of Lemma 3.10]{carriero-leaci}. We write it for the convenience of the reader.

    Fix $\alpha \in (0,1)$, $\beta = \frac{1}{2}$. Let $r>0$ and $C>0$ be given by Lemma \ref{lmm: radius for boundary}. Let $L = \norm{\nabla g_e}_\infty$.
    Let $\eps = \eps(\alpha,\beta,L), \theta=\theta(\alpha,\beta,L)$ be given by Lemma \ref{lmm: Decay Lemma}.

    For any $x \in \partial \Omega$, let $h_x$ be the function given by Lemma \ref{lmm: radius for boundary}. Then, for any $\rho \in (0,r)$, we define 
    \begin{equation*}
        \varphi_{x,\rho} (t) = \rho^{-1} h_x(\rho t) \quad \forall t \in \R, \quad
        \hat{g}_{x,\rho}(y) = \rho^{-\frac{1}{2}} g(x+\rho y) \quad \forall y \in \R^2.
    \end{equation*}
    Then it is easy to see that $\lip_{B_1} \varphi_{x,\rho} \leq C \rho$ and $\lip_{B_1} \hat{g}_{x,\rho} \leq \rho^\frac{1}{2} L$.
    Observe that, for every $\rho \leq \min\{r,1\}$, $\theta [(\lip_{B_1} \varphi_{x,\rho})^2 + (\lip_{B_1} \hat{g}_{x,\rho})^2 ] \leq C\theta \rho$
    for some constant $C>0$ that we do not rename.

    Define $r:= \min\{r,\frac{\eps}{C \theta},1\}$.
    Assume, by contradiction, that there exist $x \in S_u \cap \partial \Omega$ and $\rho \in (0,r)$ such that  
    \begin{equation*}
        F(u,\overline{B}_\rho(x)) \leq \eps \rho.
    \end{equation*}

    Translate and rotate the plane in such a way that $x=0$, $T_x \partial \Omega = \{x_2=0\}$ and $\Omega \cap B_r(x) = \{x_2 > h_x(x_1)\} \cap B_r(x)$.
    Then  $\Omega_{x,\rho} \cap B_1 = \{x_2 > \varphi_{x,\rho} (x_1)\} \cap B_1$, $u_{x,\rho} \in SBV(\R^2)$, $\HH^1(S_{u_{x,\rho}} \cap \overline{B}_1) \leq \eps$, 
    and $u_{x,\rho}(y) = \rho^{-1/2}u(x+\rho y) = \hat{g}_{x,\rho} (y)$ for any $y \in B_1 \setminus \overline{\Omega}_{\varphi_{x,\rho}}$.
    Therefore, by Lemma \ref{lmm: Decay Lemma}, 
    \begin{equation*}
        F(u,\overline{B}_{\alpha \rho}(x)) \leq \alpha^\frac{3}{2} \max\{F(u,\overline{B}_\rho(x)), C \theta \rho^2 \} \leq \alpha^\frac{3}{2} \eps \rho = \alpha^\frac{1}{2} \eps (\alpha\rho) < \eps(\alpha \rho).
    \end{equation*}
    Then, by induction, we have $F(u,\overline{B}_{\alpha^h\rho(x)}) \leq \alpha^{\frac{3}{2}h} \eps \rho$ for every $h \in \N$
    and hence 
    \begin{equation*}
        \lim_{t \to 0} t^{-1} F(u, \overline{B}_t(x)) = 0.
    \end{equation*}
    Finally, by \cite[Theorem 3.6]{degiorgi-carriero-leaci}, we conclude that $x \notin S_u$, which is a contradiction. 
    This proves that $F(u,\overline{B}_\rho(x)) > \eps \rho$.
    By replacing $\eps$ with its half, we get $F(u,B_\rho(x)) \geq \eps \rho$
    for every $x \in S_u \cap \partial \Omega$ and every $\rho \in (0,r)$. 
\end{proof}
\begin{lmm}(Energy lower bound near the boundary 1)
    \label{lmm: energy lower bound 1}
   There exist $\eps,r, \Lambda>0$, depending on $\Omega$ and $g$, such that: if 
    \begin{equation*}
        u \in \argmin\{F(v,\overline{\Omega}): v\in SBV(\R^2), v=g \text{ in } \R^2\setminus \overline{\Omega}\}, 
    \end{equation*} 
    then 
    \begin{equation*}
        F(u, B_\rho(x)) \geq \eps \rho \quad \forall x \in S_u \cap \Omega, \, \forall r > \rho > \Lambda \dist(x,\partial\Omega).
    \end{equation*}
\end{lmm}
\begin{proof}
    Let $\alpha', \eps', r$ be defined as in the proof of Proposition \ref{prop: density lower bound at the boundary, SBV}.
    Let $\eta \in (0,1)$ be such that $\alpha'\leq (1-\eta)^5$ and $\Lambda' > 2$ be such that $\alpha'(1-\eta) \leq \alpha'\left(1-\frac{1}{\Lambda'}\right) - \frac{1}{\Lambda'}$.
    We define $\alpha:= \alpha'(1-\eta)$.
    Let $\epsilon$ be the constant of the density lower bound at the interior (see \cite[Theorem 2.1.3]{delellis-focardi}).
    Let $N \in \N$ be such that $\frac{\Lambda'\alpha^{\frac{1}{4}N} \eps'}{2\alpha} < \epsilon$
    and define $\Lambda := \frac{\Lambda'}{\alpha^N} \geq \Lambda'$.

    Suppose, by contradiction, that there exist $x \in S_u \cap \Omega$ and $r>\rho>\Lambda \dist(x,\partial\Omega)$ such that 
    \begin{equation*}
        F(u,\overline{B}_\rho(x)) \leq \frac{\eps'}{2}\rho.
    \end{equation*}
    Let $d:= \dist(x,\partial\Omega)$ and $y \in \partial\Omega$ such that $|y-x|=d$. Observe that $B_{\rho-d}(y) \subset B_\rho(x)$ and hence 
    \begin{equation*}
        F(u, \overline{B}_{\rho-d}(y)) \leq \frac{\eps'}{2}\rho \leq \eps'(\rho-d).
    \end{equation*}
    Arguing as in the proof of Proposition \ref{prop: density lower bound at the boundary, SBV}, we get, possibly by taking a smaller $r$, 
    \begin{equation*}
        F(u,\overline{B}_{\alpha'(\rho-d)}(y))\leq (\alpha')^\frac{3}{2}(F(u,\overline{B}_{\rho-d}(y))\vee \frac{\eps'}{2}(\rho-d)).
    \end{equation*}
    Now observe that $\alpha'(1-\eta)\rho \leq \alpha'(\rho-d)-d$
    and $B_{\alpha'(\rho-d)-d}(x) \subset B_{\alpha'(\rho-d)}(y)$, so that
    \begin{equation*}
        \begin{split}
            F(u,\overline{B}_{\alpha\rho}(x)) &\leq (\alpha')^\frac{3}{2} (F(u,\overline{B}_\rho(x)) \vee \frac{\eps'}{2}\rho) \\ 
            &\leq \alpha^\frac{5}{4}(F(u,\overline{B}_\rho(x)) \vee \frac{\eps'}{2}\rho) \\
            &\leq \alpha^\frac{1}{4} \frac{\eps'}{2}(\alpha\rho) \leq \frac{\eps'}{2}(\alpha \rho)
        \end{split}
    \end{equation*}
    It is easy to prove by induction that 
    \begin{equation*}
        F(u,\overline{B}_{\alpha^h \rho}(x)) \leq \alpha^{\frac{5}{4}h} \frac{\eps'}{2}\rho \quad \forall h \in \N: \alpha^h \rho \geq \Lambda' d.
    \end{equation*}
    By definition of $\Lambda$, we have $\alpha^N \rho \geq \Lambda'd$.
    Let $\overline{N} \geq N$ be such that $\alpha^{\overline{N}+1} \rho < \Lambda' d \leq \alpha^{\overline{N}} \rho$. 
    Then, by the above claim and by the choice of $\overline{N}$ and $N$, we get
    \begin{equation*}
        d^{-1}F(u,\overline{B}_d(x)) \leq d^{-1}\alpha^{\frac{5}{4}\overline{N}} \frac{\eps'}{2}\rho \leq \frac{\Lambda'}{\alpha^{\overline{N}+1} \rho}\alpha^{\frac{5}{4}\overline{N}} \frac{\eps'}{2}\rho = \frac{\Lambda'\alpha^{\frac{1}{4}N} \eps'}{2\alpha} < \epsilon,
    \end{equation*}
    which contradicts the density lower bound at the interior (\cite[Theorem 2.1.3]{delellis-focardi})
\end{proof}
\begin{lmm}[Energy lower bound near the boundary 2]
    \label{lmm: energy lower bound 2}
  Let $\epsilon$ be the constant of the density lower bound at the interior (see \cite[Theorem 2.1.3]{delellis-focardi}). Let $\Lambda$ be given by Lemma \ref{lmm: energy lower bound 1}. Then, if $u$ is as in Lemma \ref{lmm: energy lower bound 1}, we have
    \begin{equation*}
        \HH^1(S_u \cap B_\rho(x)) \geq \frac{\epsilon}{\Lambda}\rho \quad \forall \,  \dist(x,\partial\Omega) < \rho \leq \Lambda \dist(x,\partial\Omega).
    \end{equation*}
\end{lmm}
\begin{proof}
    Since $\frac{\rho}{\Lambda} \leq \dist(x,\partial\Omega)$, by \cite[Theorem 2.1.3]{delellis-focardi} we have 
    \begin{equation*}
        \HH^1(S_u \cap B_\rho(x)) \geq \HH^1(S_u \cap B_{\frac{\rho}{\Lambda}}(x)) \geq \epsilon \frac{\rho}{\Lambda}.
    \end{equation*}
\end{proof}
\begin{prop}[Energy lower bound]
    \label{prop: energy lower bound}
    There exist $\eps, r>0$ such that, if $u$ is as in Lemma \ref{lmm: energy lower bound 1}, 
    \begin{equation*}
        F(u,B_\rho(x)) \geq \eps \rho \quad \forall x \in S_u,\, \forall \rho \in (0,r).
    \end{equation*}
\end{prop}
\begin{proof}
    Let $\eps', r, \Lambda$ be given by Lemma \ref{lmm: energy lower bound 1}. Let $\epsilon$ be the constant of the density lower bound at the interior (see \cite[Theorem 2.1.3]{delellis-focardi}).
    Assume that $x \in S_u$, $\rho \in (0,r)$, and let $d=\dist(x,\partial\Omega)$.
    Define $\eps:= \min\{\eps',\frac{\epsilon}{\Lambda}\}$. 
    If $\rho \geq \Lambda d$, then $F(u, B_\rho(x)) \geq \eps' \rho \geq \eps \rho$
    by Lemma \ref{lmm: energy lower bound 1}.
    If $d < \rho < \Lambda d$, then $F(u, B_\rho(x)) \geq \frac{\epsilon}{\Lambda} \rho \geq \eps \rho$ by Lemma \ref{lmm: energy lower bound 2}.
    Finally, if $\rho \leq d$, then $F(u,B_\rho(x)) \geq \epsilon \rho \geq \eps \rho$
    by \cite[Theorem 2.1.3]{delellis-focardi}.
\end{proof}
\begin{lmm} 
    Let $\eps, r$ be the constant of Proposition \ref{prop: energy lower bound}.
    There exists $\eta \in (0,1)$ such that: if $u$ is as in Proposition \ref{prop: energy lower bound} and 
    \begin{equation*}
        \HH^1(S_u \cap B_\rho(x)) < \eta^2 \rho 
    \end{equation*}
    for some $x \in \overline{\Omega}$ and $\rho \in (0,r)$, then 
    \begin{equation*}
        F(u,B_{\eta\rho}(x)) < \eps \eta \rho.
    \end{equation*}
\end{lmm}
\begin{proof}
    Suppose by contradiction that the claim is false. Then, for every $k \in \N$ there exist 
    \begin{equation*}
        u_k \in \argmin\{F(v,\overline{\Omega}): v\in SBV(\R^2), v=g \text{ in } \R^2\setminus \overline{\Omega}\}
    \end{equation*}
    $x_k \in \overline{\Omega}$, $\rho_k \in (0,r)$ such that 
    \begin{equation*}
        \HH^1(S_{u_k} \cap B_{\rho_k}(x_k)) < \frac{1}{k^2} \rho_k, \quad 
        F(u_k, B_{\frac{1}{k}\rho_k}(x_k)) \geq \eps \frac{1}{k}\rho_k.
    \end{equation*}
    Define $\Omega_k = \Omega_{x_k, \frac{1}{k}\rho_k}$ and $v_k = (u_k - g(x_k))_{x_k, \frac{1}{k}\rho_k}$.
    Then 
    \begin{equation*}
        v_k \in \argmin\{F(v,\overline{\Omega}_k): v\in SBV(\R^2), v=(g-g(x_k))_k \text{ in } \R^2\setminus \overline{\Omega}_k\}.
    \end{equation*}

    By Blaschke's Theorem \cite[Theorem 6.1, p.320]{AFP}, for a subsequence that we do not rename, 
    $\overline{\Omega}_k \to H$ locally in the Hausdorff distance, where $H$ is a closed subset of $\R^2$. 

    In the case $H = \R^2$ the result is classic, so we do not treat it here (see for example \cite{delellis-focardi} and \cite{AFP}).

    Suppose that $H \neq \R^2$. Then there exists $\overline{R}>0$ such that, for a subsequence that we do not rename, $B_{\overline{R}} \cap \partial \Omega_k \neq \emptyset$ for every $k \in \N$.
    
    Let $R \geq \overline{R}$. Define $C(R) = C\pi(R+1)^2 + 2\pi (R+1)$, 
    where $C$ is the constant of the energy upper bound (Proposition \ref{prop: energy upper bound in SBV}), and 
    \begin{equation*}
        h(\rho) = \liminf_k \HH^0(S_{v_k} \cap \partial B_\rho) + \frac{1}{8C(R)} \liminf_k \int_{\partial B_\rho} |\nabla v_k|^2 \quad \forall \rho \in (R,R+1).
    \end{equation*}
    Then, by Fatou's Lemma, by the coarea formula \cite[Theorem 2.93, p. 101]{AFP}, and by the energy upper bound (Proposition \ref{prop: energy upper bound in SBV}), 
    \begin{equation*}
        \int_{R}^{R+1}h(\rho) \, \dd{\rho} 
        \leq \liminf_k \left(\HH^1(S_{v_k} \cap B_{R+1}) + \frac{1}{8C(R)} \int_{B_{R+1}} |\nabla v_k|^2\right) \leq \frac{1}{8}.
    \end{equation*}
    Then, by Chebyshev's inequality, there exists $\rho \in (R,R+1)$ such that, taking a subsequence for which the two below limits exist, 
    \begin{equation*}
        \lim_k \HH^0(S_{v_k} \cap \partial B_\rho) + \frac{1}{8C(R)} \lim_k \int_{\partial B_\rho} |\nabla v_k|^2 \leq \frac{1}{2}.
    \end{equation*}
    This implies  
    \begin{equation*}
        \lim_k \HH^0(S_{v_k} \cap \partial B_\rho) = 0, \quad  \lim_k \int_{\partial B_\rho} |\nabla v_k|^2 \leq C'(R).
    \end{equation*} 
    For every $k \in \N$ let $z_k \in \partial B_\rho \cap \partial \Omega_k$. 
    Then, by the above information, we have
    \begin{equation*}
        |v_k(z)-v_k(z_k)| \leq \int_{\partial B_\rho} |\nabla v_k| \leq (2\pi\rho)^\frac{1}{2} \left(\int_{\partial B_\rho} |\nabla v_k|^2\right)^\frac{1}{2} \leq C(R),
    \end{equation*}
    for every $z \in \partial B_\rho$, where $C(R)$ is a (new) constant depending on $R$.
    Moreover, it is easy to see that $\norm{v_k}_{L^\infty(\overline{B}_\rho \setminus \Omega_k)} \to 0$.
    Therefore $\norm{v_k}_{L^\infty(\partial (B_\rho \cap \Omega_k))}$ is uniformly bounded by a constant depending on $R$ and, by a truncation argument, so is $\norm{v_k}_{L^\infty(B_\rho \cap \Omega_k)}$. We conclude that there is a constant $C(R)$ depending on $R$ such that $\norm{v_k}_{L^\infty(B_R)} \leq C(R)$.
    Then, by the compactness theorem for BV functions (\cite[Theorem 5.5]{evans-gariepy}), for a subsequence that we do not rename, $v_k \to v \text{ in } L^1_\loc(\R^2)$.
    Now, for every $R>0$, by the semicontinuity theorem in SBV (\cite[Theorem 2.4]{degiorgi-carriero-leaci}), we have $v \in SBV(B_R)$, 
    \begin{equation*}
        \HH^1(S_v \cap B_R) = 0, \quad
        \int_{B_R} |\nabla v|^2 \leq 2\pi R.
    \end{equation*}
    In particular, we conclude that $\HH^1(S_v) = 0$ and $v \in W^{1,2}_\loc (\R^2)$.
    Moreover, the minimality of $v_k$ implies that 
    \begin{equation*}
        \int_{B_R \cap H \setminus S_{v_k}} \nabla v_k \cdot \nabla \varphi = 0 \quad \forall \varphi \in C^\infty_c(B_R \cap H).
    \end{equation*}
    Furthermore, since $\{\nabla v_k\}_k$ is uniformly bounded in $L^2(B_R)$ for every $R$, for a subsequence that we do not rename we have $\nabla v_k \rightharpoonup \nabla v$ in $L^2_\loc(\R^2)$, and hence 
    \begin{equation*}
        \int_{B_R \cap H} \nabla v \cdot \nabla \varphi = 0 \quad \forall \varphi \in C^\infty_c(B_R \cap H), \, \forall R>0,
    \end{equation*}
    which means that $v$ is harmonic in $H$. Finally, it is easy to see that $v=0$ on $\partial H$, which is a line. Then, by Schwarz's reflection \cite[Problem 2.4]{gilbarg-trudinger}, we obtain $\tilde{v}$ that is harmonic in $\R^2$ and 
    \begin{equation*}
        \int_{B_R} |\nabla v|^2 \leq 4 \pi R \quad \forall R>0.
    \end{equation*}
    This implies, by the mean value property for harmonic functions and by Hölder's inequality, that $\nabla \tilde{v}=0$ in $\R^2$. 

    Now we claim that the following holds. 
    \begin{claim}
        For a subsequence for which the limit exists, we have 
        \begin{equation*}
            \int_{B_1} |\nabla v|^2 = \lim_k \int_{B_1} |\nabla v_k|^2
        \end{equation*}
    \end{claim}

   This is a contradiction, since $F(u_k, B_{\frac{1}{k}\rho_k}(x_k)) > \eps \frac{1}{k}\rho_k$
    implies $F(v_k, B_1) > \eps$
    and, since $H^1(S_{v_k}\cap B_1) \to 0$, 
    \begin{equation*}
        \lim_k \int_{B_1} |\nabla v_k|^2 > \eps.
    \end{equation*}

    To conclude, we prove the above claim.
    \begin{proof}[Proof of claim]
        Without renaming it, take a subsequence for which the limit exists. 
        We already know that $\int_{B_1} |\nabla v|^2 \leq \lim_k \int_{B_1} |\nabla v_k|^2$.
        Suppose by contradiction that 
        \begin{equation*}
            \int_{B_1} |\nabla v|^2 < \lim_k \int_{B_1} |\nabla v_k|^2.
        \end{equation*}
        Without loss of generality, assume that $H = \{x_2 > 0\}$. 
        Let $h_k$ be such that $\partial \Omega_k \cap B_2 = \graph(h_k) \cap B_2$,
        with $h_k \to 0$ in $C^1(B_2)$.
        Define 
        \begin{equation*}
            \Phi_k(x_1,x_2) = (x_1,x_2-h_k(x_1)) \quad \forall (x_1,x_2) \in \R^2,
        \end{equation*}
        $\Omega_k' = \Phi_k(\Omega_k)$, and $v_k' = (v_k-g_k)\circ \Phi_k^{-1}$.
        Then $v_k ' = 0$ on $\partial\Omega_k' \cap B_2 \setminus S_{v_k'}= \{x_2=0\} \cap B_2 \setminus S_{v_k'}$. 
        Moreover, it is easy to see that 
        \begin{equation*}
            \int_{B_1} |\nabla v|^2 < \lim_k \int_{B_1} |\nabla v_k'|^2.
        \end{equation*}
        \begin{claim}
            There exists $\{w_k'\}_k \subset SBV(\R^2)$ such that 
            \begin{align*}
                &w_k' = v_k' \quad \text{ in } \R^2 \setminus B_2, \\
                &w_k' = 0  \quad \text{ in } B_2 \setminus \overline{\Omega}_k', \\
                &\lim_k \left( \int_{B_2} |\nabla w_k'|^2 -\int_{B_2} |\nabla v_k'|^2\right) <0
            \end{align*}
        \end{claim}
        \begin{proof}[Proof of claim]
            Let $\varphi \in C^\infty_c(B_2)$ be such that $0\leq \varphi \leq 1$, and $\varphi \equiv 1$ in $B_1$. Define, for every $k \in \N$, $w_k' = \varphi v + (1-\varphi) v_k'$.
            Recall that $v=0$ on $\partial H = \{x_2 = 0\}$, and hence the first two properties of the claim are satisfied by construction. Now observe that 
            \begin{equation*}
                \norm{\nabla w_k'}_{L^2(A)} \\\leq \norm{(v-v_k')\nabla\varphi}_{L^2(A)} + \norm{ \varphi \nabla v + (1-\varphi) \nabla v_k'}_{L^2(A)},
            \end{equation*}
            where $A = (B_2\setminus B_1)\cap \supp(\varphi)$.
            Recall that $\norm{v_k}_{L^\infty(B_2)} \leq C$ for some constant $C$, for every $k$.
            Moreover, $v_k \to v$ in $L^1(B_2)$ and, taking a subsequence that we do not rename, $v_k \to v$ a.e. in $B_2$. Therefore, $\norm{v}_{L^\infty(B_2)} \leq C$.
            Furthermore, it is easy to see that also $\norm{v_k'}_{L^\infty(B_2)} \leq C$ for a (possibly different) constant $C$, and also $v_k' \to v$ a.e. in $B_2$.
            By the dominated convergence theorem, $\lim_k \norm{(v-v_k')\nabla\varphi}_{L^2(A)} = 0$.
            Moreover, since $\nabla v_k \rightharpoonup \nabla v$, also $\nabla v_k' \rightharpoonup \nabla v$ in $L^2(B_2)$. This implies that 
            \begin{equation*}
                \lim_k \left( \int_{B_2} |\nabla w_k'|^2 -\int_{B_2} |\nabla v_k'|^2\right) \leq \lim_k \left( \int_{B_1} |\nabla v|^2 -\int_{B_1} |\nabla v_k'|^2\right) < 0.
            \end{equation*}
        \end{proof}
        At this point we define $w_k = w_k' \circ \Phi_k + g_k$ for every $k \in \N$.
        By construction we have $w_k = v_k$ in $\R^2 \setminus \Phi_k^{-1}(B_2)$, $w_k = g_k$ in $\Phi_k^{-1}(B_2) \setminus \Omega_k$, and $S_{w_k} = S_{v_k}$.
        Then $w_k$ is an admissible competitor for $v_k$ in $\Phi_k^{-1}(B_2)$. However, it is easy to see that 
        \begin{equation*}
            \int_{\Phi_k^{-1}(B_2)\cap \Omega_k} |\nabla w_k|^2 <  \int_{\Phi_k^{-1}(B_2)\cap \Omega_k}  |\nabla v_k|^2
        \end{equation*}
        for sufficiently large $k \in \N$. This contradicts the minimality of $v_k$ and concludes the proof. 
    \end{proof}
\end{proof}
\subsubsection{Density lower bound in SBV}
\begin{thm}[Density Lower Bound in SBV]
    \label{thm: density lower bound in SBV}
    There exist $\eps,\rho > 0$ such that: if 
    \begin{equation*}
        u \in \argmin\{F(v,\overline{\Omega}): v\in SBV(\R^2), v=g \text{ in } \R^2\setminus \overline{\Omega}\}, 
    \end{equation*} 
    Then 
    \begin{equation*}
        \HH^1(S_u \cap B_\rho(x)) \geq \eps \rho \quad \forall x \in \overline{S}_u, \, \forall \rho \in (0,r).
    \end{equation*}
\end{thm}
\begin{proof}
    Let $\rho \in (0,r)$ and $x \in S_u$.
    Suppose by contradiction that $\HH^1(S_u \cap B_\rho(x)) < \eta^2 \rho$.
    Then, by the previous Lemma, $F(u,B_{\eta \rho}(x)) < \eps \eta \rho$.
    Therefore, by Proposition \ref{prop: energy lower bound}, $x \not\in S_u$, which is a contradiction. This proves that 
    \begin{equation*}
        \HH^1(S_u \cap B_\rho(x)) \geq \eps \rho \quad \forall x \in S_u, \, \forall \rho \in (0,r).
    \end{equation*}
    Now let $x \in \overline{S}_u$. Then there is $\{x_n\}_n \subset S_u$ such that $x_n \to x$.  
    Then, for sufficiently large $n$, $B_{\frac{\rho}{2}}(x_n) \subset B_\rho(x)$, and hence 
    \begin{equation*}
        \HH^1(S_u \cap B_\rho(x)) \geq \HH^1(S_u \cap B_{\frac{\rho}{2}}(x_n)) \geq \frac{\eps}{2}\rho.
    \end{equation*}
    Therefore, replacing $\eps$ with its half, the result holds for any $x \in \overline{S}_u$.
\end{proof}
\subsubsection{Proof of Theorem \ref{thm: density lower bound for K}}
The following Lemma is a slight modification of \cite[Lemma 5.2]{degiorgi-carriero-leaci}.
\begin{lmm}
    \label{lmm: argmin in SBV produces (K,u) with same energy}
Suppose that 
    \begin{equation*}
        u \in \argmin\{F(v,\overline{\Omega}): v\in SBV(\R^2), v=g \text{ in } \R^2\setminus \overline{\Omega}\}.
    \end{equation*}
For every $x \in \Omega \setminus \overline{S}_u$, set $\tilde{u}(x) = \aplim\limits_{y \to x} u(y)$. 
Then 
\begin{equation*}
    \tilde{u} \in C^\infty(\Omega \setminus \overline{S}_u) \quad \text{and} \quad \HH^1(\overline{S}_u \setminus S_u) = 0.
\end{equation*}
Moreover, $(\overline{S}_u,\tilde{u}) \in \mathcal{A}(\Omega,g)$ and $E(\overline{S}_u, \tilde{u}) = F(u, \overline{\Omega})$.
\end{lmm}
\begin{proof}
    Theorem \ref{thm: density lower bound in SBV}, combined with \cite[Theorem 2.6]{evans-gariepy}, implies $\HH^1(\overline{S}_u \setminus S_u) = 0$.
    Now observe that if $\overline{B}_\rho(x) \subset \Omega \setminus \overline{S}_u$, then $u \in W^{1,2}(B_\rho(x))$ and it is a minimizer of the functional $v \mapsto \int_{B_\rho(x)}  |\nabla v|^2$ among the functions $v \in u + W^{1,2}_0(B_\rho(x))$, and hence $v \in C^\infty(B_\rho(x))$.
    Moreover, since $u = g$ in $\R^2 \setminus \overline{\Omega}$, it follows that $\tilde{u} = g$ on $\partial \Omega$. The energy equality is easily verifiable.
\end{proof}
We are finally able to prove Theorem \ref{thm: density lower bound for K}.
\begin{proof}[Proof of Theorem \ref{thm: density lower bound for K}]
    By Remark \ref{Remark: regularity of u}, we have $u \in C(\overline{\Omega}\setminus K)$. Moreover, by Proposition \ref{prop: maximum principle}, we have $u \in L^\infty (\Omega \setminus K)$.
    We extend $u$ to $\R^2$ by setting 
    $u=g$ in $\R^2 \setminus \overline{\Omega}$.
    Recall that $g \in C^1_c(\R^2)$. Therefore $u \in W^{1,2}(\R^2 \setminus K)$ and $\int_{\R^2 \setminus K} |\nabla u| < +\infty$.
    Then, by \cite[Lemma 2.3]{degiorgi-carriero-leaci}, $u \in SBV(\R^2)$ and $S_u \subset K$.
    Moreover, it is easy to see that $u \in \argmin\{F(v,\overline{\Omega}): v\in SBV(\R^2), v=g \text{ in } \R^2\setminus \overline{\Omega}\}$. Indeed, if this is not case, Lemma \ref{lmm: argmin in SBV produces (K,u) with same energy} would give us an admissible pair that contradicts the minimality of $(K,u)$. Therefore, by Theorem \ref{thm: density lower bound in SBV}, we have 
    \begin{equation*}
        \HH^1(S_u \cap \overline{B}_\rho(x)) > \eps \rho \quad \forall x \in \overline{S}_u, \, \forall \rho \in (0,r).
    \end{equation*}
    Now let $x \in K$. If $x \in \overline{S}_u$, we have 
    \begin{equation*}
        \HH^1(K\cap \overline{B}_\rho(x)) \geq  \HH^1(S_u \cap \overline{B}_\rho(x)) > \eps \rho \quad \forall \rho \in (0,r).
    \end{equation*}
    Arguing by contradiction with Lemma \ref{lmm: argmin in SBV produces (K,u) with same energy}, it is easy to see that $\HH^1(K \setminus \overline{S}_u) =0$. This easily implies that for every $x \in K\setminus \overline{S}_u$, for sufficiently small $\rho>0$, we have $\HH^1(K\cap B_\rho(x)) = 0$. In particular, since $(K,u)$ is a normalized pair, $K \setminus \overline{S}_u = \emptyset$.
    This concludes the proof. 
\end{proof}
\section{Blow-up convergence: Proof of Proposition \ref{prop: global convergence of functions, blow-up}}
\label{section: proof of blow-up convergence}

In this section we prove Proposition \ref{prop: global convergence of functions, blow-up}.
Observe that, by the mean value theorem, we have the following estimate
\begin{equation}
    \label{eq: bound on rescaled boundary datum}
    \begin{split}
        |g_n(y)| \leq C |y| r_n^\frac{1}{2}
    \end{split}
\end{equation}
for every $y \in \R^2$, where $C=\norm{\nabla g}_\infty$.

First, we obtain convergence at the interior. 
\begin{prop}
    \label{prop: blow-up, convergence of functions}
    Let $\Omega^k$ be a connected component of $\Omega_0 \setminus K_0$. Then 
    \begin{enumerate}[label=(\roman*)]
        \item \label{item: interior convergence, boundary component} If $\Omega^k \subset \Omega*$, then, up to subsequences, 
        \begin{equation*}
            u_n \to u^k \quad \text{in } C^\infty_\loc (\Omega^k),
        \end{equation*}
        for some harmonic function $u^k: \Omega^k \to \R$. 
        \item \label{item: interior convergence, interior component} Otherwise, let $z^k \in \Omega^k$ and define $v_n = u_n - u_n(z^k)$. Then, up to subsequences, 
        \begin{equation*}
            v_n \to u^k \quad \text{in } C^\infty_\loc (\Omega^k),
        \end{equation*}
        for some harmonic function $u^k: \Omega^k \to \R$. 
    \end{enumerate}
\end{prop}
\begin{proof}
    Let $U$ be a compact subset of $\Omega^k$. Then, by the energy upper bound (Proposition \ref{prop: Energy upper bound at the boundary}), there exists a constant $C(U)$ such that   $\int_U |\nabla u_n|^2 \leq C(U)$
    for sufficiently large $n \in \N$. Moreover, since $u_n$ is harmonic on $\Omega_n \setminus K_n$, we get, for some constant $C(U)$ that we do not rename, $\norm{\nabla u_n}_{L^\infty (U)} \leq C(U)$ for sufficiently large $n\in \N$.

    \begin{enumerate}[leftmargin=*]
        \item Let us prove \ref{item: interior convergence, boundary component}. Without renaming $C(U)$, we have $\norm{u_n}_{L^\infty(U)} \leq C(U)$
        \fsln. To see this, cover $U$ with a finite collection of balls of radius $R=\frac{1}{4} \dist(U,\partial \Omega^k)$, with centers $x_1, \ldots, x_m \in U$. Then, by the mean value theorem and the above estimates, it suffices to show that $u_n(x_i)$ is uniformly bounded for $i=1, \ldots, m$. Take a point $z \in \overline{\Omega^k} \cap \partial\Omega_0 \setminus K_0$ and $\gamma:[a,b]\to \overline{\Omega^k}$ smooth curve such that $\gamma(a)=x_i, \gamma(b)=z.$ 
        Change coordinates so that $\Omega_0 = \{x_1 > 0\}$. 
        Then there is $\sigma>0$ such that the square $S_+ = [0,2\sigma]\times[z_2-\sigma, z_2+\sigma]$ is contained in $\overline{\Omega^k} \setminus K_0$.
        Let $w \in \mathring{S}_+ \cap \gamma((a,b)) $. Then there is $\rho >0$ such that $B_\rho(w) \subset S_+$.
        By the usual Hausdorff convergence, \fsln, $B_\rho(w) \subset \Omega_n \setminus K_n$. 
        Hence, since $u_n$ is harmonic in $\Omega_n \setminus K_n$, by the mean value property for harmonic functions and by Hölder's inequality, 
        \begin{equation*}
            |u_n(w)| \leq \frac{1}{\sqrt{\pi}\rho} \left(\int_{S\cap \Omega_n} |u_n|^2\right)^\frac{1}{2},
        \end{equation*}
        where $S = [-2\sigma, 2\sigma] \times [z_2 - \sigma, z_2 + \sigma]$.
        Moreover, since $w \in \gamma((a,b))$, there is $c\in (a,b)$ such that $\gamma(c)= w$.
        We define the smooth curve $\beta = \gamma \restriction_{[a,c]}$. Then, by the fundamental theorem of calculus and the previous estimates, $|u_n(x_i)| \leq |u_n(w)| + C(U) \HH^1(\beta)$. 
        
        To conclude the proof, it remains to show that  $\int_{S\cap \Omega_n} |u_n|^2$
        is uniformly bounded. For sufficiently large $n\in \N$, there is a function $f_n:[z_2-\sigma, z_2+\sigma] \to \R$ such that $S\cap \Omega_n = \{(y_1,y_2)\in \R^2: f_n(y_2) \leq y_1 \leq 2\sigma, z_2 - \sigma \leq y_2 \leq z_2 + \sigma\}$.
        Since $\partial \Omega_n \to \partial\Omega_0$ locally in the Hausdorff distance, we have $ -\frac{\sigma}{5} \leq f_n(y_2) \leq \frac{\sigma}{5}$
        \fsln, and hence, by the fundamental theorem of calculus, 
        \begin{equation*}
                \int_{S \cap \Omega_n}|u_n|^2 \leq \frac{11}{5}\sigma \int_{z_2-\sigma}^{z_2+\sigma} g_n^2(f_n(y_2), y_2) \, \D y_2 
                +\frac{22}{5}\sigma \int_{S\cap \Omega_n} |u_n| |\nabla u_n|. 
        \end{equation*}
        Using Cauchy's inequality with $\eps = \frac{5}{44 \sigma}$, \cite[Appendix B.2, p. 706]{evans}, we get
        \begin{equation*}
            \begin{split}
            \frac{1}{2} \int_{S \cap \Omega_n}|u_n|^2 &\leq \frac{11}{5}\sigma \int_{z_2-\sigma}^{z_2+\sigma} g_n^2(f_n(y_2), y_2) \, \D y_2 + \frac{242}{25}\sigma^2 \int_{S\cap \Omega_n}|\nabla u_n|^2.
            \end{split}
        \end{equation*}
        We conclude using the fact that $|g_n(f_n(y_2),y_2)|\leq C r_n^{1/2}$ and the energy upper bound (Proposition \ref{prop: Energy upper bound at the boundary}).

        Therefore $u_n$ is uniformly bounded on every compact subset of $\Omega^k$. Since $u_n$ is a sequence of harmonic functions, we deduce by \cite[Theorem 2.6, p. 35]{HFT} that, up to subsequences, $u_n$ converges uniformly to a continuous function $u: \Omega^k \to \R$ on every compact subset of $\Omega^k$. 
        Then, by \cite[Theorem 1.23, p. 16]{HFT}, $u$ is harmonic in $\Omega^k$ and, for every multi-index $\alpha$, $D^\alpha u_n \to D^\alpha u$ uniformly on compact subsets of $\Omega^k$. 
        \item Let us prove \ref{item: interior convergence, interior component}. Since $v_n = u_n - u_n(z_k)$, then $\{v_n\}_{n \in \N}$ is a sequence of harmonic functions, and $\nabla v_n = \nabla u_n$. 
        Therefore, we can repeat the steps in the proof of (i), replacing $w$ with $z^k$. 
        Observe that $v_n(z^k)$ is uniformly bounded, since $v_n(z^k) = 0$ for every $n \in \N$, and so we are done. 
    \end{enumerate}
\end{proof}

Now we prove convergence up to the boundary.
\begin{prop}[Convergence up to the boundary]
    \label{prop: convergence of functions up to the boundary}
    Let $\Omega^k$ be a connected component of $\Omega_0 \setminus K_0$ such that $\Omega^k \subset \Omega^*$.
    Then there exists 
    \begin{equation*}
        u^k \in C^\infty(\Omega^k) \cap C^1(\myset)
    \end{equation*}
    that is harmonic in $\Omega^k$ and satisfies the following property. 
    If $U$ is a compact subset of $\myset$ and $u_n$ is replaced, for sufficiently large $n$, with an appropriate (see the proof) compactly supported $C^{1,\gamma}$ extension to $\Omega_n \cup U$ (notice that, a priori, $u_n$ is defined only on $U \cap \Omega_n$), then, up to subsequences, 
    \begin{equation*}
        u_n \to u^k \text{ in } C^1(U)
    \end{equation*} 
    as $n \to \infty$. When this happens, with an abuse of notation we will write 
    \begin{equation}
        \label{eq: C^1_loc convergence abuse of notation}
        u_n \to u^k \in C^1_\loc(\myset). 
    \end{equation}
\end{prop}
\begin{proof}
    Change coordinates so that $\Omega_0 = \{x_1 > 0\}$.
     \begin{claim}
         \label{claim: bound on C^2 norm for a ball centered on T_x partial Omega}
         Let $w \in \partial\Omega_0 \setminus K_0$ and $r < \frac{1}{3} \dist(w, K_0)$. Then there exists a constant $C>0$ such that 
         \begin{equation*}
             |u_n|_{1, \gamma ;B_r(w) \cap \Omega_n} \leq C \quad \forall n \in \N.
         \end{equation*}
     \end{claim}
     \begin{proof}[Proof of Claim]
         Since $\partial \Omega_n \to \partial\Omega_0$ locally in the Hausdorff distance, 
         we can assume that $\partial \Omega_n \cap B_r(w) \subset \{|x_1| \leq \frac{r}{3}\}$.
         Now we define $U_n = B_{2r}(w) \cap \Omega_n$, so that $\partial U_n = (\partial \Omega_n \cap B_{2r}(w)) \cup (\partial B_{2r}(w) \cap \Omega_n)$. Then $T_n:=\partial \Omega_n \cap B_{2r}(w)$ is a $C^{1,\gamma}$ portion of $\partial U_n$. Moreover, $U_n \cap K_0 = \emptyset$ for sufficiently large $n$, and hence $\Delta u_n = 0$ in $U_n$. Finally, $u_n = g_n $ on $T_n$ and $g_n \in C^{1,\gamma} (\overline{U_n})$. 
         Now, if $w = (0,w^2)$, we define $w_n = \partial\Omega_n \cap \{x_2 = w^2\}$. 
         Observe that $w_n \in T_n$ and $\frac{3}{2}r < \dist(w_n, \partial U_n \setminus T_n)$. 
         Therefore, by \cite[Corollary 6.7, p. 100]{gilbarg-trudinger}, we have 
         \begin{equation*}
             |u_n|_{1, \gamma; B_n \cap U_n} \leq C(|u_n|_{0; U_n} + |g_n|_{1, \gamma; U_n})
         \end{equation*}
         where $B_n = B(w_n, \frac{3}{2}r)$ and $C$ does not depend on $n$.
         By a consequence of the coarea formula \cite[Theorem 3.12, p. 140]{evans-gariepy}, by Proposition \ref{prop: Energy upper bound at the boundary}, and by Chebyshev's inequality \cite[6.17, p. 193]{folland}, for every $n$ there exists $t_n \in (2r, \frac{5}{2}r)$ such that 
         \begin{equation*}
             \int_{\partial B_{t_n}(w) \cap \Omega_n} |\nabla u_n|^2 \dd{\HH^1} < 4C.
         \end{equation*}
         We define $\gamma_n(\theta) = (t_n \cos\theta, w^2 + t_n \sin\theta)$ for every $\theta \in [0, 2\pi)$. Then there exist $a_n,b_n \in [0,2\pi]$ such that $\partial B_{t_n}(w) \cap \Omega_n = \gamma_n((a_n,b_n))$.
         We define $h_n: (a_n,b_n) \to \R$ by $h_n(\theta) = u_n(\gamma_n(\theta))$. 
         In particular, by definition of $t_n$, 
         \begin{equation*}
            \int_{a_n}^{b_n} |h_n'(\theta)|^2 \dd{\theta} \leq t_n  \int_{\partial B_{t_n}(w) \cap \Omega_n} |\nabla u_n|^2 \dd{\HH^1} \leq C r
         \end{equation*}
         for some constant $C>0$. 
         Then, by Hölder's inequality, we have $[h_n]_{0,\frac{1}{2},(a_n,b_n)} \leq C r^{1/2}$
         for some constant $C>0$.
         Since $\gamma_n(a_n) \in \partial \Omega_n \cap B_R$ for some $R>0$, by \eqref{eq: bound on rescaled boundary datum} we have $|g_n(\gamma_n(a_n))| \leq C$. This, together with the above estimate, implies 
         \begin{equation*}
             |u_n|_{0; \partial B_{t_n}(w) \cap \Omega_n} \leq C
         \end{equation*} 
         for some constant $C>0$. 
         Now, by the maximum principle \cite[Theorem 2.3]{gilbarg-trudinger}, we have
         \begin{equation*}
             |u_n|_{0; B_{t_n}(w) \cap \Omega_n} \leq |u_n|_{0; \partial (B_{t_n}(w)\cap \Omega_n)}.
         \end{equation*}
         By \eqref{eq: bound on rescaled boundary datum} we have 
         \begin{equation*}
             |g_n|_{0; \partial \Omega_n \cap B_{t_n}(w)} \leq C
         \end{equation*}
         for some constant $C>0$.
         Combining the above estimates, we conclude that 
         \begin{equation*}
             |u_n|_{0; B_{t_n}(w) \cap \Omega_n} \leq C 
         \end{equation*}
         for some constant $C>0$, for every $n \in \N$.
         Since $t_n \geq 2r$, $U_n = B_{2r}(w) \cap \Omega_n \subset B_{t_n}(w) \cap \Omega_n$, 
         and hence $|u_n|_{0; U_n} \leq C$
         for some constant $C>0$. 
         Moreover, observe that $|g_n|_{1,\gamma;U_n} \leq C$
         for every $n$. We conclude that 
         \begin{equation*}
             |u_n|_{1, \gamma; B_n \cap U_n} \leq C
         \end{equation*}
         for some constant $C>0$. 
         Furthermore, by construction we have $B_n \subset B_{2r} (w)$ and hence $B_n \cap U_n = B_n \cap B_{2r} (w) \cap \Omega_n = B_n \cap \Omega_n$. 
         We conclude by observing that $B_r(w) \subset B_n$. 
     \end{proof}
     \begin{claim}
         \label{claim: bound on C^2 norm for a ball that intersects T_x partial Omega}
         Let $w \in \Omega_0 \setminus K_0$ and $r < \frac{1}{20} \dist(w, K_0)$. Suppose that $B_r(w) \cap \partial\Omega_0 \neq \emptyset$. Then there exists a constant $C>0$ such that 
         \begin{equation*}
             |u_n|_{1, \gamma; B_r(w) \cap \Omega_n} \leq C \quad \forall n \in \N.
         \end{equation*}
     \end{claim}
     \begin{proof}[Proof of Claim]
         A simple adaptation of the proof of Claim \ref{claim: bound on C^2 norm for a ball centered on T_x partial Omega}.
     \end{proof}
     Now assume that $U$ is a compact subset of $\Omega^k \cup (\overline{\Omega^k} \cap \partial\Omega_0 \setminus K_0 )$. Let $r < \frac{1}{20} \dist(U, K_0)$.
     We define $U' = U \cap \{x_1 \leq \frac{r}{2}\}$. Then, since $U'$ is compact, there exist $w_1, \ldots, w_m \in U'$ such that $U' \subset \cup_{i=1}^m B(w_i,r)$. 
     Moreover, by construction, $B(w_i, r) \cap \partial\Omega_0 \neq \emptyset$ for every $i \in \{1, \ldots, m\}$. Therefore, by Claim \ref{claim: bound on C^2 norm for a ball that intersects T_x partial Omega}, there is a constant $C>0$ such that  $|u_n|_{1, \gamma; U'\cap \Omega_n} \leq C$ for every  $n \in \N$.
     Now, by the proof of \cite[Lemma 6.37, p. 136]{gilbarg-trudinger}, for every $n$ there exists a $C^{1,\gamma}$ extension of $u_n$, which do not rename, to $\Omega_n \cup U'$, such that $|u_n|_{1,\gamma;U'} \leq C |u_n|_{1,\gamma; U'\cap\Omega_n}$
     for some constant $C$ that does not depend on $n$. 
     This implies that $\{u_n\}_n$ and $\{\partial_i u_n\}_n$, $i = 1, 2$, are uniformly bounded and equicontinuous in $U'$. 
     Therefore, by the Arzelà-Ascoli Theorem \cite[p. 208]{royden}, there exists $u' \in C^1(U')$ such that, up to subsequences, $u_n \to u' \text{ in } C^1(U')$.
     
     Now we define $U'' = U \cap \{x_1 \geq \frac{r}{4}\}$. Then $U''$ is a compact subset of $\Omega^k$, and hence, by Proposition \ref{prop: blow-up, convergence of functions}, up to subsequences $u_n \to u'' \text{ in } C^\infty(U'')$.
     This implies that $u' = u''$ in $U' \cap U''$. We define $u: U \to \R$ by 
     \begin{equation*}
         u(y) = \begin{cases}
             u'(y)  &\text{ if } y \in U', \\
             u''(y) &\text{ if } y \in U''.
         \end{cases}
     \end{equation*}
     Then $u$ is well defined and $u_n \to u$ in $C^1(U)$. Moreover, by Proposition \ref{prop: blow-up, convergence of functions}, we conclude that $u \in C^\infty(U \cap \Omega_0)$ and $\Delta u = 0$ in $U \cap \Omega_0$. 
     Now, for every $m \in \N$, we define 
     \begin{equation*}
         A_m = \{y \in \myset : \dist(y, K_0) \geq \frac{1}{m}\} \cap \overline{B}_m.
     \end{equation*}
     Then, for every $m \in \N$, $A_m$ is compact and $A_m \subset \mathrm{int}(A_{m+1})$. 
     Moreover, $\cup_{m=1}^\infty A_m = \myset$.
     Then, applying the above results to the compact sets $A_m$ and by a diagonal argument, we obtain a function $u^k \in C^1(\myset) \cap C^\infty(\Omega^k)$ that is harmonic in $\Omega^k$ and such that, up to subsequences, 
     \begin{equation*}
         u_n \to u^k \text{ in } C^1_\loc (\myset). 
     \end{equation*}
 \end{proof}
\begin{proof}[Proof of Proposition \ref{prop: global convergence of functions, blow-up}]
    The result follows by combining Propositions \ref{prop: blow-up, convergence of functions} and \ref{prop: convergence of functions up to the boundary}.
\end{proof}
\appendix
\section{Useful properties}
In this section we collect some elementary but useful results that are interesting or needed in the paper. 

The following lemma is used several times, and it is an elementary consequence of definition of $C^2$ boundary.
\begin{lmm}
    \label{lmm: radius for boundary}
   There exist $\overline{r}>0$ and $C>0$ with the following property. For every $x \in \partial \Omega$ there exists $h_x \in C^2(\R)$, with $h_x(0) = h_x'(0) = 0$ and $\norm{h_x''}_\infty \leq C$, such that, for every $r \in (0,\overline{r})$, $\partial \Omega \cap B_r(x)$ is the graph of $h_x$ over $x+T_x \partial \Omega$.

    More precisely, if we translate and rotate the coordinates so that $x=0$ and $T_x \partial \Omega = \{y_2 = 0\}$, we have $\Omega \cap B_r(x) = \{y_2 > h_x (y_1)\} \cap B_r(x)$.
\end{lmm}

The following proposition guarantees that minimizers do not entirely exclude the boundary datum.
\begin{prop}
    \label{prop: K is not equal to partial Omega}
    Assume \minimizer{K}{u}{\Omega}{g}. Then $K \neq \partial\Omega$.
\end{prop}
\begin{proof}
    Suppose by contradiction that $K = \partial \Omega$. Then $u$ must be constant in $\Omega$ and $E(K,u) = \HH^1(K)$.
    Let $x \in \partial \Omega$. For every $\rho > 0$, let $\zeta_\rho \in C^\infty_c(B_\rho(x))$ be such that $0\leq \zeta \leq 1$ in $B_\rho(x)$, $\zeta_\rho \equiv 1$ in $B_{\rho /2}(x)$ and 
    $|\nabla \zeta_\rho| \leq \frac{C}{\rho}$ in $B_\rho(x)$ for some constant $C$. 
    Observe that $E(K,g(x)) = E(K,u)$ since $u$ is constant, and hence \minimizer{K}{g(x)}{\Omega}{g}.
    Consider $v = \zeta_\rho g  + (1-\zeta_\rho) g(x)$ and $J = K \setminus B_{\rho/2} (x)$. Then $(J,v) \in \mathcal{A}(\Omega,g)$ since $v=g$ on $\partial\Omega \cap B_{\rho/2}(x)$.
    Observe that, if $y \in B_\rho(x)\cap \Omega$, $|\nabla v(y)|\leq (C+1)\norm{\nabla g}_{L^\infty(\Omega)}$
    and hence
    \begin{equation*}
        E (J,v,B_\rho(x)) \leq (C+1)^2 \norm{\nabla g}^2_{L^\infty(\Omega)}  \rho^2 + \HH^1(\partial\Omega \cap(B_\rho(x) \setminus B_{\rho/2}(x)) ).
    \end{equation*}
    Moreover, $E(K,g(x),B_\rho(x)) = \HH^1(\partial \Omega \cap B_\rho(x))$ and hence
    \begin{equation*}
        \HH^1(\partial \Omega \cap B_{\rho/2}(x)) \leq (C+1)^2 \norm{\nabla g}^2_{L^\infty(\Omega)}  \rho^2, 
    \end{equation*}
    which is a contradiction, since $\partial\Omega \in C^1$ implies $\HH^1(\partial \Omega \cap B_{\rho/2}(x)) \sim \rho$ as $\rho \to 0$.
\end{proof}

A simple truncation argument, using Proposition \ref{prop: K is not equal to partial Omega} and \cite[Lemma 4.3.1]{delellis-focardi}, yields the following maximum principle. 
\begin{prop}[Maximum Principle]
    \label{prop: maximum principle}
Assume \minimizer{K}{u}{\Omega}{g}.
Then $\norm{u}_{L^\infty(\Omega)} \leq \norm{g}_{L^\infty(\partial\Omega)}$. 
\end{prop}

Since they are needed throughout the paper, we write the Euler-Lagrange equations for our functional.
Outer variations yield 
\begin{equation}
    \label{eq: Euler-Lagrange integral form}
    \int_{\Omega\setminus K} \nabla u \cdot \nabla v = 0 \quad \forall v \in C^\infty_c (\Omega),
\end{equation}
which, under the right smoothness assumptions, corresponds to 
\begin{equation}
    \label{eq: Euler-Lagrange }
    \begin{cases}
        \Delta u = 0 \quad &\text{in } \Omega \setminus K, \\
        u = g \quad &\text{on } \partial \Omega \setminus K, \\
        \frac{\partial u}{\partial \nu} = 0 \quad &\text{on } K. 
    \end{cases}
\end{equation}
Finally, inner variations yield the following. 
\begin{lmm}[First derivative of the energy]
    \label{lmm: first derivative of the energy for generalized minimizers}
    Let \minimizer{K}{u}{\Omega}{g}, $z \in \R^2$, $\eta \in [C^1_c(B_R(z))]^2$ for some $R>0$.
    For $t \in \R$, define $\Phi_t(x) = x + t \, \eta (x)$ for every $x \in \R^2$ and $K_t = \Phi_t(K)$, $u_t = u \circ \Phi_t^{-1}$.
    Then 
    \begin{equation*}
            \derivative{}{t} E(K_t,u_t, B_R)|_{t=0} =  \int_{B_R(z) \cap \Omega \setminus K} (|\nabla u|^2 \diverg \eta - 2 \nabla u^T D\eta \, \nabla u) + \int_{K \cap B_R(z)} e^T D\eta \, e \, \dd{\HH^1}
    \end{equation*}
    as $t \to 0$, where $e(x) \in S^1$ is such that $T_x K  = \mathrm{span}\{e(x)\}$ for every $x \in K$ such that the approximate tangent line $T_x K$ exists.
\end{lmm}
\begin{rmk}[Regularity]
    \label{Remark: regularity of u}
    The integral form of the Euler-Lagrange equation \eqref{eq: Euler-Lagrange integral form}, together with classical elliptic regularity theory, implies that $u \in C^\infty(\Omega\setminus K) \cap C^{1,\gamma}(\overline{\Omega} \setminus K)$ for any \minimizer{K}{u}{\Omega}{g}. 
\end{rmk}

\begin{thebibliography}{DMMS92}

    \bibitem[ABR01]{HFT}
    Sheldon Axler, Paul Bourdon, and Wade Ramey.
    \newblock {\em Harmonic function theory}, volume 137 of {\em Graduate Texts in Mathematics}.
    \newblock Springer-Verlag, New York, second edition, 2001.
    
    \bibitem[AFP97]{ambrosio-fusco-pallara-free-discont-2}
    Luigi Ambrosio, Nicola Fusco, and Diego Pallara.
    \newblock Partial regularity of free discontinuity sets. {II}.
    \newblock {\em Ann. Scuola Norm. Sup. Pisa Cl. Sci. (4)}, 24(1):39--62, 1997.
    
    \bibitem[AFP99]{afp-higher-regularity}
    Luigi Ambrosio, Nicola Fusco, and Diego Pallara.
    \newblock Higher regularity of solutions of free discontinuity problems.
    \newblock {\em Differential Integral Equations}, 12(4):499--520, 1999.
    
    \bibitem[AFP00]{AFP}
    Luigi Ambrosio, Nicola Fusco, and Diego Pallara.
    \newblock {\em Functions of bounded variation and free discontinuity problems}.
    \newblock Oxford Mathematical Monographs. The Clarendon Press, Oxford University Press, New York, 2000.
    
    \bibitem[All72]{allard-interior}
    William~K. Allard.
    \newblock On the first variation of a varifold.
    \newblock {\em Ann. of Math. (2)}, 95:417--491, 1972.
    
    \bibitem[All75]{allard-boundary}
    William~K. Allard.
    \newblock On the first variation of a varifold: boundary behavior.
    \newblock {\em Ann. of Math. (2)}, 101:418--446, 1975.
    
    \bibitem[AM19]{andersson-mikayelyan}
    John Andersson and Hayk Mikayelyan.
    \newblock Regularity up to the crack-tip for the mumford-shah problem, 2019.
    
    \bibitem[AP97]{ambrosio-pallara-free-discont-1}
    Luigi Ambrosio and Diego Pallara.
    \newblock Partial regularity of free discontinuity sets. {I}.
    \newblock {\em Ann. Scuola Norm. Sup. Pisa Cl. Sci. (4)}, 24(1):1--38, 1997.
    
    \bibitem[BD01]{bonnet-david-cracktip}
    Alexis Bonnet and Guy David.
    \newblock Cracktip is a global {M}umford-{S}hah minimizer.
    \newblock {\em Ast\'{e}risque}, (274):vi+259, 2001.
    
    \bibitem[BFM08]{Bourdin-Francfort-Marigo-2008}
    Blaise Bourdin, Gilles~A. Francfort, and Jean-Jacques Marigo.
    \newblock The variational approach to fracture.
    \newblock {\em J. Elasticity}, 91(1-3):5--148, 2008.
    
    \bibitem[BL14]{bucur-luckhaus}
    Dorin Bucur and Stephan Luckhaus.
    \newblock Monotonicity formula and regularity for general free discontinuity problems.
    \newblock {\em Arch. Ration. Mech. Anal.}, 211(2):489--511, 2014.
    
    \bibitem[Bon96]{Bonnet}
    A.~Bonnet.
    \newblock On the regularity of edges in image segmentation.
    \newblock {\em Ann. Inst. H. Poincar\'{e} Anal. Non Lin\'{e}aire}, 13(4):485--528, 1996.
    
    \bibitem[CL90]{carriero-leaci}
    M.~Carriero and A.~Leaci.
    \newblock Existence theorem for a {D}irichlet problem with free discontinuity set.
    \newblock {\em Nonlinear Anal.}, 15(7):661--677, 1990.
    
    \bibitem[Dav96]{david-c1-arcs}
    Guy David.
    \newblock {$C^1$}-arcs for minimizers of the {M}umford-{S}hah functional.
    \newblock {\em SIAM J. Appl. Math.}, 56(3):783--888, 1996.
    
    \bibitem[Dav05]{david}
    Guy David.
    \newblock {\em Singular sets of minimizers for the {M}umford-{S}hah functional}, volume 233 of {\em Progress in Mathematics}.
    \newblock Birkh\"{a}user Verlag, Basel, 2005.
    
    \bibitem[DGCL89]{degiorgi-carriero-leaci}
    E.~De~Giorgi, M.~Carriero, and A.~Leaci.
    \newblock Existence theorem for a minimum problem with free discontinuity set.
    \newblock {\em Arch. Rational Mech. Anal.}, 108(3):195--218, 1989.
    
    \bibitem[DL02]{david-leger}
    Guy David and Jean-Christophe L\'{e}ger.
    \newblock Monotonicity and separation for the {M}umford-{S}hah problem.
    \newblock {\em Ann. Inst. H. Poincar\'{e} C Anal. Non Lin\'{e}aire}, 19(5):631--682, 2002.
    
    \bibitem[DL08]{delellis-tangent-measures}
    Camillo De~Lellis.
    \newblock {\em Rectifiable sets, densities and tangent measures}.
    \newblock Zurich Lectures in Advanced Mathematics. European Mathematical Society (EMS), Z\"{u}rich, 2008.
    
    \bibitem[DLF13]{delellis-focardi-higher}
    C.~De~Lellis and M.~Focardi.
    \newblock Higher integrability of the gradient for minimizers of the {$2d$} {M}umford-{S}hah energy.
    \newblock {\em J. Math. Pures Appl. (9)}, 100(3):391--409, 2013.
    
    \bibitem[DLF23]{delellis-focardi}
    Camillo De~Lellis and Matteo Focardi.
    \newblock {\em The regularity theory for the Mumford-Shah functional on the plane}.
    \newblock 2023.
    
    \bibitem[DLFG21]{delellis-focardi-ghinassi}
    Camillo De~Lellis, Matteo Focardi, and Silvia Ghinassi.
    \newblock Endpoint regularity for {$2d$} {M}umford-{S}hah minimizers: on a theorem of {A}ndersson and {M}ikayelyan.
    \newblock {\em J. Math. Pures Appl. (9)}, 155:83--110, 2021.
    
    \bibitem[DMMS92]{dal-maso-morel-solimini}
    G.~Dal~Maso, J.-M. Morel, and S.~Solimini.
    \newblock A variational method in image segmentation: existence and approximation results.
    \newblock {\em Acta Math.}, 168(1-2):89--151, 1992.
    
    \bibitem[DPF14]{dephilippis-figalli}
    Guido De~Philippis and Alessio Figalli.
    \newblock Higher integrability for minimizers of the {M}umford-{S}hah functional.
    \newblock {\em Arch. Ration. Mech. Anal.}, 213(2):491--502, 2014.
    
    \bibitem[EG15]{evans-gariepy}
    Lawrence~C. Evans and Ronald~F. Gariepy.
    \newblock {\em Measure theory and fine properties of functions}.
    \newblock Textbooks in Mathematics. CRC Press, Boca Raton, FL, revised edition, 2015.
    
    \bibitem[Eva10]{evans}
    Lawrence~C. Evans.
    \newblock {\em Partial differential equations}, volume~19 of {\em Graduate Studies in Mathematics}.
    \newblock American Mathematical Society, Providence, RI, second edition, 2010.
    
    \bibitem[FM98]{Francfort-Marigo-1998}
    G.~A. Francfort and J.-J. Marigo.
    \newblock Revisiting brittle fracture as an energy minimization problem.
    \newblock {\em J. Mech. Phys. Solids}, 46(8):1319--1342, 1998.
    
    \bibitem[Fol99]{folland}
    Gerald~B. Folland.
    \newblock {\em Real analysis}.
    \newblock Pure and Applied Mathematics (New York). John Wiley \& Sons, Inc., New York, second edition, 1999.
    \newblock Modern techniques and their applications, A Wiley-Interscience Publication.
    
    \bibitem[GT01]{gilbarg-trudinger}
    David Gilbarg and Neil~S. Trudinger.
    \newblock {\em Elliptic partial differential equations of second order}.
    \newblock Classics in Mathematics. Springer-Verlag, Berlin, 2001.
    \newblock Reprint of the 1998 edition.
    
    \bibitem[Leg99]{leger}
    J.~C. Leger.
    \newblock Flatness and finiteness in the {M}umford-{S}hah problem.
    \newblock {\em J. Math. Pures Appl. (9)}, 78(4):431--459, 1999.
    
    \bibitem[Lem11]{lemenant-r3}
    Antoine Lemenant.
    \newblock Regularity of the singular set for {M}umford-{S}hah minimizers in {$\mathbb{R^3}$} near a minimal cone.
    \newblock {\em Ann. Sc. Norm. Super. Pisa Cl. Sci. (5)}, 10(3):561--609, 2011.
    
    \bibitem[Lem15]{lemenant-rigidity}
    Antoine Lemenant.
    \newblock A rigidity result for global {M}umford-{S}hah minimizers in dimension three.
    \newblock {\em J. Math. Pures Appl. (9)}, 103(4):1003--1023, 2015.
    
    \bibitem[LF13]{delellis-focardi-densitylowerbound}
    Camillo~De Lellis and Matteo Focardi.
    \newblock Density lower bound estimates for local minimizers of the 2d {M}umford-{S}hah energy.
    \newblock {\em Manuscripta Math.}, 142(1-2):215--232, 2013.
    
    \bibitem[LM18]{lemenant-mikayelyan}
    Antoine Lemenant and Hayk Mikayelyan.
    \newblock Stationarity of the crack-front for the {M}umford-{S}hah problem in 3{D}.
    \newblock {\em J. Math. Anal. Appl.}, 462(2):1555--1569, 2018.
    
    \bibitem[Mag12]{maggi}
    Francesco Maggi.
    \newblock {\em Sets of finite perimeter and geometric variational problems}, volume 135 of {\em Cambridge Studies in Advanced Mathematics}.
    \newblock Cambridge University Press, Cambridge, 2012.
    \newblock An introduction to geometric measure theory.
    
    \bibitem[MS89]{Mumford-Shah}
    David Mumford and Jayant Shah.
    \newblock Optimal approximations by piecewise smooth functions and associated variational problems.
    \newblock {\em Comm. Pure Appl. Math.}, 42(5):577--685, 1989.
    
    \bibitem[MS01]{maddalena-solimini}
    Francesco Maddalena and Sergio Solimini.
    \newblock Blow-up techniques and regularity near the boundary for free discontinuity problems.
    \newblock {\em Adv. Nonlinear Stud.}, 1(2):1--41, 2001.
    
    \bibitem[Pri03]{priestley}
    H.~A. Priestley.
    \newblock {\em Introduction to complex analysis}.
    \newblock Oxford University Press, Oxford, second edition, 2003.
    
    \bibitem[RF10]{royden}
    HL~Royden and PM~Fitzpatrick.
    \newblock {\em Real analysis 4th Edition}.
    \newblock 2010.
    
\end{thebibliography}
    
\end{document}